\documentclass[aop]{imsart}

\RequirePackage{amsthm,amsmath,amsfonts,amssymb}
\RequirePackage[numbers]{natbib}
\usepackage{mathtools}

\startlocaldefs

\numberwithin{equation}{section}
\theoremstyle{plain}
\newtheorem{theorem}{Theorem}[section]
\newtheorem{proposition}{Proposition}[section]
\newtheorem{corollary}{Corollary}[section]
\newtheorem{lemma}{Lemma}[section]
\theoremstyle{remark}

\newtheorem{remark}{Remark}[section]

\newcommand{\noi}{\noindent}

\newcommand{\eps}{{\varepsilon}}  
\newcommand{\veps}{{\varepsilon}}

\newcommand{\Mmb}{{\mathbb{M}}}
\newcommand{\NN}{{\mathbb{N}}}

\newcommand{\RR}{{\mathbb{R}}}

\newcommand{\clb}{{\mathcal{B}}}
\newcommand{\clc}{{\mathcal{C}}}
\newcommand{\clo}{{\mathcal{O}}}
\newcommand{\clz}{{\mathcal{Z}}}
\newcommand{\cla}{{\mathcal{A}}}
\newcommand{\clw}{{\mathcal{W}}}
\newcommand{\cln}{{\mathcal{N}}}

\newcommand{\cld}{{\mathcal{D}}}

\newcommand{\clr}{{\mathcal{R}}}

\newcommand{\clf}{{\mathcal{F}}}

\newcommand{\clm}{{\mathcal{M}}}

\newcommand{\clu}{{\mathcal{U}}}
\newcommand{\clv}{{\mathcal{V}}}

\newcommand{\one}{{\boldsymbol{1}}}

\newcommand{\Ebd}{{\boldsymbol{E}}}

\newcommand{\Pbd}{{\boldsymbol{P}}}

\newcommand{\Wbd}{{\boldsymbol{W}}}
\newcommand{\xbd}{{\boldsymbol{x}}}\newcommand{\Xbd}{{\boldsymbol{X}}}

\newcommand{\Zbd}{{\boldsymbol{Z}}}

\newcommand{\Om}{\Omega}

\usepackage{xcolor}
\newcommand{\ab}[1]{\textcolor{black}{#1}}

\mathtoolsset{showonlyrefs}
\endlocaldefs

\begin{document}

\begin{frontmatter}
\title{Fluctuations of the Atlas model from inhomogeneous stationary profiles}
\runtitle{Inhomogeneous stationary profile fluctuations}

\begin{aug}
\author[A]{\fnms{Sayan}~\snm{Banerjee}\ead[label=e1]{sayan@email.unc.edu}},
\author[A]{\fnms{Amarjit}~\snm{Budhiraja}\ead[label=e2]{budhiraj@email.unc.edu}}
\and
\author[A]{\fnms{Peter}~\snm{Rudzis}\ead[label=e3]{prudzis@unc.edu}}
\address[A]{Department of Statistics and Operations Research,
University of North Carolina - Chapel Hill\printead[presep={,\ }]{e1,e2,e3}}
\end{aug}

\begin{abstract}
The infinite Atlas model describes the evolution of a countable collection of Brownian particles on the real line, where the lowest particle is given a drift of $\gamma \in [0,\infty)$. The case $\gamma=0$ is referred to as the Harris model. In this work we study equilibrium fluctuations for the Atlas model for $\gamma \in [0,\infty)$ when the system of particles starts from an inhomogeneous stationary profile with exponentially growing density. We show that the appropriately centered and scaled occupation measure of the particle positions, with suitable translations, viewed as a space-time random field, converges to a limit that can be characterized in terms of a certain stochastic partial differential equation (SPDE). The initial condition for this equation is given by a Brownian motion, the equation is driven by an additive space-time noise that is white in time and colored in space, and the linear operator governing the evolution is the infinitesimal generator of a geometric Brownian motion. We use this SPDE to also characterize the fluctuations of the ranked particle positions with a suitable centering and scaling. Our results describe the behavior of the particles in the bulk and one finds that the Gaussian process describing the asymptotic fluctuations has the same H\"{o}lder regularity as a fractional Brownian motion with Hurst parameter $1/4$. This connection with a fractional Brownian motion becomes even more exact when the inhomogeneous stationary profiles approach the homogeneous stationary profile in a suitable manner. One finds that, unlike the setting of a homogeneous profile (Dembo and Tsai (2017)), the behavior on the lower edge of the particle system is very different from the bulk behavior and in fact the variance of the Gaussian limit diverges to $\infty$ as one approaches the lower edge. Indeed, our results show that, with the gaps between particles given by one of the inhomogeneous stationary distributions, the lowest particle, started from $0$, with a linear in time translation, converges in distribution to an explicit non-Gaussian limit as $t\to \infty$. In the 
special case of the Harris model this limit is given as the law of the difference of two independent Gumbel distributions.
\end{abstract}



\end{frontmatter}


\section{Introduction}
\label{sec:intro}  

The Atlas model describes the evolution of a collection of Brownian particles on the real line, where the lowest particle is given a drift of $\gamma \in [0,\infty)$. Its dynamics are expressed by the system of equations
\begin{equation} \label{eq:Atlassystem}
dX_i(t) = \gamma \mathbf{1}\left\{X_i(t) = \min_{0 \leq j < N} X_j(t)\right\} dt + dW_i(t), \quad 0 \leq i < N,
\end{equation}
where, depending on the context, $N$ is some positive integer or $N = \infty$, and the $W_i(\cdot)$ are independent standard Brownian motions. The concern of this paper is with the infinite Atlas model, that is, the case where $N = \infty$. Our main results describe the fluctuations of the occupation measure process encoding particle locations in the infinite Atlas model, started from spatially inhomogeneous configurations given in \eqref{eq:inhomogeneous2} that are characterized by exponentially growing particle density as one moves away from the lowest particle. These configurations are known to be stationary in time, up to a deterministic linear time translation, under the dynamics given in \eqref{eq:Atlassystem}, and will henceforth be referred to as `inhomogeneous stationary profiles'. Previously, fluctuation results were obtained by Dembo and Tsai \cite{dembo2017equilibrium} when the starting configuration is  distributed as a Poisson process on $\RR_+ \doteq [0,\infty)$ with intensity $2\gamma$ for $\gamma>0$. This initial configuration is known to keep the distribution of gaps between successive particles invariant in time, but not the particle locations themselves. We will, however, abusing terminology, refer to this configuration as the `homogeneous stationary profile'.

\subsection{Background}
When $\gamma = 0$ in \eqref{eq:Atlassystem} the resulting particle system, which was first studied by Harris \cite{harris1965collisions} and referred to throughout as
 the Harris model, has a history which goes back further than the general $\gamma$ case. See \cite{harris1965collisions} and \cite{durrgoldlebo1985asymptotics}. The Atlas model for general $\gamma$ first appears in the context of stochastic portfolio theory, as a prototypical example of the so-called rank-based diffusions. This class of models was introduced in work of Fernholz \cite{fernholz2002stochastic}, and further developed in \cite{banner2005atlas} and \cite{fernholz2009stochastic}, as a model for stochastic evolution of equity markets. 
Subsequently, this model and its variants have been studied extensively in several contexts including scaling limits of interacting particle systems like the simple exclusion process \cite{karatzas2016systems}, non-linear Markov processes \cite{shkol, jourdain2008propagation}, free boundary problems \cite{cabezas2019brownian} and stochastic control problems \cite{aldous41up,tang2018optimal}. Recently, rank-based diffusions have been shown to have connections with large deviation phenomena for the Kardar-Parisi-Zhang (KPZ) equation \cite{tsai2023high}.


Turning now to the case $N=\infty$, we describe the set of initial conditions for which \eqref{eq:Atlassystem} is well-posed. Following Tsai \cite{tsai_stat} we say a vector $\xbd = (x_i)_{i \in \NN_0} \in \RR^{\infty}$ is rankable if there is a  permutation $p: \NN_0 \to \NN_0$ such that $x_{p(i)} \le x_{p(j)}$ for all $0\le i<j<\infty$. We denote by $p_{\xbd}$ the unique such permutation in which ties are broken in lexicographic order. Define the configuration space
\begin{equation}
\clu \doteq \left\{\xbd = (x_i)_{i \in \NN_0} \in \RR^{\infty}: \lim_{k\to \infty} x_k = \infty,
\mbox{ and } \sum_{i=1}^{\infty} e^{-\alpha x_i^2} <\infty \mbox{ for all } \alpha \in (0, \infty)\right\}.
\end{equation}
Fix $\gamma \in [0, \infty)$. With the above terminology, a solution to \eqref{eq:Atlassystem} is a continuous adapted process that is rankable at each time instant, a.s., and solves the equation
\begin{equation}\label{eq:statproc}
dX_i(t) = \gamma \one\{p_{\Xbd(t)}(i) = 0\} dt + dW_i(t), \; i \in \NN_0.
\end{equation}
Here $\Xbd(t) = (X_i(t))_{i\in \NN_0}$. Similarly, we denote $\Wbd(t) = (W_i(t))_{i\in \NN_0}$. Proposition 3.1 in \cite{shkolnikov2011compete} implies there is a unique weak solution to the system of equations \eqref{eq:statproc} with state space $\clu$, for  any initial distribution $\mu \sim \Xbd(0) = (X_i(0))_{i\in \NN_0}$ that is supported on $\clu$.

\ab{Whether it is necessary for the initial data to lie in $\clu$ in order for a weak solution to exist is unknown to us. Note however that the initial data cannot be too closely packed---for instance, an infinite collection of i.i.d.\ Brownian motions started from a single point will be almost surely unrankable at each time $t > 0$, due to the presence of accumulation points.}

When
$\Xbd(t)$ is a solution, we denote,
$$
X_{p_{\Xbd(t)}(i)}(t) = X_{(i)}(t) \text{ for $i \in \NN_0$ and $t \ge 0$}.
$$
We assume throughout, without loss of generality, that  $X_i(0) = X_{(i)}(0), \, i \in \NN_0$. 
Previously, results on weak existence and uniqueness for the system \eqref{eq:statproc} were obtained in \cite{PP}, and those on strong solutions were described in \cite{ichiba2013strong}, for less general initial configurations. 

If $\gamma > 0$, a kind of competition ensues between the drift received by the lowest particle and the density profile of the higher up particles. This manifests itself in equilibrium behavior, either in the particle gap distribution, or in a re-centered version of the point process itself. 
We denote the sequence of gaps by
\begin{equation}\label{eq:gapt}
Z_i(t) \doteq X_{(i)}(t) - X_{(i-1)}(t), \quad 1 \leq i < N.
\end{equation}
\normalsize
The process $\Zbd(t) \doteq (Z_i(t))_{1 \leq i < N}$, of interest in its own right, can be viewed as a Brownian motion in the positive orthant $\mathbb{R}_+^N$ with oblique reflections from the boundary. For finite $N$, this allows for techniques to be applied from the well developed theory of reflected Brownian motion \cite{HR,harrison1987multidimensional}. Specifically, in the finite $N$ case, for $\gamma > 0$, \cite[Corollary 10]{PP} says that the sequence of gaps $(Z_i(t))_{1 \leq i \leq N - 1}$ converges in total variation distance as $t \to \infty$ to the product-form stationary distribution
\begin{equation} \label{eq:statfinite}
\bigotimes_{i = 1}^{N-1} \text{Exp}\left(2\gamma(1 - i/N)\right).
\end{equation}
In particular, the stationary measure is unique. By taking a formal limit as $N \rightarrow \infty$ above, it is natural to guess that
\begin{equation} \label{eq:homogeneous}
\pi_0^\gamma \doteq \bigotimes_{i = 1}^\infty \text{Exp}(2\gamma)
\end{equation}
is a stationary measure for the infinite Atlas model, which is established in \cite{PP}. In sharp contrast with the finite $N$ case, the measure \eqref{eq:homogeneous} is not the unique stationary distribution for the infinite Atlas model, and the paper \cite{sarantsev2017stationary} proves the existence of a one-parameter family of product-form stationary distributions, namely
\begin{equation} \label{eq:inhomogeneous}
\pi_a^\gamma \doteq \bigotimes_{i = 1}^\infty \text{Exp}(2\gamma + ia), \quad a > 0.
\end{equation}
We emphasize that \eqref{eq:statfinite} and \eqref{eq:homogeneous} only makes sense for a strictly positive drift $\gamma$, and indeed there is no stationary measure for the Harris model ($\gamma=0$) when $N < \infty$. On the other hand, for the case $N=\infty$, the measure \eqref{eq:inhomogeneous} is stationary for any $\gamma \geq 0$, provided that $a > 0$. 

For $i \in \mathbb{N}_0$, let $\tilde{X}_{(i)}(t) \doteq X_{(i)}(t) + \frac{at}{2}$. The process $(\tilde{X}_{(i)}(t))_{i \in \mathbb{N}_0}$ may be identified with the point process associated with its occupation measure $\sum_{i \in \mathbb{N}_0} \delta_{\tilde{X}_{(i)}(t)}$. Extending the stationarity results from \cite{sarantsev2017stationary}, Tsai shows in \cite{tsai_stat} that if $a > 0$, the point process $(\tilde{X}_{(i)}(t))_{i \in \mathbb{N}_0}$ is stationary with respect to the following measure 
\begin{equation} \label{eq:inhomogeneous2}
\nu_a^{\gamma}(A) \doteq \frac{1}{\Gamma(1+ 2\gamma/a)} \int_{A} e^{2\gamma x_{(0)}} d\Pi_a(\xbd),\; A \in \clb(\RR^{\infty}),
\end{equation}
for each $\gamma \ge 0$, where for $\xbd = (x_{i})_{i\in \NN_0}$, $x_{(i)} = x_{p_{\xbd}(i)}$, $i \in \NN_0$, \ab{and $\Pi_a$ is the law of a Poisson point process on $\mathbb{R}$ with intensity measure $ae^{ax}dx$}, regarded as a probability measure on $\RR^{\infty}$. 
Throughout, by the law of a point process, we will mean the probability measure on $\RR^{\infty}$ that corresponds to the law of the ordered vector of points in the point process.
Note that in the case of the Harris model ($\gamma=0$), the above result says that the stationary distribution is simply the Poisson point process $\Pi_a$. 

The normalization \eqref{eq:inhomogeneous2} is chosen to make $\nu_a^\gamma$ probability measure. To see this, let $(U_i)_{i \in \mathbb{N}_0} \sim \Pi_a$, where $U_0 < U_1 < U_2 < \cdots$ a.s. One may compute 
$
\Pbd(U_0 > y) = e^{-e^{ay}},
$
thus obtaining
\begin{equation}
\int_{\RR^{\infty}} e^{2\gamma x_{(0)}} d\nu_a^\gamma(\xbd) = \Ebd(e^{2\gamma U_0})
= a\int_{-\infty}^{\infty}e^{2\gamma u} e^{au} e^{-e^{au}} du = \Gamma(1+ 2\gamma/a). \label{eq:ZisGamma}
\end{equation}
By similar calculations, if $(X_{(i)}(0))_{i \in \mathbb{N}_0} \sim \nu_a^\gamma$, one can verify that the gap sequence $(Z_i(t))_{i \in \mathbb{N}}$ is distributed as $\pi_a^\gamma$ for all $t \geq 0$, recovering the stationarity result from \cite{sarantsev2017stationary}.

A number of works have recently investigated the long time behavior of the gap process $\Zbd(\cdot)$ for the infinite Atlas model. In particular, results on domains of attraction of $\pi_a^\gamma$ for $\gamma>0, \, a \ge 0,$ have been obtained in \cite{AS}, \cite{DJO} and \cite{banerjee2022domains}. For $\gamma=0, \, a>0$, $\{\pi_a^0 : a >0\}$ have been shown to be the only extremal (ergodic) invariant distributions for the gap process in \cite{RuzAiz}. For $\gamma>0$, characterizing all extremal invariant distributions remains a challenging open problem; \cite{banerjee2022extremal} shows that $\pi_a^\gamma$ is extremal for every $\gamma>0, \, a \ge 0$ and, moreover, these are the only product-form invariant distributions under mild regularity assumptions.

\subsection{Fluctuation results for homogeneous stationary profiles}
To put our results in context, we first summarize the results of Dembo and Tsai \cite{dembo2017equilibrium} on the fluctuations of the Atlas model started from  a Poisson process on $\RR_+ \doteq [0,\infty)$ with intensity $2\gamma$ with $\gamma>0$. 


We consider the coordinate processes $(\Xbd,\Wbd)$  defined on the canonical path space $$
(\Om, \clf) \doteq (\clc([0, \infty): \RR^{\infty}\times \RR^{\infty}), \clb(\clc([0, \infty): \RR^{\infty}\times \RR^{\infty})),$$ where for a Polish space $S$, $\clc([0, \infty): S)$ is the space of continuous functions from $[0, \infty)$ to $S$ that is equipped with the usual local uniform topology, $\RR^{\infty}$ is equipped with the product topology and for two Polish spaces $S_1$ and $S_2$, $S_1\times S_2$ is also equipped with the product topology. If $\mu$ is the distribution of $\Xbd(0)$ in $\mathbb{R}^\infty$, we denote by $\Pbd^{\gamma}_{\mu}$
the unique
probability measure on
$(\Omega,\mathcal{F})$ 
under which $\Wbd = (W_i, i \in \NN_0)$ is an infinite collection of independent standard Brownian motions with respect to the filtration
$\sigma\{(\Wbd(s), \Xbd(s)), 0\le s\le t\}$, $\Xbd(0)$ is distributed as $\mu$, and $\Xbd$ solves \eqref{eq:statproc} with driving Brownian motions $\{W_i, i \in \NN_0\}$.

For $\gamma > 0$, let $\mu_0^\gamma$ denote the distribution of a Poisson point process on $\mathbb{R}_+$ with intensity measure $2\gamma \mathbf{1}_{[0,\infty)}(y)dy$, viewed as a probability measure on $\RR^{\infty}$. We let $\Pbd_0^\gamma \doteq 
\Pbd_{\mu_0^\gamma}^\gamma$, denoting the corresponding expectation operator by $\Ebd_0^\gamma$. Note that under $\Pbd_0^\gamma$, the process $\Xbd$ has the stationary gap distribution $\pi_0^\gamma$, defined in \eqref{eq:homogeneous}.

The paper \cite{dembo2017equilibrium} considers the fluctuations, under the measure, $\Pbd_0^{\gamma}$, i.e. in the case where the process is started from $\Xbd(0) \sim \mu_0^\gamma$. Consider the random field 
\begin{equation}\label{ds}
\mathcal{X}^\veps(t,x) \doteq \veps^{1/4}\left( i_\veps(x) - 2\gamma X_{(i_\veps(x))}(\veps^{-1}t) \right),
\end{equation}
where $\veps > 0$, and for $x \geq 0$, $i_\varepsilon(x) \doteq \lfloor (2\gamma \varepsilon^{1/2})^{-1}x \rfloor$. Denote the usual Gaussian heat kernel by
\begin{equation} \label{eq:heatkernel}
p_t(z) \doteq (2\pi t)^{-1/2}e^{-z^2/2t}, t>0, \; z \in \RR.
\end{equation}
and denote the Neumann heat kernel on $[0,\infty)$ by 
$$p_t^{\text{N}}(y,x) \doteq p_t(y - x) + p_t(y + x), t>0,; x,y \in [0, \infty).
$$
Also define
$$
\tilde{\Psi}_t(y,x) \doteq 2 - \int_0^y p_t^{\text{N}}(z,x)dz, \; t>0.
$$
Let $\{\mathcal{X}(t,x), (t,x) \in [0,\infty) \times [0,\infty)\}$ be the continuous mean $0$ Gaussian process, with covariances given by 
$$
\Ebd[\mathcal{X}(t,x)\mathcal{X}(t',x')] = 2\gamma ( \int_0^\infty \tilde{\Psi}_t(y,x)\tilde{\Psi}_{t'}(y,x')dy + \int_0^{t \wedge t'} \int_0^\infty p_{t - s}^{\text{N}}(y,x) p_{t' - s}(y,x')dy ds).
$$
The main result of \cite{dembo2017equilibrium} shows that under $\Pbd_0^\gamma$, as $\veps \to 0$, $\mathcal{X}^\veps$ converges to 
$\mathcal{X}$, uniformly on compact subsets of $[0,\infty)\times [0,\infty)$, in distribution (see above Theorem \ref{thm:main1} for a precise formulation of this convergence).

The field $\mathcal{X}$ may equivalently be described as the unique 
mild
solution of the stochastic heat equation on $[0,\infty)$ with Neumann boundary condition:
$$
\partial_t \mathcal{X}(t,x) - \frac{1}{2}\partial_{xx} \mathcal{X}(t,x) = (2\gamma)^{1/2}\dot{\mathcal{W}}(t,x), \; \mathcal{X}(0,x) = (2\gamma)^{1/2}B(x), \;  (t,x) \in [0,\infty) \times [0,\infty).
$$
Here $\dot{\mathcal{W}}$ is a standard Gaussian white noise on $[0,\infty) \times [0,\infty)$, and $B$ is Brownian motion on $[0,\infty)$, independent of $\clw$. By a \textit{mild solution} of the above equation \ab{with Neumann boundary condition} one means the continuous random field defined as 
$$
\mathcal{X}(t,x) = (2\gamma)^{1/2}\left( \int_0^\infty \tilde{\Psi}_t(y,x)B(dy) + \int_{[0,t] \times (0,\infty)} p_{t - s}^{\text{N}}(y,x) \mathcal{W}(ds\, dy) \right),
$$
where the terms on the right-hand side are stochastic integrals with respect to white noise measures in the sense of Walsh \cite{walsh}. 
 \ab{Note in particular that   this evolution equation is described in terms of the Neumann heat kernels.}

\subsection{Our contributions}
In this work we are interested in studying fluctuations from equilibrium for the infinite Atlas model when the equilibrium is taken to be one of the measures $\nu_a^{\gamma}$ for $a>0$, corresponding to an inhomogeneous stationary profile, defined in \eqref{eq:inhomogeneous2}.
One expects a qualitatively different behavior in these regimes as, owing to the spatial inhomogeneity, one cannot uniformly scale space as in \eqref{ds}. The first simple but illuminating observation we make is that if $(X_{i}(0))_{i \in \mathbb{N}_0} \sim \nu_a^0$, then $(Y_{i}(0) \doteq \exp\{a(X_i(0)\})_{i \in \mathbb{N}_0}$ is distributed as a Poisson process with rate one (Lemma \ref{lem:ispp}).
Also, if instead $(X_{i}(0))_{i \in \mathbb{N}_0}$ is distributed as $\nu_a^{\gamma}$
for $\gamma >0$, then $(Y_{i}(0)-Y_{0}(0))_{i \in \mathbb{N}}$ is again a Poisson process (see second part of Lemma \ref{lem:ispp}).
From \cite{tsai_stat}, we know that $(X_{i}(t) + at/2)_{i \in \mathbb{N}_0}, \, t \ge 0,$ is a stationary process in time under
$\Pbd_a^{\gamma} \doteq \Pbd_{\nu_a^{\gamma}}^{\gamma}$
for any $\gamma\ge 0$. These two observations lead us to work with the process $\left(Y_i(t) \doteq \exp\{a(X_i(t) + \frac{at}{2})\}\right)_{i \in \mathbb{N}_0}, \, t \ge 0,$ as a key intermediate object towards proving the desired fluctuation results. This sequence can be viewed as a collection of rank-based geometric Brownian motions (see \eqref{eq:rgbm}).
A substantial portion of this article is dedicated to proving Theorem \ref{thm:formain1} which furnishes a fluctuation result for the occupation measure of (a rescaled version of) this process encoded via the space-time function defined in \eqref{eq:924}. Reformulating this result in terms of the occupation measure of the process $\Xbd(\cdot)$, we obtain our first main result Theorem \ref{thm:main1}. This theorem says the following. For $\veps \in (0,1)$, define
$$
X^{\veps}_i(t) \doteq X_i(t) + \frac{at}{2} + \frac{1}{2a}\log \veps, \; t \ge 0, i \in \NN_0,
$$
and for $t\ge0, \, x\in (-\infty, \infty)$, let 
$$
R^{\veps}(t,x) = \sum_{i=0}^{\infty}\one_{(-\infty,x]}(X^{\veps}_i(t)).
$$
Then Theorem \ref{thm:main1} shows that the random field
$$
\hat R^{\veps}(t,x) \doteq \veps^{1/4} (R^{\veps}(t,x) - \veps^{-1/2} e^{ax})
$$
converges, uniformly on compact subsets of $[0,\infty)\times (-\infty, \infty)$, in distribution, to $u(t, e^{ax})$, where $\{u(t,x), t\ge 0, x\in (0,\infty)\}$ is the solution of the stochastic partial differential equation (SPDE):
$$
\partial_t u(t,x) = \frac{a^2x^2}{2} \partial_{xx} u(t,x) + ax \dot{\clw}(t,x), \; u(0,x) = B(x), \; (t,x) \in [0,\infty)\times (0,\infty),
$$
where $B$ is standard Brownian motion and ${\clw}$ is the  white noise measure on $[0, \infty) \times (0, \infty)$, independent of $B$. 

This result, in turn, can be used to obtain fluctuation results for the order statistics processes $\{X_{(i)}(t), t\ge 0, i \in \NN_0\}$. Specifically, let
$
X^{\veps}_{(i)}(t) \doteq X_{(i)}(t) + \frac{at}{2} + \frac{1}{2a}\log \veps, \; t \ge 0, i \in \NN_0.
$
Then our second main result, Theorem \ref{thm:main2},  shows that the random field
$$
\hat \xi^{\veps}(t,x) \doteq \veps^{-1/4}\left(X^{\veps}_{(i_{\veps}(x))}(t) - \frac{\log x}{a}\right), \;\;
i_{\veps}(x) \doteq \lfloor \veps^{-1/2}x \rfloor, \ \ (t,x) \in [0,\infty) \times (0,\infty),
$$
weakly converges to $ \xi(t,x) \doteq \frac{1}{ax} u(t, x)$, uniformly on compact subsets of $[0,\infty)\times (0,\infty)$, where $u$ is the solution of the above SPDE with $u(0,x) = B(x)$. Using this result we see that, for each $x>0$, $\veps^{-1/4}\left(X^{\veps}_{(i_{\veps}(x))}(\cdot) - X^{\veps}_{(i_{\veps}(x))}(0)\right)$ converges weakly to a Gaussian process with an explicitly computable covariance function (see Corollary \ref{cor:fixxprocess}). In Remark \ref{rem:fBM}, we use this covariance structure to show that the sample paths of the associated Gaussian process have the same H\"older regularity as a \emph{fractional Brownian motion} with Hurst parameter $1/4$ (written $fBM(1/4)$). 
Further, this Gaussian process weakly converges to a normal random variable as $t \rightarrow \infty$, quantifying the tightness resulting from the dense packing of particles in the bulk of the ensemble.
We also note in this remark that, as the particle density parameter $a$ approaches zero, this process, appropriately rescaled, in fact converges in distribution to $fBM(1/4)$, `interpolating' our result with the fBM fluctuations for the ordered particles in the case $a=0$ obtained in \cite{dembo2017equilibrium}. 
Similar equilibrium fluctuations are known for a tagged particle in doubly-infinite Harris model(see \cite{harris1965collisions},\cite{durrgoldlebo1985asymptotics}) and for various types of discrete lattice models \cite{arrat, demfer, lanvol,rostvar,hernandez2017equilibrium}.

Finally, we investigate the time asymptotics of the lowest particle $X_{(0)}(\cdot)$. We show in Theorem \ref{thm:lpasymptotics} that, if the system is started with the lowest particle at the origin and gaps between particles distributed as $\pi_a^{\gamma}$ defined in \eqref{eq:inhomogeneous}, then $X_{(0)}(t) + at/2$ converges weakly to a random variable with an explicit non-Gaussian distribution depending on $\gamma$. In the special case of $\gamma=0$ this corresponds to the distribution of the difference between two independent Gumbel random variables.

Several remarks are in order. First, in contrast with \cite{dembo2017equilibrium}, where the time is scaled by a factor of $\veps^{-1}$ (see \eqref{ds}), here
there is \emph{no time scaling} with respect to $\veps$. Also, unlike in \cite{dembo2017equilibrium}, where the centering term $i_{\veps}(x)$ involves a scaling of $x$ by $\veps^{-1/2}$, here, due to the increasing density of particles, in addition to a linear translation in time, the centering is given as $\log x$ plus an order $\log \veps$ term.
Second, the SPDE we obtain is of a  different form
in that the linear operator in the equation, instead of being a Laplacian (with a suitable boundary condition), is given as the generator of a geometric Brownian motion, and the noise term, instead of being white in space and time, has a spatial coloring. Third, whereas the fluctuation result in \cite{dembo2017equilibrium} extends all the way to the lowest particle (namely, it holds for $X_{(i_{\veps}(x))}(\cdot)$ for all $x \ge 0$), our fluctuation result holds only in the bulk (for $X_{(i_{\veps}(x))}(\cdot)$ with $x>0$) (this also means that, unlike \cite{dembo2017equilibrium},  our equation does not involve a boundary condition in $x$). This difference is indeed qualitative: whereas \cite{dembo2017equilibrium} use their results to conclude that the lowest particle (which starts from $0$), appropriately rescaled, converges to a fractional Brownian motion, we show in Theorem \ref{thm:lpasymptotics} that the lowest particle in our setting, starting from $0$ and translated by a linear in time factor, weakly converges in law as $t \rightarrow \infty$ without any $\veps$-scaling to a non-Gaussian limit. Thus, for inhomogeneous stationary profiles, \emph{particles near the lower edge behave very differently than particles in the bulk}.

At a  technical level, the fact that the solution to the SPDE that appears in our work involves the transition probability kernel $q_t$ of a \emph{geometric Brownian motion} (see \eqref{eq:spdesoln}), as compared with the heat kernel of standard Brownian motion appearing in \cite{dembo2017equilibrium}, makes the tightness proofs more demanding (see Section \ref{sec2.1}) that 
lead us to develop estimates on the transition kernel $q_t$ and the related kernel $\hat q_t$, and on moments and probabilities of a collection of rank-based geometric Brownian motions,
that we expect would be of independent interest.

Finally, regarding Theorem \ref{thm:lpasymptotics}, we note here that \cite[Corollary 1.3]{tsai_stat} establishes concentration results for $X_{(0)}(t) + at/2$ for any $t \ge 0$ (when $X_{(0)}(0)=0$ and gaps are distributed as $\pi_a^{\gamma}$). In comparison, we identify distributional convergence as $t \rightarrow \infty$ and give an explicit formula for the limit law. See Remark \ref{rem:tsa} for additional discussion on this point.

\subsection{Proof overview}
We begin by noting that several constructions and proof strategies in our work are inspired by \cite{dembo2017equilibrium}. As noted previously,  we work with the rescaled versions of the  point processes $\{Y_i(t), i \in \NN_0\}$. The associated suitably centered and normalized occupation measure can be analyzed through the random field $\{\check \clv^{\veps}(t,x), (t,x) \in [0, \infty)\times (0,\infty)\}$, defined in \eqref{eq:924}, by considering indicator functions
of the form $\mathbf{1}_{(0,x]}$, $x >0,$ as test functions. In order to study the time evolution of these fields, it is convenient to replace the indicator function with a suitably smooth mollification which allows for an application of It\^{o}'s formula. The time evolution for the mollified random field, which is captured by \eqref{eq:VWerrM} is at the heart of the analysis. The error introduced by replacing the original field by the mollified field is controlled in Lemma \ref{lem:secreq} and the tightness of the various terms in the evolution
equation in \eqref{eq:VWerrM} is established in Lemma \ref{lem:tight}. The remaining work in the proof of Theorem \ref{thm:main1} is identifying the limit points, which is first done under $\Pbd^{0}_a$, and then, using a change of measure idea, for $\Pbd^{\gamma}_a$ for $\gamma \neq 0$. For our second main result, Theorem \ref{thm:main2}, the main ingredient is Proposition \ref{prop:yaxis} which, together with the first result, gives the fluctuation for the ranked particle system 
$\{Y_{(i)}(t), i \in \NN_0\}$. The result for the ranked $\{X_{(i)}(t), i \in \NN_0\}$ as given in Theorem \ref{thm:main2} then follows by a log transform and analyzing the terms in a related Taylor series expansion.
Finally, for the proof of the lowest particle asymptotics in 
Theorem \ref{thm:lpasymptotics} the main idea is to show that
$\tilde{X}_{(0)}(0)$ and $\tilde{X}_{(0)}(t)$ are asymptotically independent as $t \to \infty$, under $\Pbd^{\gamma}_a$. This is done by showing that the lowest particle at time $0$ migrates to the bulk at large times $t$, and its dynamics consequently decouples from the lower edge.

\subsection{Related Work and Future Research}
Here we mention two related avenues of work and possibilities for future research in these directions.

In \cite{hernandez2017equilibrium}, the authors obtain SPDE limits (analogous to \cite{dembo2017equilibrium}) for stationary current fluctuations of a zero-range process with a source at the origin, which they view as a \emph{discrete Atlas model}. Although it can be checked that this model has a unique invariant distribution, a natural question is to obtain discrete interacting particle systems which share the characteristics of the Atlas model with inhomogeneous stationary profiles and to investigate their stationary fluctuations.

Another interesting direction involves \emph{non-equilibrium hydrodynamic limits and fluctuations}. \ab{The key reason behind taking the initial distribution to be of the form \eqref{eq:inhomogeneous2} is to exploit \emph{stationarity properties} of the joint particle laws in time, which makes some calculations simpler. Although not explored here, by a coupling/comparison argument, it may be possible to extend our results
to a class of perturbations of the initial distribution in \eqref{eq:inhomogeneous2} (thereby yielding some form of
universality). In future research, we will investigate hydrodynamic limits and  fluctuations for more general classes of, possibly inhomogeneous, out of equilibrium starting profiles.}

Some progress has been made in this direction. In \cite{cabezas2019brownian}, a hydrodynamic limit is obtained for the occupation measure of the Atlas model started from points distributed as a Poisson process with rate $\lambda$ for general $\lambda>0$. The limit is obtained in terms of a Stefan problem and applies to a class of non-equilibrium homogeneous starting profiles. \ab{Non-equilibrium fluctuations have been investigated} for some interacting particle systems in \cite{franceschini2023non,gonccalves2022clt} via refined control on relative entropies and, among other things, we plan to explore this technique in the context of rank-based diffusions and related discrete particle systems.

\section{Main Results}
To state our main results, let $\gamma \in [0,\infty)$, and recall the stationary measures $\nu_a^\gamma$, $a > 0$, defined above in \eqref{eq:inhomogeneous2}. Recall $\Pbd_a^\gamma \doteq \Pbd_{\nu_a^\gamma}^\gamma$, and denote the corresponding expectation operator by $\Ebd_a^\gamma$.
Note from \eqref{eq:inhomogeneous2} and stationarity that, for all $t \in [0,\infty)$, and nonnegative $\sigma\{\Xbd(t)\}$ measurable random variables $U$,
\begin{equation} \label{eq:margCOM}
\Ebd_a^\gamma[U] = \Ebd_a^0[e^{2\gamma(X_{(0)}(t) + \frac{at}{2})}U].
\end{equation}


Let $\clm(\RR)$ denote the space of all $\sigma$-finite measures on $(\RR, \clb(\RR))$
equipped with the smallest $\sigma$-field that makes the function $\mu \mapsto \mu(B)$ a Borel measurable map from $\clm(\RR)$ to $\RR$ for every $B \in \clb(\RR)$. 
For $\veps \in (0,1)$, let
\begin{equation}
X^{\veps}_i(t) \doteq X_i(t) + \frac{at}{2} + \frac{1}{2a}\log \veps, \; t \ge 0, i \in \NN_0,
\end{equation}
and define a stochastic process $\{\clr^{\veps}(t)\}_{t\ge 0}$ with values in $\clm(\RR)$ as
$$\langle\clr^{\veps}(t), \phi\rangle \doteq \sum_{i=0}^{\infty} \phi(X^{\veps}_i(t)), \; \phi \in \Mmb_+(\RR),$$
where $\Mmb_+(\RR)$ is the space of real nonnegative measurable functions on $\RR$.
Consider the random field $\{R^{\veps}(t,x), (t,x) \in \RR_+ \times \RR\}$ defined as
$$R^{\veps}(t,x) \doteq \langle\clr^{\veps}(t), \one_{(-\infty, x]}(\cdot)\rangle.$$
Note that $\Ebd^0_{a} R^{\veps}(t,x) = \veps^{-1/2} e^{ax}$ for all $t \ge 0$ and  $x \in \RR$.
Define the centered random field $\{\hat R^{\veps}(t,x), (t,x) \in \RR_+ \times \RR\}$
as
\begin{equation}
\hat R^{\veps}(t,x) \doteq \veps^{1/4} (R^{\veps}(t,x) - \veps^{-1/2} e^{ax}).
\end{equation}

We will study the weak convergence of the above random field in a suitable  topology. In order to describe the weak limit of the above random field we introduce the following stochastic partial differential equation (SPDE):
\begin{equation}\label{eq:spde}
\partial_t u(t,x) = \frac{a^2x^2}{2} \partial_{xx} u(t,x) + ax \dot{\clw}, \; u(0,x) = u_0(x), \; (t,x) \in \RR_+ \times (0,\infty),
\end{equation}
for a fixed $u_0 \in \clc((0, \infty): \RR)$, where $\dot{\clw}$ is the  white noise measure on $[0, \infty) \times (0, \infty)$ (see \cite{walsh}).
The mild solution of such a SPDE is the continuous $u:\RR_+ \times (0,\infty) \to \RR$
given as
\begin{equation}\label{eq:spdesoln}
u(t,x) = \int_0^{\infty} q_t(y,x) u_0(y) dy +
a\int_{[0,t]\times (0, \infty)} y q_{t-s}(y, x) \clw(ds\, dy), \; (t,x) \in \RR_+ \times (0,\infty),
\end{equation}
where, the second term on the right side in \eqref{eq:spdesoln} is a stochastic integral with respect to the white noise measure in the sense of Walsh \cite{walsh}, 
and  for $(t,x,y) \in (0, \infty)\times (0,\infty)\times (0,\infty)$
\begin{equation}\label{eq:qdef}
q_t(x,y) = \frac{1}{ax} p_t\left(\frac{\log x}{a} - \frac{\log y}{a} + \frac{at}{2}\right),
\end{equation}
with $p_t(\cdot)$ being the heat kernel defined by \eqref{eq:heatkernel}. It is easy to verify that $y \mapsto q_t(x,y)$ defined above is the fundamental solution of the linear partial differential equation (PDE)
$
\partial_t \varphi = \frac{a^2y^2}{2}\partial_{yy} \varphi$,
which justifies referring to $u$ defined by \eqref{eq:spdesoln} as the solution of the SPDE \eqref{eq:spde}. Sometimes to emphasize the dependence on the initial condition, we will denote the solution $u$ of \eqref{eq:spde} as $u_{u_0}$.


Given an open set $\clo$ in $\RR$, we will say a random field $\{\clz^{\veps}(t,x), (t,x) \in \RR_+\times \clo\}$ converges, uniformly on compact subsets of $\RR_+ \times \clo$, in distribution to a continuous random field
$\{\clz(t,x), (t,x) \in \RR_+ \times \clo\}$, if there exists a sequence $\{\tilde \clz^{\veps}(t,x), (t,x) \in \RR_+ \times \clo\}$ of continuous random fields such that $\tilde \clz^{\veps} \to \clz$ in distribution in $C(\RR_+ \times \clo : \RR)$, as $\veps \to 0$, where the latter space is equipped with the local uniform topology and 
$$\sup_{0\le t \le T} \sup_{x \in K} |\tilde \clz^{\veps}(t,x) - \clz^{\veps}(t,x)| \to 0 \mbox{ in probability, as } \veps\to 0,$$
for every $T \in (0,\infty)$ and compact $K \subset \clo$.

We can now present our first main result. The proof is given in Section \ref{sec:mainthmpfs}.
\begin{theorem}\label{thm:main1}
Fix $a \in (0, \infty)$ and $\gamma \ge 0$. 
Let $B$ and $W$ be mutually independent white noise measures on $(0,\infty)$ 
and $\RR_+ \times (0,\infty)$, respectively. Then, under $\Pbd^{\gamma}_a$, as $\veps\to 0$, $\hat R^{\veps}$ converges in distribution,  uniformly on compact subsets of $\RR_+\times \RR$, to $ R(t,x) \doteq u(t, e^{ax})$ where $u$ is the solution of the SPDE in \eqref{eq:spde} with $u_0(x) = B((0,x])$.
\end{theorem}
We remark that,  using integration by parts, the $u$ given in the above theorem can be written as a sum of two stochastic integrals, see \eqref{eq:832}.

The next result gives the fluctuations of the ranked particles.
Let
$$X^{\veps}_{(i)}(t) \doteq X_{(i)}(t) + \frac{at}{2} + \frac{1}{2a}\log \veps, \; t \ge 0, i \in \NN_0,$$
and define for $(t,x) \in \RR_+ \times (0,\infty)$,
\begin{equation}\label{eq:defnhatxin}
\hat \xi^{\veps}(t,x) \doteq \veps^{-1/4}\left(X^{\veps}_{(i_{\veps}(x))}(t) - \frac{\log x}{a}\right), \;\;
i_{\veps}(x) \doteq \lfloor \veps^{-1/2}x \rfloor.
\end{equation}
The fact that  $\frac{\log x}{a}$ is the appropriate centering can be seen by recalling  that $\{Y_i(t) = \exp\{a(X_{i}(t) + \frac{at}{2})\}, i \in \NN_0\}$ are points of a standard Poisson process under $\Pbd^{0}_a$.

\begin{theorem}\label{thm:main2}
Fix $a \in (0, \infty)$ and $\gamma \ge 0$. 
Let $B$ and $W$ be mutually independent white noise measures on $(0,\infty)$ 
and $\RR_+ \times (0,\infty)$. Then, as $\veps\to 0$, $\hat \xi^{\veps}$ converges in distribution, under $\Pbd^{\gamma}_a$, uniformly on compact subsets of $\RR_+ \times (0,\infty)$, to $ \xi(t,x) \doteq \frac{1}{ax} u(t, x)$ where $u$ is the solution of the SPDE in \eqref{eq:spde} with $u_0(x) = B((0,x])$.
\end{theorem}

As a consequence of the above we have the following corollary. 
Throughout, `$\Rightarrow$' will denote convergence in distribution.
\begin{corollary} \label{cor:fixxprocess}
Fix $a \in (0, \infty)$, $\gamma \ge 0$ and $x \in (0,\infty)$.
Then, as $\veps \to 0$, under $\Pbd^{\gamma}_a$,
\begin{equation} \veps^{-1/4}\left(X^{\veps}_{(i_{\veps}(x))}(\cdot) - X^{\veps}_{(i_{\veps}(x))}(0)\right) \Rightarrow G^x(\cdot) \doteq \frac{1}{ax}(u(\cdot,x) - u(0,x)), \label{eq:fixx}
\end{equation}
in $C([0, \infty): \RR)$. Furthermore, $G^x(\cdot)$ is a continuous, mean-zero Gaussian process with covariances, for $t, t' \in [0,\infty)$, given as
$$
\Ebd_a^\gamma[G^x(t)G^x(t')] = \frac{2}{a^2 x} \left(\Phi\left(\frac{a}{2}t^{1/2}\right) + \Phi\left(\frac{a}{2}(t')^{1/2}\right) - \Phi\left(\frac{a}{2}|t' - t|^{1/2}\right) - \frac{1}{2}\right).
$$
Here $\Phi(\xi) \doteq \int_{-\infty}^\xi p_1(z)dz$ is the standard normal cumulative distribution function (cdf).

\end{corollary}

Indeed, it is immediate from Theorem \ref{thm:main2} that the left-hand side of \eqref{eq:fixx} converges weakly to $\frac{1}{ax}(u(\cdot,x) - u(0,x))$, and the rest is just a matter of calculating the covariances. See Section \ref{ssec:corproof}.

\begin{remark}\label{rem:fBM}
    The above corollary leads to the following observations.

    (i) For fixed $x>0$, the process $G^x(\cdot)$ has the same local path behavior as a fractional Brownian motion with Hurst parameter $1/4$ (abbreviated as $fBM(1/4))$. In fact, by directly checking the conditions in \cite[Theorems 6 and 7, Page 50]{marcus1968holder}, one obtains the following:
    \begin{align}
    \limsup_{h \rightarrow 0} \frac{|G^x(t+h) - G^x(t)|}{h^{1/4}\left(\log \log \frac{1}{h}\right)^{1/2}} &= \frac{\sqrt{ax} \, \pi^{1/4}}{2^{3/4}} \ \text{ almost surely for any } t \ge 0,\\
    \limsup_{|t'-t|=h \rightarrow 0; \, \, t',t \in [0,1]} \frac{|G^x(t+h) - G^x(t)|}{h^{1/4}\left(\log \frac{1}{h}\right)^{1/2}} &= \frac{\sqrt{ax} \, \pi^{1/4}}{2^{3/4}} \ \text{ almost surely.}
    \end{align}
    (ii) Highlighting dependence on $a$, consider the process $H^x_a(t) \doteq \frac{(2\pi)^{1/4}\sqrt{ax}}{\sqrt{2}}G^x(t)$, $t \ge 0$. Then, it can be checked by computing limits of covariances and using the fact that all the considered processes are Gaussian that
    $$
    H^x_a(\cdot) \Rightarrow fBM(1/4) \ \text{as} \ a \rightarrow 0
    $$
    in $C([0,\infty):\RR)$. Hence, as the initial profile becomes less densely packed (captured by $a \rightarrow 0$), the rescaled order statistics fluctuations approach that of $fBM(1/4)$. This result can be interpreted as `interpolating' our result with that of \cite[Corollary 1.4]{dembo2017equilibrium} which proves that analogous fluctuations in the case of the homogeneous stationary profile (i.e. when $a=0$) indeed converge to $fBM(1/4)$.\\\\
    (iii) For any fixed $x>0$, Corollary \ref{cor:fixxprocess} implies that $G^x(t)\Rightarrow N(0, 2/(a^2x))$ as $t \rightarrow \infty$. Thus, unlike the homogeneous case, the process limits of rescaled fluctuations of order statistics are tight in time, which intuitively results from the dense packing of particles in the bulk for inhomogeneous stationary profiles. Moreover, the increasing variance of the normal random variable as $x \rightarrow 0$ reveals the higher variability of order statistics as one approaches the lower edge of the particle ensemble.
\end{remark}

The above results tell us about the behavior of the particles in the bulk. The next result gives us information on the asymptotics of the lowest particle.
\begin{theorem} \label{thm:lpasymptotics}
Fix $a \in (0, \infty)$ and $\gamma \ge 0$.
Then, under $\Pbd^{\gamma}_a$, as $t\to \infty$, $(X_{(0)}(0), X_{(0)}(t) + at/2)$
converges in distribution to $(\eta_1,\eta_2)$, where $\eta_1$ and $\eta_2$ are independent random variables with density function, $f_{\eta}$ given as
\begin{equation}\label{eq:feta}
f_\eta(x) \doteq \frac{1}{\Gamma(1+ 2\gamma/a)} e^{(2\gamma+a)x} e^{-e^{ax}}.\end{equation}
In particular, suppose we denote by $\tilde \nu_a^{\gamma}$ the law of the infinite vector
$$0=\tilde U_0< \tilde U_1 < \ldots <\infty,$$
where the vector $(\tilde U_i -\tilde U_{i-1})_{i \in \NN}$ is distributed as 
$$\pi^{\gamma}_a \doteq \otimes_{i=1}^{\infty} \mbox{Exp}(2\gamma+ia).$$
Then, under $\Pbd^{\gamma}_{\tilde \nu_a^{\gamma}}$,
$X_{(0)}(t)+at/2$ converges in distribution to $\eta_2-\eta_1$.
\end{theorem}
The second part of the above theorem gives the time asymptotic behavior of the lowest particle in the Atlas and Harris models when initially the lowest particle starts at $0$ and the gaps are given according to the stationary distribution $\pi_a^{\gamma}$.
We remark that in the case when $\gamma=0$ (Harris model) the random variables $\eta_i$ are distributed as independent Gumbel random variables.  
\begin{remark}
\label{rem:tsa}
In \cite{tsai_stat} Tsai establishes a related result. In the notation of the current paper, the result in \cite{tsai_stat} says that, for $a>0$ and $\gamma \ge 0$, there is a $c\in (0,\infty)$ such that
$$\Pbd^{\gamma}_{\tilde \nu_a^{\gamma}}(
|X_{(0)}(t)+at/2| \ge x) \le c \exp\{-\frac{1}{2}(2\gamma+a)x\}, \mbox{ for all } t\ge 0 \mbox{ and } x \ge 0.$$
We note that the above result is valid for all $t\ge 0$ whereas Theorem 
\ref{thm:lpasymptotics} characterizes the behavior as $t\to \infty$. On the other hand Theorem 
\ref{thm:lpasymptotics} gives an explicit simple form  limit law instead of an upper bound on tail probabilities.
\end{remark}


\section{Proof of Fluctuation Results}
In this section we will prove Theorems \ref{thm:main1} and
\ref{thm:main2}. In the statements of lemmas and theorems in this and subsequent sections, we denote constants by $C$, $C_1$, $C_2$, etc., sometimes writing expressions such as $C = C(p_1,p_2,\dots,p_n)$ if we wish to emphasize the dependence of $C$ on parameters $p_1, p_2 \dots, p_n$. Within proofs, we also denote constants by $c_1, c_2, c_3, $ etc., where it is implicitly understood that these may depend on the parameters fixed at the beginning of the proof.

Define 
\begin{equation}
Y_i(t) \doteq \exp\{a(X_i(t) + \frac{at}{2})\}, 
\; Y_i^{\veps}(t) \doteq \veps^{1/2} Y_i(t),\; t \ge 0, \; i \in \NN_0, \; \veps \in (0,\infty).
\end{equation}
The following lemma says that, under $\Pbd^{0}_a$, for every $t \ge 0$, 
$\{Y_i(t), i \in \NN_0\}$ are points of a rate $1$ Poisson process. Let $Y_{(i)}(t) \doteq \exp\{a(X_{(i)}(t) + \frac{at}{2})\}$ for  $t\ge 0$, $i \in \NN_0.$
\begin{lemma}\label{lem:ispp}
For $t \ge 0$ and $x \ge 0$, let 
$N_t(x) = \sum_{i=0}^{\infty} \one_{[0,x]}(Y_i(t))$.
Then, for all $a>0$, under $\Pbd^{0}_a$, for every $t \ge 0$, $\{N_t(x), x \ge 0\}$ is a rate $1$ Poisson process.

Furthermore, letting $N_t^*(x) = \sum_{i=1}^{\infty} \one_{[0,x]}(Y_{(i)}(t) - Y_{(0)}(t))$ for $t \ge 0$ and $x \ge 0$, we have that, for every $a>0$ and  $t \ge 0$, under $\Pbd^{\gamma}_a$,  $\{N_t^*(x), x \ge 0\}$ is a rate $1$ Poisson process.
\end{lemma}

\begin{proof}
Fix $t\ge 0$ and $a>0$. Under $\Pbd_a^0$,  if we view $(X_i(t) + \frac{at}{2})_{i \in \mathbb{N}_0}$ as a point process on $\mathbb{R}$, then its distribution is given by the probability measure $\nu_a^0$, defined by \eqref{eq:inhomogeneous2}, namely a Poisson point process on $\mathbb{R}$ with intensity measure $\lambda_a(dx) \doteq ae^{ax}dx$. \ab{Consequently, by the Poisson mapping theorem (see p.\ 18 of \cite{kingBook})}, if we define $H_a : \mathbb{R} \to (0,\infty)$ by $H_a(x) \doteq e^{ax}$, then $(Y_i(t))_{i \in \mathbb{N}} = (H_a(X_i(t) + at/2))_{i \in \mathbb{N}_0}$ is distributed as a Poisson point process on $(0,\infty)$ with intensity measure $\lambda_a \circ H_a^{-1}(dy) = \mathbf{1}(0,\infty)dy$, from which the first statement follows.

To prove the second statement, it suffices to show that, for each $t \geq 0$, under $\Pbd^\gamma_a$, the increments $(Y_{(i+1)}(t) - Y_{(i)}(t))_{i \in \mathbb{N}}$ form a sequence of i.i.d.\ Exp(1) random variables. This is immediate from the formula \eqref{eq:margCOM}, and the fact that, by the first statement, under the measure $\Pbd^0_a$, the increments $(Y_{(0)}(t), Y_{(1)}(t) - Y_{(0)}(t), Y_{(1)}(t) - Y_{(0)}(t), \dots)$ are distributed as i.i.d.\ Exp(1) random variables.
\end{proof}

Note that, for $a \in (0, \infty)$ and $\gamma \ge 0$, under $\Pbd^{\gamma}_a$, $\{Y_i^{\veps}\}_{i \in \NN_0}$ solve the system of equations, for $t\ge 0$ and $i \in \NN_0$,
\begin{equation}
Y_i^{\veps}(t) = Y_i^{\veps}(0) + a^2\int_0^t Y_i^{\veps}(s) ds +
a\int_0^t Y_i^{\veps}(s) \one\{Y_i^{\veps}(s) = Y_{(0)}^{\veps}(s)\} ds  + a\int_0^t Y_i^{\veps}(s) dW_i(s),
\label{eq:rgbm}
\end{equation}
where
$Y_{(i)}^{\veps}(t) \doteq \veps^{1/2}Y_{(i)}(t)$,
$t\ge 0$,  $i \in \NN_0$. Through these evolution equations, $\{Y_i^{\veps}(\cdot)\}_{i \in \NN_0}$ can be viewed as a collection of rank-based geometric Brownian motions.

The following lemma gives an important moment bound. 



\begin{lemma}
\label{lem:tailbd}
Fix $a \in (0,\infty)$, $\gamma \geq 0$, $m > 1$, $q \geq 1$, and $T \in (0,\infty)$. Then there exists $C \in (0,\infty)$ such that for all $\veps > 0$ and $j >1$,
\begin{equation} \label{eq:unordered}
\sum_{i = j}^\infty \left[\Ebd^{\gamma}_a \sup_{0 \leq t \leq T} (Y_i^\veps(s))^{-mq} \right]^{1/q} \leq C\veps^{-m/2}j^{-m + 1},
\end{equation}
and
\begin{equation} \label{eq:ordered}
\left[\Ebd^{\gamma}_a \left(\sum_{i = j}^\infty \sup_{0 \leq t \leq T} (Y_{(i)}^\veps(s))^{-m}\right)^q \right]^{1/q} \leq C\veps^{-m/2}j^{-m + 1}.
\end{equation}
\end{lemma}


\begin{proof}
The proof uses some constructions and ideas from  Lemma 3.1 in \cite{dembo2017equilibrium}. For $i \in \mathbb{N}_0$ and $t \geq 0$, let $X_i^\veps(t) \doteq X_i(t) + a^{-1}\log \veps^{1/2}$, and let
\begin{align*}
X_i^{\veps,\ell}(t) \doteq X_i^\veps(0) + W_i(t), \;
Y_i^{\veps,\ell}(t) \doteq e^{a\left( X_i^{\veps,\ell}(t) + \frac{at}{2} \right)}.
\end{align*}
Note that $Y_i^{\veps,\ell}(t) \leq Y_i^{\veps}(t)$ for all $i \in \mathbb{N}_0$ and all $t \geq 0$, a.s. For a fixed $r \geq 1$, we have 
\begin{multline} 
\Ebd_a^\gamma\left[ \sup_{0 \leq t \leq T} Y_i^\veps(t)^{-r} \right]  \leq \Ebd_a^\gamma\left[ \sup_{0 \leq t \leq T} Y_i^{\veps,\ell}(t)^{-r} \right] \\
 = \Ebd_a^\gamma\left[ \sup_{0 \leq t \leq T} \veps^{-r/2}Y_i(0)^{-r} e^{-r\left( a W_i(t) + \frac{a^2t}{2} \right)} \right] 
 \leq \veps^{-r/2} \Ebd_a^\gamma\left[ Y_i(0)^{-r} \right] \Ebd_a^\gamma\left[ e^{ra \overline{W}_i(T)} \right], \label{eq:supYmoment}
\end{multline}
where $\overline{W}_i(T) = \sup_{0 \leq t \leq T} W_i(t)$. By the reflection principle,
\begin{equation} \label{eq:refprinc}
\Ebd_a^\gamma\left[ e^{ra \overline{W}_i(T)} \right] \leq 2 \Ebd_a^\gamma\left[ e^{ra W_i(T)} \right] \doteq c_1 = c_1(r,a,T).
\end{equation}
Furthermore, from Lemma \ref{lem:ispp}, under $\Pbd_a^0$, for $i \in \NN$, the difference $Y_i(0) - Y_0(0)$ is distributed as a sum of $i$ independent, $\text{Exp}(1)$ random variables, \ab{namely $\text{Gamma}(i,1)$}. 
Also note that for $p \in \RR_+$,
$\Ebd_a^0(e^{pX_{(0)}(0)}) = \int_0^{\infty} u^{p/a} e^{-u} du <\infty.$
Therefore, recalling \eqref{eq:inhomogeneous2},  for $i \geq 2r + 1$,
\begin{multline}
\Ebd_a^\gamma\left[ Y_i(0)^{-r} \right]  = \Ebd_a^0\left[ e^{2\gamma X_{(0)}(0)} Y_i(0)^{-r} \right] 
 \leq \Ebd_a^0\left[ e^{4\gamma X_{(0)}(0)} \right]^{1/2} \Ebd_a^0\left[ Y_i(0)^{-2r} \right]^{1/2} \\
 \leq c_2 \left( \int_0^\infty \frac{x^{i - 2r - 1}e^{-x}dx}{(i - 1)!} \right)^{1/2} 
 \leq c_2\left( \frac{\lceil i - 2r - 1 \rceil !}{(i - 1)!} \right)^{1/2}, \label{eq:initmoment}
\end{multline}
where $c_2 =(\int_0^{\infty} u^{4\gamma/a} e^{-u} du)^{1/2}$ and the last inequality follows on noting that
 $x\mapsto \Gamma(x)$ is monotonically increasing on $[1,\infty)$.
 Thus with some $\tilde c_2$ (depending on $r$)
 \begin{equation}\label{eq:ip1}
\Ebd_a^\gamma\left[ Y_i(0)^{-r} \right]\le \tilde c_2 (i+1)^{-r} \mbox{ for all } i\ge 1.
 \end{equation}
 By \eqref{eq:supYmoment}, \eqref{eq:refprinc}, and \eqref{eq:ip1}, for $i \geq 1$,
$$
\Ebd_a^\gamma\left[ \sup_{0 \leq t \leq T} Y_i^\veps(t)^{-r} \right] \leq c_3\veps^{-r/2} (i+1)^{-r},
$$
where $c_3=\tilde c_2 c_1(r,a,T)$. It follows that, for any $m > 1$, $q \geq 1$, and $j >1$, 
$$
\sum_{i = j}^\infty \left( \Ebd_a^\gamma\left[ \sup_{0 \leq t \leq T} Y_i^\veps(t)^{-mq} \right] \right)^{1/q} \leq c_4 \sum_{i = j+1}^\infty (\veps^{1/2}i)^{-m} \leq c_5\veps^{-m/2}j^{-m + 1},
$$
where $c_4, c_5$ depend only on $c_3$ and $q$.
This proves \eqref{eq:unordered}.

To prove \eqref{eq:ordered}, we first introduce some notation. Suppose $\overline{s} = (s_i : i \in \mathbb{N}_0)$ is a sequence in $\mathbb{R}$ which satisfies the following property: (P) \textit{For any $d_1 \in (\inf\overline{s}, \infty)$, the interval $[d_1, \infty)$ contains only finitely many points of the sequence $\overline{s}$}. We define $\left(\text{Rk}_i(\overline{s}) : i \in \mathbb{N}_0 \right)$ to be the unique sequence (associated with the lexicographic tie-breaking) in $\mathbb{R}$ such that the set equality $\{s_0, s_1, s_2, \dots\} = \{\text{Rk}_0(\overline{s}), \text{Rk}_1(\overline{s}), \text{Rk}_2(\overline{s}), \dots\}$ holds, and $\text{Rk}_0(\overline{s}) \geq \text{Rk}_1(\overline{s}) \geq \text{Rk}_2(\overline{s}) \geq \cdots$. The operators $\text{Rk}_j$ are order-preserving, in the sense that if $\overline{u} = \{u_i : i \in \mathbb{N}_0\}$ is another sequence satisfying property (P) such that $s_j \le u_j$ for all $j \in \NN_0$, then $\text{Rk}_j(\overline{s}) \leq \text{Rk}_j(\overline{u})$ for all $j \in \mathbb{N}_0$.

The sequences $\{Y_i(s)^{-m} : i \in \mathbb{N}_0\}$, for $s \in [0,T]$, and $\{\sup_{0 \leq t \leq T} Y_i(t)^{-m} : i \in \mathbb{N}_0\}$ both satisfy property (P), under $\Pbd^{\gamma}_a$, for all $a\in (0,\infty)$ and $\gamma \ge 0$. In particular, for the second sequence, this follows by the Borel-Cantelli Lemma, since,
$\inf \{\sup_{0 \leq t \leq T} Y_i(t)^{-m} : i \in \mathbb{N}_0\} \geq 0$, and
by \eqref{eq:unordered}, for any $d_1 >  0$,
$$
\sum_{i = 1}^\infty \Pbd_a^\gamma\left( \sup_{0 \leq t \leq T} Y_i^\veps(t)^{-m} \geq d_1 \right) \leq d_1^{-1} \sum_{i = 1}^\infty \Ebd_a^\gamma\left[ \sup_{0 \leq t \leq T} Y_i^\veps(t)^{-m} \right] < \infty. 
$$
The verification for the first sequence is done similarly.

For any $i, j \in \mathbb{N}_0$ and $s \in [0,T]$, 
$$
Y_{(i + j)}(s)^{-m} \leq \text{Rk}_i\left( Y_{k + j}(s)^{-m} : k \in \mathbb{N}_0 \right) \leq \text{Rk}_i\left( \sup_{0 \leq t \leq T}Y_{k + j}(t)^{-m} : k \in \mathbb{N}_0 \right).
$$
Taking the supremum over $s \in [0,T]$ and summing over $i \in \mathbb{N}_0$, we obtain that 
\begin{multline*}
\sum_{i = j}^\infty \sup_{0 \leq t \leq T}Y_{(i)}(t)^{-m}  = \sum_{i = 0}^\infty \sup_{0 \leq t \leq T} Y_{(i + j)}(t)^{-m} 
 \leq \sum_{i = 0}^\infty \text{Rk}_i\left( \sup_{0 \leq t \leq T}Y_{k + j}(t)^{-m} : k \in \mathbb{N}_0 \right) \\
 = \sum_{i = 0}^\infty \sup_{0 \leq t \leq T}Y_{i + j}(t)^{-m} = \sum_{i = j}^\infty \sup_{0 \leq t \leq T}Y_i(t)^{-m}.
\end{multline*}
The bound \eqref{eq:ordered} follows by taking the $L^q(\Pbd_a^\gamma)$-norm of both sides and using the bound \eqref{eq:unordered} already proved.
\end{proof}

 Consider the collection
\begin{equation}
\cla_T \doteq \{\psi \in C^2( [0,T]\times (0, \infty)): |\psi|_{\cla_T,m} <\infty \mbox{ for all } m \in \NN\},
\end{equation}
where
$$|\psi|_{\cla_T,m} \doteq \sup_{t \in [0,T]} (|\partial_t \psi(t, \cdot)|_{\cla,m}
+ |\partial_x \psi(t, \cdot)|_{\cla,m} + |\partial_{xx}\psi(t, \cdot)|_{\cla,m}
+ |\psi(t, \cdot)|_{\cla,m}),
$$
and for $\phi \in C((0, \infty):\RR) $ and $m \in \NN$
$$|\phi|_{\cla,m} \doteq \sup_{y \in (0,\infty)} 
|\phi(y)|(1+y)^{m} .$$ 
A related collection of test functions, with exponential instead of polynomial weights, was considered in  \cite{dembo2017equilibrium}. \ab{The reason our setting requires a broader class of test functions will be seen later in the section when we introduce a test function $\psi^{x,\delta}$ 
for our analysis. This  test function is natural for the problem, in that it solves the PDE \eqref{eq:pdefor}. In Lemma \ref{lem:psiinat} we show that the test function is in $\cla_T$, however the right-tail of $\psi^{x,\delta}$ decays at a slower than exponential rate and so it does not belong to the test function family considered in \cite{dembo2017equilibrium}.}

An immediate consequence of Lemma \ref{lem:tailbd} is the following result.
 \begin{corollary} \label{cor:uniform}
Let $\phi \in \cla_T$.
Then, for every $a>0$, $\gamma \ge 0$, $m \in \NN_0$,
$$\sum_{i=0}^k (Y^{\veps}_i(s))^m \phi(t,Y^{\veps}_i(s)) \to \sum_{i=0}^{\infty} (Y^{\veps}_i(s))^m \phi(t,Y^{\veps}_i(s)), \mbox{ uniformly for } s,t \in [0,T], \;\Pbd^{\gamma}_a \mbox{ a.s.}$$
\end{corollary}

\begin{proof}
Fix $r > 1$. By definition of the class $\mathcal{A}_T$, there exist a constant $c_1$ such that, for all $(t,y) \in [0,T] \times (0,\infty)$,
$
y^m|\phi(t,y)| \leq c_1 y^m(1 + y)^{-m - r} \le c_1 y^{-r}.
$
The result now follows  from Lemma \ref{lem:tailbd} (equation \eqref{eq:ordered}).
\end{proof}

For $k \in \NN$, $t\ge 0$, and a $\phi:(0,\infty) \to \RR$, define
$\langle Q^{\veps,k}_t, \phi\rangle \doteq \sum_{i=0}^k \phi(Y_i^{\veps}(t)).$

The following lemma is immediate from \eqref{eq:rgbm} on applying It\^{o}'s formula.
\begin{lemma}\label{lem:ito}
Let $\psi \in \cla_T$. Then, for every $k \in \NN$, $t \in [0,T]$,
$\Pbd^{\gamma}_a$ a.s.
\begin{multline*}
\langle Q^{\veps,k}_t, \psi(t, \cdot)\rangle
= \langle Q^{\veps,k}_0, \psi(0, \cdot)\rangle
 + a \int_0^t \sum_{i=0}^k Y^{\veps}_i(s)\one\{Y^{\veps}_i(s)= Y^{\veps}_{(0)}(s)\}
\partial_y \psi(s, Y^{\veps}_i(s)) ds \\
+ \int_0^t \sum_{i=0}^k\left(\partial_s \psi(s, Y^{\veps}_i(s)) +
a^2 Y^{\veps}_i(s) \partial_y \psi(s, Y^{\veps}_i(s)) + \frac{a^2}{2}
(Y^{\veps}_i(s))^2 \partial_{yy}\psi(s, Y^{\veps}_i(s))\right)ds \\
+ M^{\veps, \psi}_k(t),
\end{multline*}
where, for $t\ge 0$,
\begin{equation}\label{eq:mdef}
M^{\veps, \psi}_k(t) = a\sum_{i=0}^k \int_0^t Y^{\veps}_i(s)
\partial_y \psi(s, Y^{\veps}_i(s)) dW_i(s).
\end{equation}
\end{lemma}
The process $M^{\veps, \psi}_k(t)$ is a martingale with respect to the
canonical filtration $\clf_t = \sigma\{\Wbd(s), \Xbd(s): 0 \le s \le t\}$.
The following lemma shows that this martingale converges uniformly, as $k\to \infty$, in $L^q(\Pbd^{\gamma}_a)$.
\begin{lemma}\label{lem:mart}
Let $\psi \in \cla_T$. 
Then there is a continuous $\clf_t$-martingale $M^{\veps, \psi}_{\infty}(t)$ such that
$\Ebd^{\gamma}_a |M^{\veps, \psi}_{\infty}(T)|^q <\infty$ for all $q \in \NN$ and,
as $k\to \infty$,
$$\Ebd^{\gamma}_a\sup_{0\le t \le T} |M^{\veps, \psi}_k(t) - M^{\veps, \psi}_{\infty}(t)|^q \to 0.$$
\end{lemma}

\begin{proof}
It suffices to show that $D^\veps_{k,k'}(t) \doteq M_{k}^{\veps,\psi}(t) - M_{k'}^{\veps,\psi}(t)$,  converges to $0$ in $L^q(C([0,T],\mathbb{R}), \Pbd_a^\gamma)$ as $k, k' \to \infty$. This is
immediate from Burkholder-Davis-Gundy (BDG) inequality, on observing that, for $y>0$, $[y
\partial_y \psi(s, y)]^2 \le c_1y^{-4}$ and applying 
Lemma \ref{lem:tailbd} (estimate \eqref{eq:unordered}). 
We omit the details.
\end{proof}

For a $\phi \in \cla_T$, we define, for $s,t\in  [0,T]$,
$\langle Q^{\veps}_s, \phi(t, \cdot)\rangle \doteq \sum_{i=0}^{\infty}  \phi(t,Y^{\veps}_i(s)),$
where the series converges $\Pbd^{\gamma}_a$ a.s. in view of Corollary \ref{cor:uniform}.
Occasionally, for clarity, we will write the left side as $\langle Q^{\veps}_s, \phi(t, y)\rangle$.

Note that, for $t\in [0,T]$ and  $\psi \in \cla_T$,
\begin{multline*}
\int_0^t \sum_{i=0}^k Y^{\veps}_i(s)\one\{Y^{\veps}_i(s)= Y^{\veps}_{(0)}(s)\}
\partial_y \psi(s, Y^{\veps}_i(s)) ds\\
= \int_0^t  Y^{\veps}_{(0)}(s)\partial_y \psi(s, Y^{\veps}_{(0)}(s))\sum_{i=0}^k\one\{Y^{\veps}_i(s)= Y^{\veps}_{(0)}(s)\}
 ds
 \to \int_0^t  Y^{\veps}_{(0)}(s)\partial_y \psi(s, Y^{\veps}_{(0)}(s))
 ds,
\end{multline*}
a.s.
Using the above observations, Corollary \ref{cor:uniform}, and  Lemmas \ref{lem:mart} and \ref{lem:ito},
we now have the following result.
\begin{lemma}\label{lem:limito}
Let $\psi \in \cla_T$. Then, for every  $t \in [0,T]$,
$\Pbd^{\gamma}_a$ a.s.
\begin{multline*}
\langle Q^{\veps}_t, \psi(t, \cdot)\rangle
= \langle Q^{\veps}_0, \psi(0, \cdot)\rangle
 + a \int_0^t Y^{\veps}_{(0)}(s)\partial_y \psi(s, Y^{\veps}_{(0)}(s)) ds \\
+ \int_0^t \langle Q^{\veps}_s , \partial_s \psi(s, y)
+ a^2 y \partial_y \psi(s, y)
 + \frac{a^2}{2}
y^2 \partial_{yy}\psi(s, y)\rangle ds
+ M^{\veps, \psi}_{\infty}(t).
\end{multline*}
\end{lemma}

Define for $t>0$ and $x,y \in (0, \infty)$
$$\hat q_t(x,y) \doteq \frac{1}{ax} p_t\left(\frac{\log x}{a} - \frac{\log y}{a} - \frac{at}{2}\right).$$
Using properties of the heat kernel it is easily seen that $y \mapsto \hat q_t(x,y)$ is the fundamental solution of the PDE 
\begin{equation}
\partial_t \varphi =  a^2 y \partial_y \varphi + \frac{a^2y^2}{2}\partial_{yy} \varphi.
\end{equation}
In order to prove Theorem \ref{thm:main1} it would be useful to analyze the time evolution of the random field $\check \clv^{\veps}(t,x)$ in \eqref{eq:924}.
For this, in view of Lemma \ref{lem:limito}, it will be convenient to replace 
the indicator function $\mathbf{1}{[0,x]}$ with a suitable mollification. It is not surprising that the mollification obtained by convolving with the kernel $\hat q_t$ introduced below is particularly tractable.

For $x \in (0, \infty)$ and $t>0$, define
$$
\Psi^x(t,y) \doteq \int_{0}^{x} \hat q_t(z,y) dz, \; y \in (0,\infty).
$$

The following elementary lemma notes the fact that, for $t>0$, the evolution kernel $\hat{q}_t(x,y)$ is a probability density in $x$ and in $y$.

\begin{lemma} \label{lem:areprobs}
For any $t > 0$, the following equalities hold:
\begin{equation} \label{eq:hatqisprob}
\int_0^\infty \hat{q}_t(z,y)dy = \int_0^\infty \hat{q}_t(z,y)dz = 1.
\end{equation}
\end{lemma}

\begin{proof}
Note that if $\Phi_t(u)$, $u \in \mathbb{R}$, is the cdf for $\mathcal{N}(0,t)$ distribution, then $z \mapsto \Phi_t(\frac{\log z}{a} - \frac{\log y}{a} - \frac{at}{2})$ is the cdf of a probability distribution on $(0,\infty)$ for each fixed $y>0$ and $t>0$, and its derivative is $\hat{q}_t(z,y)$. This shows $z \mapsto \hat{q}_t(z,y)$ is a probability density. To see that $y \mapsto \hat{q}_t(z,y)$ is a probability density, we calculate directly, 
\begin{align}
\int_0^\infty \hat{q}_t(z,y)dy = \int_{-\infty}^\infty \frac{1}{az}p_t(w)az e^{-az - \frac{a^2t}{2}}dw = \int_{-\infty}^\infty p_t(w + at)dw = 1.
\end{align}
where we have made the change of variables $w = \frac{\log z}{a} - \frac{\log y}{a} - \frac{at}{2}$. 
\end{proof}

Note that, by Fubini's theorem, the above lemma implies a frequently useful conservation of mass property: For any $f \in L^1(0,\infty)$, 
\begin{equation} \label{eq:cofmass}
\int_{(0,\infty) \times (0,\infty)} f(y)\hat{q}_t(z,y)dy dz =  \int_0^\infty f(z)dz= \int_{(0,\infty) \times (0,\infty)} f(z)\hat{q}_t(z,y)dz dy.
\end{equation}
Since $q_t(z,y) = \hat{q}_t(z,ye^{-at})$, it follows directly from the lemma that $z \mapsto q_t(z,y)$ is also a probability density. On the other hand, $y \mapsto q_t(z,y)$ is not, and a similar calculation to the one in the proof above shows that $\int_0^\infty {q}_t(z,y)dy = e^{a^2t}$.

We will need the following $L^q$-bounds for the derivatives of $\Psi^x$.

\begin{lemma} \label{lem:derpsi}
Fix $T \in (0,\infty)$, $0 < \ell < L <\infty$ and $q \in [1,\infty)$. Then, there
is a $C=C(q, \ell, L,T) \in (0,\infty)$ such that, for all $t \in [0,T]$, $x \in [\ell,L]$, and $i= 1,2,3,4$,
$$\int_0^\infty |g_i(t,y)|^q dy \le C t ^{-m_i(q)},$$
where
\begin{align*}
g_1(t,y) = y \partial_y \Psi^x(t,y), \; m_1(q) =(q-1)/2,& \;
g_2(t,y) =  \partial_t \Psi^x(t,y), \; m_2(q) = q-1/2, \\ 
g_3(t,y) = y \partial_t \partial_y \Psi^x(t,y), \; m_3(q) =(3q-1)/2, & \; g_4(t,y) = y \partial_x \partial_y \Psi^x(t,y), \; m_4(q) = q-1/2.
\end{align*}
\end{lemma}

\begin{proof}
Consider $i=1$. By the change of variables $u = \frac{\log z}{a} - \frac{\log y}{a} - \frac{at}{2}$, and with $p_t$ as in \eqref{eq:heatkernel}, we may write 
\begin{equation}
\Psi^x(t,y) = \int_0^x (az)^{-1}p_t\left( \frac{\log z}{a} - \frac{\log y}{a} - \frac{at}{2} \right)dz = \int_{-\infty}^{\frac{\log x}{a} - \frac{\log y}{a} - \frac{at}{2}} p_t(u)du. \label{eq:Psixexp}
\end{equation}
Differentiating this expression with respect to $y$, we obtain 
\begin{equation}
y \partial_y \Psi^x(t,y) = -\frac{1}{a}p_t\left( \frac{\log x}{a} - \frac{\log y}{a} - \frac{at}{2} \right) \label{eq:ydPsi}
\end{equation}
Letting $w = \frac{\log x}{a} - \frac{\log y}{a} - \frac{at}{2}$, one may show
\begin{equation}
\int_0^\infty |y \partial_y \Psi^x(t,y)|^q dy = c_1 ax\int_{-\infty}^\infty p_t(w)^{q-1} p_t(w + at) dw \leq c_2t^{-(\frac{q-1}{2})}.
\end{equation}

Consider now $i=2$. Another direct calculation (note that $p_t(w)$ solves the heat equation) yields
\begin{equation}
\partial_t \Psi^x(t,y) = \left( -\frac{a}{2} - t^{-1}\left( \frac{\log x}{a} - \frac{\log y}{a} - \frac{at}{2} \right) \right)p_t\left( \frac{\log x}{a} - \frac{\log y}{a} - \frac{at}{2} \right). \label{eq:dPsixt}
\end{equation}
We obtain by the same change of variables as before
\begin{align}
\begin{split}
\int_0^\infty |\partial_t \Psi^x(t,y)|^{q} dy & = ax \int_{-\infty}^\infty \left| -\frac{a}{2} - t^{-1}w \right|^{q} p_t(w)^{q - 1}p_t(w + at) dw \\
& \leq c_3 t^{-(\frac{q-1}{2})} \int_{-\infty}^\infty \left| \frac{a}{2} - t^{-1}w \right|^{q} p_t(w) dw \leq c_4 t^{-q + \frac{1}{2}}.
\end{split}
\end{align}

For $i=3$, differentiating \eqref{eq:ydPsi} with respect to $t$, one may show that
\begin{align}
& y \partial_t \partial_y \Psi^x(t,y) \\
& = \left(\frac{a^{-1}t^{-1}}{2} - \frac{t^{-1}}{2}\left( \frac{\log x}{a} - \frac{\log y}{a} - \frac{at}{2} \right) - \frac{a^{-1}t^{-2}}{2}\left( \frac{\log x}{a} - \frac{\log y}{a} - \frac{at}{2} \right)^2 \right) \\
& \quad\quad \times p_t\left( \frac{\log x}{a} - \frac{\log y}{a} - \frac{at}{2} \right).
\end{align}
Hence,  
\begin{align}
\int_0^\infty |y \partial_t \partial_y \Psi^x(t,y)|^q dy
& = \int_{-\infty}^\infty \left| \frac{a^{-1}t^{-1}}{2} - \frac{t^{-1}w}{2} - \frac{a^{-1}t^{-2}w^2}{2} \right|^q p_t(w)^{q - 1} p_t(w + at) axdw \\
& \leq c_5 t^{-(\frac{q-1}{2})} \int_{-\infty}^\infty t^{-q}(1+ |w|^q + t^{-q} |w|^{2q}) p_t(w) dw \leq c_6 t^{-(\frac{3q - 1}{2})}.
\end{align}

Finally for $i=4$, differentiating \eqref{eq:ydPsi} with respect to $x$ yields
\begin{align}
y \partial_x \partial_y \Psi^x(t,y) = \frac{1}{a^2x t} \left( \frac{\log x}{a} - \frac{\log y}{a} - \frac{at}{2} \right)p_t\left( \frac{\log x}{a} - \frac{\log y}{a} - \frac{at}{2} \right).
\end{align}
Therefore,
\begin{align}
\int_0^\infty |y \partial_x \partial_y \Psi^x(t,y)|^q dy & = \int_{-\infty}^\infty (\frac{1}{a^2 xt})^q |w|^q ax p_t(w)^{q-1} p_t(w + at) dw \\
& \leq c_7 t^{-q} \int_{-\infty}^\infty |w|^q p_t(w)^{q-1} p_t(w + at) dw \leq c_8 t^{-q + 1/2},
\end{align}
where the next to last inequality uses the fact that $x \in [\ell, L]$. This concludes the proof.
\end{proof}


The behavior of the function $\Psi^x(t,y)$ is hard to control when $t$ approaches $0$, \ab{since this quantity is $O(t^{-1/2})$ as $t \to 0$}. To handle this we consider the following $\delta$-perturbation.
Fix $x \in (0, \infty)$, $\delta \in (0, \infty)$,
$t \in (0,\infty)$, and define
$$\phi_t^{x,\delta}(s,y)  \doteq \psi^{x,\delta}(t-s,y), \; 0 \le s \le t, \; y \in (0,\infty),$$
where
$$\psi^{x,\delta}(s,y) \doteq \Psi^x(s+\delta,y), \; 0 \le s \le t, \; y \in (0,\infty).$$
Note that 
$\psi^{x,\delta}$ solves
\begin{equation}\label{eq:pdefor}
\begin{aligned}
\partial_t \psi^{x,\delta}(t,y) &= a^2y\partial_y \psi^{x,\delta}(t,y) +
\frac{a^2y^2}{2} \partial_{yy}\psi^{x,\delta}(t,y),\; (t,y) \in (0, \infty)\times (0, \infty),\\
\psi^{x,\delta}(0,\cdot) &= \Psi^x(\delta,\cdot).
\end{aligned}
\end{equation}

\begin{lemma} \label{lem:psiinat}
For  every $x \in (0, \infty)$, $T >0$, and $\delta > 0$, $\psi^{x,\delta} \in \cla_T$.
\end{lemma}

\begin{proof}
Fix $m \in \mathbb{N}$ and $x,T,\delta$ as in the lemma. The constants $c_1,c_2,c_3,$ etc. in this proof are allowed to depend on $a, x, T$, $\delta$, and $m$. We will refer a number of times to the calculations done in the proof of Lemma \ref{lem:derpsi}. Recall that $\psi^{x,\delta}(s,y) = \Psi^x(s + \delta, y)$. 
By \eqref{eq:Psixexp} and standard Gaussian tail bounds
\begin{multline} 
 (1 + y)^m|\psi^{x,\delta}(s,y)| 
 \leq c_1(1 + y)^m\left(\mathbf{1}[0,c_2](y) + \mathbf{1}[c_2,\infty)(y)\exp\left( -c_3(\log y)^2\right) \right) \leq c_4.
  \label{eq:inatBd1}
\end{multline}
Also, by \eqref{eq:ydPsi}, 
\begin{align}
(1 + y)^m|\partial_y \psi^{x,\delta}(s,y)| \leq c_5 \frac{(1 + y)^m}{y(s + \delta)^{1/2}}\exp\left(-c_6(\log y)^2\right) \leq c_7. \label{eq:inatBd2}
\end{align}
Differentiating \eqref{eq:ydPsi}, we have 
\begin{align*}
\partial_{yy} \psi^{x,\delta}(s,y) &= -\frac{1}{a^2 y^2(s + \delta)} \left( \frac{\log x}{a} - \frac{\log y}{a} - \frac{3a(s + \delta)}{2} \right)\\
&\quad \times  p_{s + \delta}\left( \frac{\log x}{a} - \frac{\log y}{a} - \frac{a(s + \delta)}{2} \right).
\end{align*}
Therefore, using the bound,
$$
\left|\frac{\log x}{a} - \frac{\log y}{a} - \frac{3a(s + \delta)}{2} \right| \leq c_8(1 + |\log y|), \; s \in [0,T],\; y \in (0,\infty),
$$
we have
\begin{align}
(1 + y)^m|\partial_{yy} \psi^{x,\delta}(s,y)| \leq \frac{c_{9}(1 + y)^m}{y^2}  (1 + |\log y|) \exp\left(-c_{10}(\log y)^2 \right) \leq c_{11}. \label{eq:inatBd3}
\end{align}
Lastly, \eqref{eq:dPsixt} gives us 
\begin{align}
|\partial_t \Psi^x(t,y)| \leq c_{12}(1 + |\log y|) \exp\left(-c_{13}(\log y)^2\right) \leq c_{14}. \label{eq:inatBd4}
\end{align}
The bounds \eqref{eq:inatBd1}, \eqref{eq:inatBd2}, \eqref{eq:inatBd3}, and \eqref{eq:inatBd4} imply the result.
\end{proof}

It follows from Lemma \ref{lem:psiinat} that for each $x \in (0, \infty)$, $t > 0$,
and $\delta>0$, $\phi_t^{x,\delta} \in \cla_t$.
Thus applying Lemma \ref{lem:limito} with $T$ replaced by $t$ and
$\psi$ replaced with $\phi_t^{x,\delta}$, and using \eqref{eq:pdefor}, we have, for $t\ge 0$,
\begin{multline}
\langle Q^{\veps}_t, \phi_t^{x,\delta}(t, \cdot)\rangle
 = \langle Q^{\veps}_0, \phi_t^{x,\delta}(0, \cdot)\rangle
 + a \int_0^t Y^{\veps}_{(0)}(s)\partial_y \phi_t^{x,\delta}(s, Y^{\veps}_{(0)}(s)) ds + M^{\veps, \phi_t^{x,\delta}}_{\infty}(t).
 \label{eq:qvepst}
\end{multline}

Consider now $\delta=\veps$.
From \eqref{eq:cofmass}, we have the equality
\begin{equation}\label{eq:340}
\int^{\infty}_0 \Psi^x(\veps, y) dy = x = \int^{\infty}_0 \one_{(0, x]}(y) dy,
\end{equation}
from which it follows that
\begin{equation}\label{eq:firstreq}
\sup_{\ell \le x \le L} \left |\veps^{-1/4} \left(\int^{\infty}_0 \Psi^x(\veps, y) dy
- \int^{\infty}_0 \one_{(0, x]}(y) dy\right)\right| = 0.
\end{equation}
The following lemma, which controls the error on replacing the indicator
$\mathbf{1}[0,x]$ with a mollified approximation, will be proved in  Section \ref{sec2.1}.
\begin{lemma}
\label{lem:secreq}
For all $0<\ell <L <\infty$ and $T\in (0, \infty)$,
$$\sup_{\ell \le x \le L} \sup_{0\le t \le T}
\left| \veps^{1/4} \sum_{i=0}^{\infty}
\left(\Psi^x(\veps, Y^{\veps}_i(t)) - \one_{(0,x]}(Y^{\veps}_i(t))\right)\right| \to 0,
$$
in probability, as $\veps \to 0$.
\end{lemma}

Note from \eqref{eq:340}, that
\begin{align*}
\veps^{1/4}
\left(\langle Q^{\veps}_t, \phi_t^{x,\veps}(t, \cdot)\rangle
- \veps^{-1/2} x\right) &= 
\veps^{1/4}\left(\langle Q^{\veps}_t, \phi_t^{x,\veps}(t, \cdot)\rangle
- \veps^{-1/2}\int_0^{\infty}\Psi^x(\veps, y) dy\right).
\end{align*}
Letting
\begin{equation}\hat \clv^{\veps}(t,x)
\doteq \veps^{1/4}\left(\langle Q^{\veps}_t, \phi_t^{x,\veps}(t, \cdot)\rangle
- \veps^{-1/2}\int_0^{\infty}\Psi^x(\veps, y) dy\right) 
\label{eq:hatvdef}
\end{equation}
and
$$\hat \clw^{\veps}(t,x)
\doteq \veps^{1/4}\left(\langle Q^{\veps}_0, \phi_t^{x,\veps}(0, \cdot)\rangle
- \veps^{-1/2}\int_0^{\infty}\Psi^x(\veps, y) dy\right),$$
we have from \eqref{eq:qvepst},
\begin{equation} \label{eq:VWerrM}
\hat \clv^{\veps}(t,x) = \hat \clw^{\veps}(t,x)
+  a \veps^{1/4} \int_0^t Y^{\veps}_{(0)}(s)\partial_y \phi_t^{x,\veps}(s, Y^{\veps}_{(0)}(s)) ds + \veps^{1/4} M^{\veps, \phi_t^{x,\veps}}_{\infty}(t).
\end{equation}
Furthermore, with 
\begin{equation}\label{eq:924}
\check \clv^{\veps}(t,x) \doteq \veps^{1/4}
\left(\sum_{i=0}^{\infty}
 \one_{(0,x]}(Y^{\veps}_i(t)) - \veps^{-1/2}\int^{\infty}_0 \one_{(0, x]}(y) dy\right), (t,x) \in [0, \infty)\times (0, \infty),
\end{equation}
we have, on noting that $\phi^{x,\veps}_t(t, \cdot) = \Psi^x(\veps,\cdot)$, and using \eqref{eq:firstreq}, \eqref{eq:hatvdef}, and Lemma \ref{lem:secreq}, that
\begin{equation}\label{eq:almostsame}
\hat \clv^{\veps} - \check \clv^{\veps} \to 0 \mbox{ in probability, uniformly for compact sets in } [0,\infty)\times (0,\infty), \mbox{ as } \veps \to 0.
\end{equation}

The following lemma shows that the contribution from the lowest particle $Y^{\veps}_{(0)}(\cdot)$ in \eqref{eq:qvepst} vanishes asymptotically as $\veps \rightarrow 0$.
\begin{lemma} \label{lem:Y0term}
For all $0<\ell<L<\infty$, $T\in (0,\infty)$, $a>0$, $\gamma \ge 0$, and $q \in [1,\infty]$,
\begin{equation}
\sup_{\ell\le x \le L} \veps^{1/4} \sup_{0\le t \le T}\int_0^t  Y^{\veps}_{(0)}(s)|\partial_y \phi_t^{x,\veps}(s, Y^{\veps}_{(0)}(s))|
 ds \to 0 \mbox{ in } L^q(\Pbd^{\gamma}_a) \mbox{ as } \veps \rightarrow 0.
\end{equation}
\end{lemma}

\begin{proof}
Note that $\phi_t^{x,\veps}(s, y) = \Psi^x(t - s + \veps, y)$. By \eqref{eq:ydPsi}, for all $0 \leq s \leq t \leq T$, $\veps>0$, and $x,y \in (0,\infty)$,
$$
y |\partial_y \Psi^x(t - s + \veps, y)| = \frac{1}{a}p_{t - s + \veps}\left( \frac{\log x}{a} - \frac{\log y}{a} - \frac{at}{2}\right) \leq c_1(t - s + \veps)^{-1/2}.
$$
Therefore,  
\begin{align}
\sup_{\ell\le x \le L} \veps^{1/4} \sup_{0\le t \le T}\int_0^t  Y^{\veps}_{(0)}(s)|\partial_y \phi_t^{x,\veps}(s, Y^{\veps}_{(0)}(s))|
 ds &\leq c_1 \veps^{1/4} \sup_{0 \leq t \leq T} \int_0^t (t - s + \veps)^{-1/2} ds\\
 &\leq 2c_1\veps^{1/4}(T + \veps)^{1/2},
\end{align}
which implies the result.
\end{proof}

Let $\hat \clm^{\veps}(t,x) \doteq \veps^{1/4} M^{\veps, \phi_t^{x,\veps}}_{\infty}(t)$. 
The proof of the following key tightness result is given in Section \ref{sec2.1}.
\begin{lemma}
\label{lem:tight}
The collection
$\{(\hat \clw^{\veps}, \hat \clm^{\veps}, \hat \clv^{\veps}),\; \veps>0\}$
is tight in $[C([0,\infty)\times(0,\infty):\RR)]^3$.
\end{lemma}

Theorem \ref{thm:main1} will be an immediate consequence of the following result.
The proof is given in Section \ref{ssec:main3prf}.
\begin{theorem}\label{thm:formain1}
Let $B$ and $W$ be   mutually independent white noise measures on $(0,\infty)$ 
and $[0,\infty)\times (0,\infty)$ given on some probability space $(\Om, \clf, \Pbd)$.
Define, for $(t,x) \in (0, \infty)\times (0,\infty)$ 
\begin{align*}
\hat \clw(t,x) \doteq \int_0^{\infty} \Psi^x(t,y) B(dy), \;\;
\hat \clm(t,x) &\doteq \int_{(0,t)\times (0,\infty)} ay q_{t-s}(y,x) W(ds\, dy)
\end{align*}

Also set
$\hat \clw(0,x) \doteq B((0,x])$, $\hat \clm(0,x) =0$, $x \in (0,\infty)$.
Let
$
\hat \clv(t,x) \doteq \hat \clw(t,x) + \hat \clm(t,x)$, $(t,x) \in [0,\infty)\times (0,\infty)$.

Then, as $\veps \to 0$, in $[C([0,\infty)\times(0,\infty):\RR)]^3$,
$
(\hat \clw^{\veps}, \hat \clm^{\veps}, \hat \clv^{\veps}) \Rightarrow (\hat \clw, \hat \clm, \hat \clv).
$
\end{theorem}

The following proposition,
which allows us to go from particle counts in a given interval to the position of a typical ranked particle in the bulk,
will be needed in the proof of Theorem
\ref{thm:main2}. Note that $\chi^{\veps}$ in the proposition is exactly the same as
$\check \clv^{\veps}$ in \eqref{eq:924}.
For $(t,x) \in [0, \infty) \times (0, \infty)$, define
$$I^{\veps}_t(x) \doteq \inf\{i \in \NN_0: Y_{(i)}(t) > \veps^{-1/2}x\}
= \sum_{i=0}^{\infty} \one_{(0, \veps^{-1/2}x]} (Y_i(t)).$$
\begin{proposition}
\label{prop:yaxis}
For $(t, x) \in [0, \infty) \times (0, \infty)$, define
$\check \chi^\veps(t,x) \doteq \veps^{1/4} (i_{\veps}(x) - Y_{(i_{\veps}(x))}(t)),$
where $i_{\veps}(x)$ is as introduced in \eqref{eq:defnhatxin},
$$\tilde \chi^{\veps}(t,x) \doteq \veps^{1/4} (I^{\veps}_0(x) - Y_{(I^{\veps}_0(x))}(t)),
\mbox{ and }
\chi^{\veps}(t,x) \doteq \veps^{1/4}(I^{\veps}_t(x) - x\veps^{-1/2}).$$
Then for all $T \in (0, \infty)$ and $0< \ell<L<\infty$
\begin{align}
\sup_{t \in [0,T]} \sup_{x \in [\ell, L]}
|\chi^{\veps}(t,x) - \tilde\chi^{\veps}(t,x)| \to 0 \mbox{ in  $\Pbd^{\gamma}_a$ probability, as } \veps\to 0. \label{eq:show1}\\
\sup_{t \in [0,T]} \sup_{x \in [\ell, L]}
|\tilde \chi^{\veps}(t,x) - \check\chi^{\veps}(t,x)| \to 0 \mbox{ in  $\Pbd^{\gamma}_a$ probability, as } \veps\to 0. \label{eq:show2}
\end{align}
\end{proposition}
 Proof of the proposition is in Section \ref{sec:proofprop}.

\subsection{Proofs of Theorems \ref{thm:main1} and \ref{thm:main2}}
\label{sec:mainthmpfs}
We can now complete the proof of Theorems \ref{thm:main1} and \ref{thm:main2}.\\ \ \\

\noi {\bf Proof of Theorem \ref{thm:main1}.}
We note that, for $t>0$ and $x, y \in (0,\infty)$,
\begin{align*}
\partial_y \Psi^x(t,y) &= \partial_y \int_0^x \hat q_t(z,y) dz
= -\frac{1}{ay} p_t\left( \frac{\log y}{a}, \frac{\log x}{a} - \frac{at}{2}\right) = - q_t(y,x).
 \end{align*}
 Integrating by parts, this shows that
 $$\hat \clw(t,x) = \int_0^{\infty} \Psi^x(t,y) B(dy) = \int_0^{\infty} B((0,y]) q_t(y,x) dy.$$
 Thus, for $(t,x) \in [0,\infty)\times (0,\infty)$, 
\begin{equation}\label{eq:832}
u(t,x) = \hat \clv(t,x) = \int_0^{\infty} \Psi^x(t,y) B(dy) +
\int_{(0,t)\times (0,\infty)} ay q_{t-s}(y,x) W(ds\, dy).
\end{equation}
 where $u$ is given as a solution of the SPDE in \eqref{eq:spde} with $u_0(x) = B((0,x])$. Taking $x = e^{az}$ we see that $\check \clv^{\veps}(t,x) = \hat R^{\veps}(t, z)$. The result now follows on combining Theorem \ref{thm:formain1} with \eqref{eq:almostsame}.
\hfill \qed

We now proceed to the proof of Theorem \ref{thm:main2}.
\\ \ \\

\noindent {\bf Proof of Theorem \ref{thm:main2}.}
We recall  that $\chi^{\veps}$ is the same as $\check \clv^{\veps}$ defined in \eqref{eq:924}.
From Proposition \ref{prop:yaxis} and
Theorem \ref{thm:main1}, under $\Pbd_a^{\gamma}$
$\check \chi^\veps$ converges in distribution,  uniformly on compact subsets of $[0, \infty) \times (0,\infty)$, to  $u$, where
$u$ is the solution  of the SPDE in \eqref{eq:spde} with $u_0(x) = B((0,x])$ and $B, W$ are as in Theorem \ref{thm:formain1}.
Thus we have, under $\Pbd_a^{\gamma}$,
\begin{equation}\label{eq:148}
\veps^{-1/4} \left( \frac{Y_{(i^{\veps}(x))}(t)}{\veps^{-1/2}} -x\right)\to u(t,x),
\end{equation}
in distribution,  and $\frac{Y_{(i^{\veps}(x))}(t)}{\veps^{-1/2}} \to x$ in probability,
uniformly on compact subsets of $[0, \infty) \times (0,\infty)$.

Fix $T \in (0, \infty)$ and $0<\ell<L<\infty$.
As an immediate consequence of the above uniform convergence in probability it follows that
\begin{equation}
\limsup_{\eta \to 0}\limsup_{\veps\to 0} 
\Pbd^{\gamma}_a\left(\inf_{t \in [0,T]}\inf_{x \in [\ell, L]} \frac{Y_{(i^{\veps}(x))}(t)}{\veps^{-1/2}} < \eta\right) = 0.
\end{equation}
For $y_1,y_2 \in (0,\infty)$ we can write
$$\log y_1 - \log y_2 = (y_1-y_2) \frac{1}{y_1} 
+ (y_1-y_2)^2 r(y_1, y_2),$$
where
$r:(0, \infty) \times (0,\infty) \to \RR$
is a measurable map with
$|r(y_1, y_2)| \le (y_1 \wedge y_2)^{-2}$.
Thus we can write
\begin{align}
\veps^{-1/4} \left( \log \left(\frac{Y_{(i^\veps(x))}(t)}{\veps^{-1/2}}\right) -\log x\right)
&= \veps^{-1/4} \left(\frac{Y_{(i^\veps(x))}(t)}{\veps^{-1/2}} - x\right)\frac{1}{x} \nonumber\\
&\quad
+ \veps^{-1/4}\left(\frac{Y_{(i^\veps(x))}(t)}{\veps^{-1/2}} - x\right)^2
r(\frac{Y_{(i^\veps(x))}(t)}{\veps^{-1/2}}, x).\label{eq:344}
\end{align}
Also, for any $\alpha>0$,
\begin{align*}
&\Pbd^{\gamma}_a\left(\sup_{t \in [0,T]}\sup_{x \in [\ell, L]} \veps^{-1/4}\left(\frac{Y_{(i^{\veps}(x))}(t)}{\veps^{-1/2}} - x\right)^2
|r(\frac{Y_{(i^{\veps}(x))}(t)}{\veps^{-1/2}}, x)| > \alpha\right)\\
&\le \Pbd^{\gamma}_a\left(\inf_{t \in [0,T]}\inf_{x \in [\ell, L]} \frac{Y_{(i^\veps(x))}(t)}{\veps^{-1/2}} < \eta\right)\\
&\quad + \Pbd^{\gamma}_a\left(\sup_{t \in [0,T]}\sup_{x \in [\ell, L]} \veps^{-1/4}\left(\frac{Y_{(i^{\veps}(x))}(t)}{\veps^{-1/2}} - x\right)^2 (\eta^{-2} + \ell^{-2})
 > \alpha\right).
\end{align*}
Taking limit as $\veps\to 0$ and using \eqref{eq:148}
we have
\begin{align}
&\limsup_{\veps\to 0}\Pbd^{\gamma}_a\left(\sup_{t \in [0,T]}\sup_{x \in [\ell, L]} \veps^{-1/4}\left(\frac{Y_{(i^{\veps}(x))}}{\veps^{-1/2}} - x\right)^2
|r(\frac{Y_{(i^{\veps}(x))}}{\veps^{-1/2}}, x)| > \alpha\right)\nonumber\\
&\le \limsup_{\eta \to 0}\limsup_{\veps\to 0} 
\Pbd^{\gamma}_a\left(\inf_{t \in [0,T]}\inf_{x \in [\ell, L]} \frac{Y_{(i^{\veps}(x))}(t)}{\veps^{-1/2}} < \eta\right) = 0.\label{eq:344b}
\end{align}
The result now follows upon observing that
$$
\veps^{-1/4} \left( \log \left(\frac{Y_{(i^\veps(x))}(t)}{\veps^{-1/2}}\right) -\log x\right) = a\veps^{-1/4} \left(X_{(i^\veps(x))}(t) + \frac{at}{2} + \frac{1}{2a}\log \veps - \frac{\log x}{a}\right)
$$
and appealing to \eqref{eq:148}, \eqref{eq:344} and \eqref{eq:344b}. \hfill \qed

\subsection{Proofs of Lemmas \ref{lem:secreq} and \ref{lem:tight}} 
\label{sec2.1}

We will prove Lemma \ref{lem:tight} first and then Lemma \ref{lem:secreq}. \ab{The proofs depend on the classical Kolmogorov-Chentsov criterion for tightness (see  \cite[Theorem 1.4.1]{kun}).}
We will also make use of the following moment bounds for integrals with respect to Poisson processes (cf. \cite{bassanbona1990PRMmoments}). \ab{We omit the straightforward proof.}

\begin{lemma} \label{lem:poissonint}
(i) Let $R(du)$ be a Poisson random measure on $\mathbb{R}$ of intensity $\mu(du)$, where $\mu$ is a $\sigma$-finite measure on $\mathbb{R}$. For all $n \in \mathbb{N}$ and all nonnegative measurable functions $f$ on $\mathbb{R}$, 
$$
\left|\Ebd \left( \int_{\mathbb{R}} f(u)R(du) \right)^n  \right| \leq C_n\max_{1 \leq k \leq n} \left(\int_{\mathbb{R}} f(u)^k \mu(du) \right)^{n/k},
$$
where $C_n$ is a constant depending on $n$.

(ii) Denote the corresponding compensated Poisson process by $\widetilde{R}(du) \doteq R(du) - \mu(du)$. For all $n \in \mathbb{N}$ and all $\mu$-integrable functions $f$ on $\mathbb{R}$,
$$
\Ebd \left| \int_{\mathbb{R}} f(u)\widetilde{R}(du) \right|^{2n}  \leq C_n \max_{1 \leq k \leq n} \left(\int_{\mathbb{R}} |f(u)|^{2k} \mu(du) \right)^{n/k},
$$
where $C(n)$ is a constant depending on $n$.
\end{lemma}



In view of  the Kolmogorov-Chentsov criterion (cf.  \cite[Theorem 1.4.1]{kun}), tightness of the collection $\{(\hat\clw^\veps, \hat\clm^\veps, \hat \clv^\veps), \veps > 0\}$ (Lemma \ref{lem:tight}) will be a consequence of the following two lemmas.

\begin{lemma} \label{lem:tightatpt}
For any $x_0 \in (0,\infty)$, the collection $\{(\hat\clw^\veps(0,x_0), \hat\clm^\veps(0,x_0), \hat\clv^\veps(0,x_0)), \veps > 0\}$ is a tight collection of $\RR^3$ valued random variables.
\end{lemma}

\begin{lemma} \label{lem:contbounds}
For any $q \in \mathbb{N} \cap [2,\infty)$, $T \in (0,\infty)$, $0 < \ell < L < \infty$, there is
a $C=C(q,\ell,L,T) \in (0,\infty)$ such that, for all 
$t, t' \in [0,T]$, $x,x' \in [\ell,L]$, and $\veps > 0$, the following bounds hold:
\begin{longlist}
\item 
$
\|\hat\clm^\veps(t,x') - \hat\clm^\veps(t,x)\|_q \leq C|x' - x|^{1/4}$,
\item 
$
\|\hat\clm^\veps(t',x) - \hat\clm^\veps(t,x)\|_q \leq C|t' - t|^{1/4}\left(1 + \left|\log\left(\frac{1}{t'-t}\right)\right|^{1/2}\right)$, 
\item 
$
\|\hat\clw^\veps(t,x') - \hat\clw^\veps(t,x)\|_q \leq C|x' - x|^{1/8}$
\item 
$
\|\hat\clw^\veps(t',x) - \hat\clw^\veps(t,x)\|_q \leq C|t' - t|^{1/4}$,
\item 
$
\|\hat\clv^\veps(t,x') - \hat\clv^\veps(t,x)\|_q \leq C|x' - x|^{1/8}$, and
\item 
$
\|\hat\clv^\veps(t',x) - \hat\clv^\veps(t,x)\|_q \leq C|t' - t|^{1/4}\left(1 + \left|\log\left(\frac{1}{t'-t}\right)\right|^{1/2}\right)$

\end{longlist}
\end{lemma}

\begin{remark} \normalfont
Since \eqref{eq:VWerrM} holds, the bounds (i)-(iv), together with Lemma \ref{lem:Y0term}, are sufficient to establish tightness of the triple $\{(\hat\clw^\veps, \hat\clm^\veps, \hat \clv^\veps), \veps > 0\}$. However, we will make use of the bounds (v) and (vi) to prove Lemma \ref{lem:secreq}.
\end{remark}

\begin{proof}[Proof of Lemma \ref{lem:tightatpt}]
In view of \eqref{eq:VWerrM}, Lemma \ref{lem:Y0term}, and the fact that $\hat\clm^\veps(0,x_0) = 0$ for all $\veps > 0$, it is sufficient to show that $\{\hat\clw^\veps(0,x_0), \veps > 0\}$ is tight. For any $M > 0$, 
\begin{align}
\Pbd_a^\gamma\left( \left|\hat\clw^\veps(0,x_0)\right| \geq M \right) & = \Pbd_a^\gamma\left( \veps^{1/4}\left| \sum_{i = 0}^\infty \Psi^{x_0}(\veps, Y_i^\veps(0)) - \veps^{-1/2}\int_0^\infty \Psi^{x_0}(\veps, y)dy \right| \geq M \right) \\
& \leq M^{-1}\veps^{1/4}\Ebd_a^\gamma\left| \sum_{i = 0}^\infty \Psi^{x_0}(\veps, Y_i^\veps(0)) - \veps^{-1/2}\int_0^\infty \Psi^{x_0}(\veps, y)dy \right| \\
& = M^{-1}\Ebd_a^0\left[ (Y_{(0)}(0))^{2\gamma/a}\veps^{1/4}\left| \int \Psi^{x_0}(\veps,y)\tilde{N}(\veps^{-1/2}dy) \right| \right].
\end{align}
Here $\tilde{N}(\veps^{-1/2}dy) \doteq N(\veps^{-1/2}dy) - \veps^{-1/2}\mathbf{1}[0,\infty)(y)dy$, where $N(\veps^{-1/2}dy) \doteq \sum_{i = 0}^\infty \delta_{Y_i^\veps(0)}(dy)$. The last line above follows by \eqref{eq:margCOM}. We recall from Lemma \ref{lem:ispp} that, under $\Pbd_a^0$, $N(\veps^{-1/2}dy)$ is distributed as a Poisson random measure of intensity $\veps^{-1/2}\mathbf{1}[0,\infty)(y)dy$, and hence $\tilde{N}(\veps^{-1/2}dy)$ is the corresponding compensated measure. Applying Cauchy-Schwarz inequality to the last line yields the following upper bound 
\begin{align}
& M^{-1}\left[\Ebd_a^0 (Y_{(0)}(0))^{4\gamma/a} \right]^{1/2} \veps^{1/4}\left[\Ebd_a^0 \left| \int \Psi^{x_0}(\veps,y)\tilde{N}(\veps^{-1/2}dy) \right|^2\right]^{1/2} \\
& \leq c_1 M^{-1} \left( \int_0^\infty |\Psi^{x_0}(\veps,y)|^2 dy \right)^{1/2} \leq c_1 M^{-1} \left( \int_0^\infty \Psi^{x_0}(\veps,y) dy \right)^{1/2} \\ 
& = c_1M^{-1} \left( \int_0^{x_0} \int_0^\infty \hat{q}_\veps(z,y) dz dy \right)^{1/2} = c_1M^{-1} x_0^{1/2}.
\end{align}
The first inequality above follows by Lemma \ref{lem:poissonint}(ii) and noting that $\Ebd_a^0\left[ Y_{(0)}(0)^{4\gamma/a} \right] < \infty$ 
on recalling that, $Y_{(0)}(0)$ is distributed as 
$\text{Exp}(1)$ under $\Pbd_a^0$. The second inequality follows because $0 \leq \Psi^{x_0}(\veps,y) \leq 1$ (see Lemma \ref{lem:areprobs}). The first equality is obtained by interchanging order of integration, and the second equality follows again on recalling from Lemma \ref{lem:areprobs} that $\int_0^\infty \hat{q}_\veps(z,y)dz = 1$. Tightness of $\{\hat\clw^\veps(0,x_0), \veps > 0\}$ is an immediate consequence of this bound.
\end{proof}

\begin{proof}[Proof of Lemma \ref{lem:contbounds}(i)] We may assume without loss of generality that $x' > x$. We have 
that $\hat\clm^\veps(t,x') - \hat\clm^\veps(t,x)$ equals
\begin{align}
  a \veps^{1/4}\sum_{i = 0}^\infty \int_0^t Y_i^\veps(s)\left[ \partial_y \Psi^{x'}(t - s + \veps, Y^\veps_i(s)) - \partial_y \Psi^x(t - s + \veps, Y^\veps_i(s))\right] dW_i(s).
\end{align}
Applying the BDG inequality, we obtain 
\begin{multline}
 \|\hat\clm^\veps(t,x') - \hat\clm^\veps(t,x)\|_q^2 
 \leq c_1\Big\| \int_0^t \veps^{1/2} \sum_{i = 0}^\infty (Y_i^\veps(s))^2\left( (\partial_y \Psi^{x'} - \partial_y \Psi^x)(t - s + \veps, Y^\veps_i(s)) \right)^2 ds \Big\|_{q/2} \\
 \leq c_1 \int_0^t \Big( \Ebd^\gamma_a\Big[ \Big( \veps^{1/2} \sum_{i = 0}^\infty (Y_i^\veps(s))^2\Big( (\partial_y \Psi^{x'} - \partial_y \Psi^x)(t - s + \veps, Y^\veps_i(s)) \Big)^2 \Big)^{q/2} \Big] \Big)^{2/q} ds \\
 = c_1 \int_0^t \Big( \Ebd^0_a\Big[ (Y_{(0)}(s))^{\frac{2\gamma}{a}} \Big( \veps^{1/2} \int_0^\infty y^2\Big( (\partial_y \Psi^{x'} - \partial_y \Psi^x)(t - s + \veps, y) \Big)^2 N(\veps^{-1/2}dy)  \Big)^{q/2} \Big] \Big)^{2/q} ds,
\label{eq:Mxdiff}
\end{multline}
where we use Minkowski's inequality to obtain the third line, and in the fourth line we use that $N(\veps^{-1/2}dy)$ is a Poisson random measure on $\mathbb{R}$ of intensity $\veps^{-1/2}\mathbf{1}[0,\infty)(y)dy$. By Cauchy-Schwarz, the last line is bounded above by 
\begin{align}
 \;\; \;\; c_2 \int_0^t \Ebd^0_a\left[ \left( \veps^{1/2} \int_0^\infty \left( y(\partial_y \Psi^{x'} - \partial_y \Psi^x)(t - s + \veps, y) \right)^2 N(\veps^{-1/2}dy)  \right)^{q} \right]^{1/q} ds, \label{eq:Mxdiffbd1}
\end{align} 
upon noting that $\Ebd^0_a[ (Y_{(0)}(s))^{\frac{4\gamma}{a}}]^{1/q}$ does not depend on $s$, by stationarity, and is finite as noted in the proof of Lemma \ref{lem:tightatpt}. In view of Lemma \eqref{lem:poissonint}(i), 
\begin{align}
&\Ebd^0_a\left[ \left( \veps^{1/2} \int_0^\infty \left( y(\partial_y \Psi^{x'} - \partial_y \Psi^x)(t - s + \veps, y) \right)^2 N(\veps^{-1/2}dy)  \right)^{q} \right]^{1/q} \\
& \leq c_3 \max_{1 \leq k \leq q} \veps^{\frac{1}{2} - \frac{1}{2k}}\left( \int_0^\infty \left( y(\partial_y \Psi^{x'} - \partial_y \Psi^x)(t - s + \veps, Y^\veps_i(s)) \right)^{2k} dy \right)^{1/k}. \label{eq:Mxdiffmoments}
\end{align}
Observe that for $1 \leq k \leq q$,
\begin{multline}
 \int_0^\infty \left( y(\partial_y \Psi^{x'} - \partial_y \Psi^x)(t - s + \veps, y) \right)^{2k} dy 
 \leq \int_0^\infty \left( \int_x^{x'} |y \partial_u \partial_y \Psi^u(t - s + \veps, y)|du \right)^{2k} dy \\
 \leq \left( \int_x^{x'} \left( \int_0^\infty |y \partial_u \partial_y \Psi^u(t - s + \veps, y)|^{2k} dy \right)^{1/2k} du \right)^{2k} \\
 \leq c_4 \left( |x' - x| (t - s + \veps)^{-1 + 1/4k} \right)^{2k} = c_4(t - s + \veps)^{-2k + 1/2}|x' - x|^{2k},
\end{multline}
where we use Minkowski's inequalty in the third line and Lemma \ref{lem:derpsi}($i=4$) in the fourth. The first line above is also bounded by $c_5(t - s + \veps)^{-k + 1/2}$ by Lemma \ref{lem:derpsi}($i=1$). We conclude that 
\begin{align*}
&\left( \int_0^\infty \left( y(\partial_y \Psi^{x'} - \partial_y \Psi^x)(t - s + \veps, y) \right)^{2k} dy \right)^{1/k} \\
&\leq c_6 \min\left\{ (t - s + \veps)^{-1 + \frac{1}{2k}}, (t - s + \veps)^{-2 + \frac{1}{2k}}|x' - x|^2 \right\}.
\end{align*}
By the above and \eqref{eq:Mxdiffmoments}, the last quantity in \eqref{eq:Mxdiffbd1} is bounded above by 
\begin{align}
& c_{7} \max_{1 \leq k \leq q} \veps^{\frac{1}{2} - \frac{1}{2k}} \int_0^t \min\left\{ \frac{1}{(t - s + \veps)^{1 - \frac{1}{2k}}}, \frac{|x' - x|^2}{(t - s + \veps)^{2 - \frac{1}{2k}}}\right\} ds \\
& \leq c_{7} \max_{1 \leq k \leq q} \veps^{\frac{1}{2} - \frac{1}{2k}} \left( \int_{\veps}^{|x' - x| \vee \veps} \frac{d\sigma}{\sigma^{1 - \frac{1}{2k}}} + \int_{|x' - x| \vee \veps}^\infty \frac{|x' - x|^2  }{\sigma^{2 - \frac{1}{2k}}}d\sigma\right). \label{eq:Mxdiffbd3}
\end{align}
Consider the following cases. First, if $|x' - x| \leq \veps$, then \eqref{eq:Mxdiffbd3} is bounded above by
\begin{align}
c_{7} \max_{1 \leq k \leq q} \veps^{\frac{1}{2} - \frac{1}{2k}} \int_\veps^\infty \frac{|x' - x|^2}{\sigma^{2 - \frac{1}{2k}}} d\sigma & = c_{7} \max_{1 \leq k \leq q} \veps^{\frac{1}{2} - \frac{1}{2k}} \frac{1}{(1 - \frac{1}{2k}) \veps^{1 - \frac{1}{2k}}}|x' - x|^2 \\
& \leq c_{8} \veps^{-\frac{1}{2}}|x' - x|^2 \leq c_{8}|x' - x|^{3/2}.
\end{align}
Second, if $|x' - x| > \veps$, then \eqref{eq:Mxdiffbd3} is bounded above by
\begin{align}
& c_{7} \max_{1 \leq k \leq q} \veps^{\frac{1}{2} - \frac{1}{2k}} \left( 2k|x' - x|^{1/2k} + \frac{|x' - x|^2}{(1 - \frac{1}{2k})} \frac{1}{|x' - x|^{1 - \frac{1}{2k}}} \right) \\
& \leq c_{9} \max_{1 \leq k \leq q} \veps^{\frac{1}{2} - \frac{1}{2k}} |x' - x|^{1/2k} \leq c_{9}|x' - x|^{1/2}.
\end{align}
In either case, \eqref{eq:Mxdiffbd3}, and hence \eqref{eq:Mxdiffbd1}, is bounded above by $c_{10}|x' - x|^{1/2}$. In view of \eqref{eq:Mxdiff}, the result follows.
\end{proof}

\begin{proof}[Proof of Lemma \ref{lem:contbounds}(ii)] Assume without loss of generality that $t' > t$. Then 
\begin{align*}
&\hat\clm^\veps(t',x) - \hat\clm^\veps(t,x) \\
& = a\veps^{1/4} \sum_{i = 0}^\infty \Bigg\{ \int_0^t Y_i^\veps(s) \left[ \partial_y \Psi^x(t' - s + \veps, Y_i^\veps(s)) - \partial_y \Psi^x(t - s + \veps, Y_i^\veps(s)) \right] dW_i(s) \\
& \hspace{1in} + \int_t^{t'} Y_i(s) \partial_y \Psi^x(t' - s + \veps, Y_i^\veps(s)) dW_i(s) \Bigg\}.
\end{align*}
From this and the BDG inequality, we obtain
\begin{align}
& \| \hat\clm^\veps(t',x) - \hat\clm^\veps(t,x) \|_q \\
& \leq c_1\left\| \veps^{1/2} \sum_{i = 0}^\infty \int_0^t (Y_i^\veps(s))^2 \left[ \partial_y \Psi^x(t' - s + \veps, Y_i^\veps(s)) - \partial_y \Psi^x(t - s + \veps, Y_i^\veps(s)) \right]^2 ds \right\|_{q/2}^{1/2} \\
& \quad + c_1\left\| \veps^{1/2} \sum_{i = 0}^\infty \int_t^{t'} (Y_i(s))^2 \left[\partial_y \Psi^x(t' - s + \veps, Y_i^\veps(s)) \right]^2 ds \right\|_{q/2}^{1/2}.
\end{align}
We denote the first term above by $T_1 = T_1(x,t,t',\veps)$ and the second term above by $T_2 = T_2(x,t,t',\veps)$. To bound $T_1$, we apply Minkowski's inequality and the change of measure \eqref{eq:margCOM}, and apply Cauchy-Schwarz, in a similar way as we did in \eqref{eq:Mxdiff}, to obtain 
\begin{align}
T_1 & = \left( \Ebd_a^\gamma\left[ \left( \veps^{1/2} \sum_{i = 0}^\infty \int_0^t (Y_i^\veps(s))^2 \left( \int_{t - s + \veps}^{t' - s + \veps} \partial_u \partial_y \Psi^x(u, Y_i^\veps(s)) du \right)^2 ds \right)^{q/2} \right] \right)^{1/q} \\
& \leq c_2 \left(\int_0^t \left\{ \Ebd_a^0\left( \veps^{1/2} \int_0^\infty \left( \int_{t - s + \veps}^{t' - s + \veps} y\partial_u \partial_y \Psi^x(u, y) du \right)^2 N(\veps^{-1/2} dy) \right)^q  \right\}^{1/q} ds \right)^{1/2}.\\ \label{eq:T1bd1}
\end{align}
Denote the term inside the integral $\int_0^t$ by $I_1 = I_1(x,t,t',s,\veps)$. By Lemma \ref{lem:poissonint} and Lemma \ref{lem:derpsi}($i=3$).
\begin{multline}
I_1  \leq c_3 \max_{1 \leq k \leq q} \veps^{\frac{1}{2} - \frac{1}{2k}} \left( \int_0^\infty \left( \int_{t - s + \veps}^{t' - s + \veps} y\partial_u \partial_y \Psi^x(u, y) du \right)^{2k} dy \right)^{1/k} \\
 \leq c_3 \max_{1 \leq k \leq q} \veps^{\frac{1}{2} - \frac{1}{2k}} \left( \int_{t - s + \veps}^{t' - s + \veps} \left( \int_0^\infty |y\partial_u \partial_y \Psi^x(u, y)|^{2k} dy \right)^{1/2k} du \right)^{2} \\
 \leq c_3 \max_{1 \leq k \leq q} \veps^{\frac{1}{2} - \frac{1}{2k}} \left( \int_{t - s + \veps}^{t' - s + \veps} u^{-\frac{3}{2} + \frac{1}{4k}} du \right)^2 
 \leq c_4 \sum_{k = 1}^q \veps^{\frac{1}{2} - \frac{1}{2k}} \frac{(t' - t)^{1 - \frac{1}{2k}}}{(t' - s + \veps)^{1 - \frac{1}{2k}}(t - s + \veps)^{1 - \frac{1}{2k}}}. \label{eq:I1bd1}
\end{multline}
Note that in the last line we use the following inequality: For any $\alpha \in (0,1)$ and $0 < A < B < \infty$, $B^\alpha - A^\alpha \leq (B - A)^\alpha$. By \eqref{eq:T1bd1} and \eqref{eq:I1bd1},
\begin{align}
T_1 & \leq c_5 \left(\sum_{k = 1}^q \veps^{\frac{1}{2} - \frac{1}{2k}} \int_0^t \frac{(t' - t)^{1 - \frac{1}{2k}}}{(t' - s + \veps)^{1 - \frac{1}{2k}}(t - s + \veps)^{1 - \frac{1}{2k}}} ds\right)^{1/2} \\
& = c_5 \left(\sum_{k = 1}^q \veps^{\frac{1}{2} - \frac{1}{2k}} \int_0^t \frac{(t' - t)^{1 - \frac{1}{2k}}}{(u + \veps)^{1 - \frac{1}{2k}}(t' - t + u + \veps)^{1 - \frac{1}{2k}}} du \right)^{1/2},\label{eq:T1bd2}
\end{align}
where in the last equality we have used the substitution $u = t - s$. By considering two cases, we will show that the right-hand side of \eqref{eq:T1bd2} is bounded above by $c_6|t' - t|^{1/4}$.  First, if $t' - t \leq \veps$, 
then the summand in the last line of \eqref{eq:T1bd2} is bounded above by 
\begin{align}
\veps^{\frac{1}{2} - \frac{1}{2k}} \int_0^t \frac{(t' - t)^{1-1/2k}}{(u + \veps)^{2 - \frac{1}{k}}}du &\leq c_7 \veps^{\frac{1}{2} - \frac{1}{2k}}\frac{(t' - t)^{1-1/2k}}{\veps^{1 - \frac{1}{k}}}\left(1 + \log(T\veps^{-1} + 1)\right)\\
&\leq c_7 \veps^{\frac{1}{2} - \frac{1}{2k}}\frac{(t' - t)^{1-1/2k}}{\veps^{1 - \frac{1}{k}}}\left(1 + \log(2T(t'-t)^{-1})\right)\\
&\leq c_8(t' - t)^{1/2}\left(1 + |\log(2T)| + \left|\log\left(\frac{1}{t'-t}\right)\right|\right),
\end{align}
where we have used $(t' - t + u + \veps) \geq (u + \veps)$ for all $u \in [0,t]$,  and the logarithmic term arises from considering the $k=1$ case. Second, we consider the case $t' - t > \veps$. When $k=1$, the summand in the last line of \eqref{eq:T1bd2} is bounded above by  
\begin{align}
&\int_0^{t'-t} \frac{1}{(u + \veps)^{\frac{1}{2}}} du + \int_{t'-t}^T \frac{(t' - t)^{\frac{1}{2}}}{(u + \veps)} du \\
&\le (2\sqrt{2} + |\log (2T)|)(t'-t)^{1/2} + (t'-t)^{1/2}\left|\log\left(\frac{1}{t'-t}\right)\right|,
\end{align}
where we used $\veps < t'-t \le T$ in the last step.
For $k \ge 2$, we bound this summand by the following:
\begin{align}
 \veps^{\frac{1}{2} - \frac{1}{2k}} \int_0^t \frac{(t' - t)^{1 - \frac{1}{2k}}}{(t' - t)^{\frac{1}{2} - \frac{1}{2k}}(u + \veps)^{\frac{3}{2} - \frac{1}{2k}}}du &\leq \veps^{\frac{1}{2} - \frac{1}{2k}} (t' - t)^{\frac{1}{2}} \int_0^t \frac{du}{(u + \veps)^{\frac{3}{2} - \frac{1}{2k}}} \\& \leq c_9\veps^{\frac{1}{2} - \frac{1}{2k}} (t' - t)^{\frac{1}{2}}\veps^{-\frac{1}{2} + \frac{1}{2k}} \leq c_9(t' - t)^{1/2},
\end{align}
where we have used the inequality  $(t' - t + u + \veps)^{1 - \frac{1}{2k}} \geq (t' - t)^{\frac{1}{2} - \frac{1}{2k}}(u + \veps)^{\frac{1}{2}}$. These bounds imply that $T_1 \le c_{10}(t'-t)^{1/4}\left(1 + \left|\log\left(\frac{1}{t'-t}\right)\right|^{1/2}\right)$. 

Next we bound $T_2$. We have 
\begin{align}
T_2 & = c_1 \left( \Ebd_a^\gamma \left( \veps^{1/2} \sum_{i = 0}^\infty \int_t^{t'} (Y_i(s))^2 \left[\partial_y \Psi^x(t' - s + \veps, Y_i^\veps(s)) \right]^2 ds\right)^{q/2}  \right)^{1/q} \\
& \leq c_{11}\left( \int_t^{t'} \left( \Ebd_a^0 \left( \veps^{1/2} \int_0^\infty y^2[\partial_y \Psi^x(t' - s + \veps,y)]^2 N(\veps^{-1/2}dy) \right)^q  \right)^{1/q} ds \right)^{1/2}.
\end{align}
Denote the integrand with respect to $ds$ by $I_2(s)$. By Lemma \ref{lem:poissonint} and Lemma \ref{lem:derpsi}($i=1$),
\begin{align}
I_2(s) & \leq c_{12} \max_{1 \leq k \leq q} \veps^{\frac{1}{2} - \frac{1}{2k}} \left( \int_0^\infty \left(y\partial_y \Psi^x(t' - s + \veps,y)\right)^{2k}dy \right)^{1/k} \\ \label{eq:i2bd1}
 & \leq c_{13} \max_{1 \leq k \leq q} \veps^{\frac{1}{2} - \frac{1}{2k}}(t' - s + \veps)^{-1 + \frac{1}{2k}}.
\end{align}
Hence,
\begin{align*}
T_2 &\leq 
c_{14} \max_{1 \leq k \leq q}\left( \veps^{\frac{1}{2} - \frac{1}{2k}} \int_0^{t'-t} (u + \veps)^{-1 + \frac{1}{2k}} du \right)^{1/2}.
\end{align*}
If $t' - t \leq \veps$, then for all $k \geq 1$, 
$$
\veps^{\frac{1}{2} - \frac{1}{2k}} \int_0^{t'-t} (u + \veps)^{-1 + \frac{1}{2k}} du \leq c_{15}\veps^{-1/2}(t' - t) \leq c_{15}(t' - t)^{1/2}.
$$
 If, on the other hand, $t' - t > \veps$, then for $1 \leq k \leq q$,
\begin{align}
\veps^{\frac{1}{2} - \frac{1}{2k}} \int_0^{t'-t} (u + \veps)^{-1 + \frac{1}{2k}} du \leq \veps^{\frac{1}{2} - \frac{1}{2k}} \int_0^{t'-t} u^{-1 + \frac{1}{2k}} du  = 2k\veps^{\frac{1}{2} - \frac{1}{2k}}(t' - t)^{\frac{1}{2k}} \leq c_{16}(t' - t)^{1/2}.
\end{align}
We conclude that $T_2 \leq c_{17}|t' - t|^{1/4}$.
\end{proof}

\begin{proof}[Proof of Lemma \ref{lem:contbounds}(iii)] By definition, and the mass conservation property in \eqref{eq:cofmass},
\begin{align}
\hat \clw^\veps(x, t) 
& = \veps^{1/4}\left( \sum_{i = 0}^\infty \Psi^x(t + \veps, Y_i^\veps(0)) - \veps^{-1/2} \int_0^\infty \Psi^x(t + \veps,y)dy \right).
\end{align}
Consequently, for $0 < x < x' < \infty$, 
\begin{align}\label{eq:wdiff1}
&\|\hat \clw^\veps(x',t) - \hat \clw^\veps(x,t) \|_q \\
& = \veps^{1/4} \left( \Ebd_a^\gamma\left[ \left| \sum_{i = 0}^\infty (\Psi^{x'} - \Psi^x)(t + \veps, Y_i^\veps(0)) - \veps^{-1/2} \int_0^\infty (\Psi^{x'} - \Psi^x)(t + \veps,y)dy \right|^q \right] \right)^{1/q} \\
& = \veps^{1/4} \left( \Ebd_a^0\left[ Y_{(0)}(0)^{2\gamma/a} \left| \int_0^\infty (\Psi^{x'} - \Psi^x)(t + \veps,y)\widetilde{N}(\veps^{-1/2}dy) \right|^q \right] \right)^{1/q} \\
& \leq c_1 \veps^{1/4} \left( \Ebd_a^0\left[ \left| \int_0^\infty (\Psi^{x'} - \Psi^x)(t + \veps,y)\widetilde{N}(\veps^{-1/2}dy) \right|^{2q} \right] \right)^{1/2q},
\end{align}
where $\widetilde{N}(\veps^{-1/2}dy) = N(\veps^{1/2}dy) - \veps^{-1/2}\mathbf{1}[0,\infty)(y)dy $ is the compensated Poisson point process, and the last inequality follows by Cauchy-Schwarz. By Lemma \ref{lem:poissonint}(ii) and writing $(\Psi^{x'} - \Psi^x)(t + \veps, y) = \int_x^{x'} \hat{q}_{t + \veps}(z,y)dz$, the last line in \eqref{eq:wdiff1} is bounded above by 
\begin{align}
c_2 \max_{1 \leq k \leq q} \veps^{\frac{1}{4} - \frac{1}{4k}} \left( \int_0^\infty \left| \int_x^{x'} \hat{q}_{t + \veps}(z,y)dz  \right|^{2k}dy \right)^{1/2k}.
\end{align}
We will be done if we can show that, for $1 \leq k \leq q$, $$\veps^{\frac{1}{4} - \frac{1}{4k}} \left( \int_0^\infty \left| \int_x^{x'} \hat{q}_{t + \veps}(z,y)dz  \right|^{2k}dy \right)^{1/2k} \leq c_3|x' - x|^{1/8}.$$

\textit{Case 1. $k = 1$}. Recalling Lemma \ref{lem:areprobs}, we obtain
\begin{align}
&\left( \int_0^\infty \left| \int_x^{x'} \hat{q}_{t + \veps}(z,y)dz  \right|^{2} dy \right)^{1/2} \leq \left( \int_x^{x'} \int_0^\infty \hat{q}_{t + \veps}(z,y) dy dz \right)^{1/2} = |x' - x|^{1/2},
\end{align}
where the inequality follows by noting that $0 \leq \int_x^{x'} \hat{q}_{t + \veps}(z,y)dz \leq 1$ and Fubini's Theorem. 

\textit{Case 2. $2 \leq k \leq q$.} Let $\alpha_0 = |\log L|$ and $\alpha_1 = \alpha_0 + 1$. First we write
\begin{align}
&\veps^{\frac{1}{4} - \frac{1}{4k}}\left(\int_0^\infty \left| \int_x^{x'} \hat{q}_{t + \veps}(z,y)dz  \right|^{2k} dy \right)^{1/2k} \\
& \leq c_4 \Big\{ \veps^{\frac{1}{4} - \frac{1}{4k}}\Big(\int_0^{e^{\alpha_1}} \Big| \int_x^{x'} \hat{q}_{t + \veps}(z,y)dz  \Big|^{2k}dy \Big)^{1/2k} \\
& \hspace{1in} + \veps^{\frac{1}{4} - \frac{1}{4k}}\Big(\int_{e^{\alpha_1}}^\infty \Big| \int_x^{x'} \hat{q}_{t + \veps}(z,y)dz  \Big|^{2k} dy \Big)^{1/2k} \Big\}. \label{eq:splitmoment}
\end{align}
To bound the first term in \eqref{eq:splitmoment}, note that if $|x' - x| > \veps$ then\begin{equation}\label{eq:bd1ne}
\veps^{\frac{1}{4} - \frac{1}{4k}}\left(\int_0^{e^{\alpha_1}} \left| \int_x^{x'} \hat{q}_{t + \veps}(z,y)dz  \right|^{2k}dy \right)^{1/2k} \leq c_5 \veps^{\frac{1}{4} - \frac{1}{4k}} \leq c_5|x' - x|^{1/8},
\end{equation}
where we have used the observation that $0 \leq \int_x^{x'} \hat{q}_{t + \veps}(z,y)dz \leq 1$. If on the other hand $|x' - x| \leq \veps$, we have that 
\begin{equation}
\int_x^{x'} \hat{q}_{t+\veps}(z,y)dz = \int_x^{x'} (az)^{-1}p_{t + \veps}\left( \frac{\log z}{a} - \frac{\log y}{a} - \frac{a(t + \veps)}{2} \right) dz \leq c_6\frac{|x' - x|}{(t + \veps)^{1/2}}, \label{eq:bd1qint}
\end{equation}
and consequently, 
\begin{multline}
\veps^{\frac{1}{4} - \frac{1}{4k}}\left(\int_0^{e^{\alpha_1}} \left| \int_x^{x'} \hat{q}_{t + \veps}(z,y)dz  \right|^{2k}dy \right)^{1/2k}  \leq c_7\veps^{\frac{1}{4} - \frac{1}{4k}}\frac{|x' - x|}{(t + \veps)^{1/2}} 
\leq c_7|x' - x|^{1/2}, \label{eq:bd2qint}
\end{multline}
on noting $\veps \ge |x' - x|$. The bounds \eqref{eq:bd1ne} and \eqref{eq:bd2qint} give the desired bounds for the first term in \eqref{eq:splitmoment}.

To bound the second term in \eqref{eq:splitmoment}, changing variables in \eqref{eq:bd1qint} yields 
\begin{multline}
\int_x^{x'} \hat{q}_{t+\veps}(z,y)dz
 = \int_{\frac{\log x}{a} - \frac{\log y}{a} - \frac{a(t + \veps)}{2}}^{\frac{\log x'}{a} - \frac{\log y}{a} - \frac{a(t + \veps)}{2}} p_{t + \veps}(w) dw 
 = \int_{B(x,y,t+\veps)}^{B(x',y,t+\veps)} \frac{1}{\sqrt{2\pi}}e^{-w^2/2} dw, \label{eq:psixdiff}
\end{multline}
where 
$$
B(\tilde x,\tilde y,\tilde t) \doteq \frac{\log \tilde x}{a\tilde{t}^{1/2}} - \frac{\log \tilde y}{a\tilde{t}^{1/2}} - \frac{a\tilde{t}^{1/2}}{2} \quad \text{ for all $\tilde x,\tilde y \in (0,\infty)$ and $\tilde t \in (0,\infty)$.}
$$
Recall $\log \ell \leq \log x \leq \log x' \leq \log L \doteq \alpha_0$. For $\log y \geq \alpha_1$ and $w \in [B(x,y,t + \veps), B(x',y,t + \veps)]$, 
$$
\frac{1}{\sqrt{2\pi}}e^{-w^2/2} \leq c_8 \exp\left( -\frac{(\log y - \alpha_0)^2}{2a^2(t + \veps)} \right). \label{eq:ylarge}
$$
and for such $y$, \eqref{eq:psixdiff} is bounded in absolute value by 
\begin{align}
    c_9\frac{1}{(t+\veps)^{1/2}}|x' - x|\exp\left( -\frac{(\log y - \alpha_0)^2}{2a^2(t + \veps)} \right) \le c_{10}|x' - x|\exp\left( -\frac{(\log y - \alpha_0)^2}{4a^2(t + \veps)} \right), 
\end{align}
where we have used the fact  that $x,x' \in [\ell,L]$, so there exists a constant $\tilde c$ such that $|B(x',y,t + \veps) - B(x',y,t + \veps)| \leq \tilde c|x' - x|/(t + \veps)^{1/2}$.
Consequently, 
\begin{align}
\left(\int_{e^{\alpha_1}}^\infty \left| \int_x^{x'} \hat{q}_{t + \veps}(z,y)dz  \right|^{2k} dy \right)^{1/2k} & \leq c_{10}|x' - x| \left( \int_{e^{\alpha_1}}^\infty \exp\left( -\frac{k(\log y - \alpha_0)^2}{2a^2(t + \veps)} \right)dy  \right)^{1/2k} \\
& \leq c_{10}|x' - x| \left( \int_{\alpha_1}^\infty \exp\left(z -\frac{k(z - \alpha_0)^2}{2a^2(t + \veps)} \right)dz  \right)^{1/2k},
\label{eq:secondtermbd}
\end{align}
where we have made the substitution $z = \log y$ in the last line. By observing that
\begin{align}
 \int_{\alpha_1}^\infty \exp\left(z -\frac{k(z - \alpha_0)^2}{2a^2(t + \veps)} \right)dz 
< \infty,
\end{align}
we see that \eqref{eq:secondtermbd} is bounded above by $c_{11}|x' - x|$, which implies the desired bound for the second term in \eqref{eq:splitmoment}. This completes the proof.
\end{proof}

\begin{proof}[Proof of Lemma \ref{lem:contbounds}(iv).] Assume without loss of generality that $t' > t$. We have 
\begin{align}
& \|\hat\clw^\veps(t',x) - \hat\clw^\veps(t,x) \|_q \\
& = \veps^{1/4}\left\| \sum_{i = 0}^\infty \Psi^x(t' + \veps, Y_i^\veps(0)) - \Psi^x(t + \veps, Y_i^\veps(0)) \right\|_q \\
& = \veps^{1/4}\left( \Ebd^0_a\left[ (Y_{(0)}^\veps(0))^{\frac{2\gamma}{a}} \left| \int_0^{\infty} \left[ \Psi^x(t' + \veps, y) - \Psi^x(t + \veps, y) \right]N(\veps^{-1/2}dy) \right|^q \right] \right)^{1/q} \\
& \quad \leq c_1 \veps^{1/4} \left( \Ebd^0_a \left|\int_0^{\infty} \left[\Psi^x(t' + \veps, y) - \Psi^x(t + \veps, y) \right]N(\veps^{-1/2}dy)\right|^{2q}  \right)^{1/2q}, \label{eq:wtdiff}
\end{align}
on using the Cauchy-Schwarz inequality in the last line. By conservation of mass property in Lemma \ref{lem:areprobs}, $\int_0^\infty \Psi^x(t' + \veps,y) dy = \int_0^\infty \Psi^x(t + \veps,y)dy$, and therefore we may replace $N$ with $\widetilde{N}$ in the last expression above. Thus, by Lemma \ref{lem:poissonint}(ii), the last quantity in \eqref{eq:wtdiff} is bounded above by
\begin{align}
   c_1 \max_{1 \leq k \leq q} \veps^{\frac{1}{4} - \frac{1}{4k}} \left( \int_0^\infty |\Psi^x(t' + \veps, y) - \Psi^x(t + \veps, y)|^{2k}dy \right)^{1/2k}. \label{eq:wtdiffub}
\end{align}
Applying Minkowski's inequality and Lemma \ref{lem:derpsi}($i=2$),
\begin{multline}
\Bigg(\int_0^\infty |\Psi^x(t' + \veps, y)  - \Psi^x(t + \veps, y)|^{2k}dy \Bigg)^{1/2k} = \left(\int_0^\infty \left| \int_{t + \veps}^{t' + \veps} \partial_u \Psi^x(u,y)du \right|^{2k} dy\right)^{1/2k} \\
 \leq \int_{t + \veps}^{t' + \veps} \left( \int_0^\infty |\partial_u\Psi^x(u,y)|^{2k}dy \right)^{1/2k} du 
 \leq c_2\int_{t + \veps}^{t' + \veps} u^{-1 + \frac{1}{4k}} du \leq c_2 (t + \veps)^{-\frac{1}{4} + \frac{1}{4k}} \int_{t + \veps}^{t' + \veps} u^{-\frac{3}{4}} du \\
 \leq c_3 (t + \veps)^{-\frac{1}{4} + \frac{1}{4k}}[(t' + \veps)^{1/4} - (t + \veps)^{1/4}] 
 \leq c_3 (t + \veps)^{-\frac{1}{4} + \frac{1}{4k}}|t' - t|^{1/4}.
\end{multline}
 Therefore, \eqref{eq:wtdiffub} is bounded above by 
$c_3|t' - t|^{1/4}$,
as desired.
\end{proof}

\begin{proof}[Proof of Lemma \ref{lem:contbounds}(v) and (vi)] By \eqref{eq:VWerrM}, we may write
\begin{equation}
\hat{\clv}^\veps(t,x) = \hat{\clw}^\veps(t,x) + \hat{\mathcal{L}}^\veps(t,x) + \hat{\clm}^\veps(t,x),
\end{equation}
where 
\begin{align}
\hat{\mathcal{L}}^\veps(t,x) \doteq a\veps^{1/4}\int_0^t Y_{(0)}^\veps(s)\partial_y \Psi^x(t - s + \veps, Y_{(0)}^\veps(s))ds.
\end{align}
In view of parts (i)-(iv) of the lemma, it is sufficient to establish bounds analogous to (v)-(vi) for $\hat{\mathcal{L}}^\veps$.  Assume without loss of generality that $x' > x$ and $t' > t$. With respect to $\Pbd_a^0$, for each $s \in [0,\infty)$, $Y_{(0)}(s)$ is distributed as $\text{Exp}(1)$. Using this, we compute 
\begin{align}
& \| \hat{\mathcal{L}}^\veps(t,x') - \hat{\mathcal{L}}^\veps(t,x) \|_q = \left( \Ebd_a^\gamma \left| a\veps^{1/4}\int_0^t Y_{(0)}^\veps(s)\partial_y (\Psi^{x'} - \Psi^x)(t - s + \veps, Y_{(0)}^\veps(s))ds \right|^q  \right)^{1/q} \\
& \leq a\veps^{1/4}\int_0^t \left( \Ebd_a^\gamma \left| Y_{(0)}^\veps(0)\partial_y (\Psi^{x'} - \Psi^x)(t - s + \veps, Y_{(0)}^\veps(0)) \right|^q \right)^{1/q} ds  \\
& \le c_2\veps^{1/4} \int_0^t \left(\int_0^\infty |y \partial_y(\Psi^{x'} - \Psi^x)(t - s + \veps, y)|^{2q} e^{-\veps^{-1/2}y} \veps^{-1/2}dy\right)^{1/2q} ds \\
& = c_2\veps^{1/4} \int_0^t \left(\int_0^\infty \left| \int_{x}^{x'} y \partial_u \partial_y\Psi^u(t - s + \veps, y)du \right|^{2q} e^{-\veps^{-1/2}y} \veps^{-1/2}dy\right)^{1/2q} ds \\
& \leq c_2 \veps^{\frac{1}{4} - \frac{1}{4q}} \int_0^t \int_x^{x'} \left( \int_0^\infty |y \partial_u \partial_t\Psi^u(t - s + \veps, y)|^{2q}dy \right)^{1/2q}du ds \\
& \leq c_3 \veps^{\frac{1}{4} - \frac{1}{4q}}|x' - x| \int_0^t (t - s + \veps)^{-1 + \frac{1}{4q}} ds \leq c_3|T + \veps|^{1/4q} |x' - x|.
\end{align}
Here the first inequality is obtained by applying Minkowski's inequality and noting that the process is stationary. In the second inequality, we use \eqref{eq:margCOM} and apply Cauchy-Schwarz. In the last two lines, we apply Minkowski's inequality again, and we use Lemma \ref{lem:derpsi}($i=4$). Part (v) follows. To obtain (vi), we write 
\begin{multline}
\hat{\mathcal{L}}^\veps(t',x) - \hat{\mathcal{L}}^\veps(t,x) = a\veps^{1/4}\int_t^{t'} Y_{(0)}^\veps(s) \partial_y \Psi^x(t' - s + \veps, Y_{(0)}^\veps(s))ds \\
 + a\veps^{1/4} \int_0^t Y_{(0)}^\veps(s) \left[\partial_y \Psi^x(t' - s + \veps, Y_{(0)}^\veps(s)) - \partial_y \Psi^x(t - s + \veps, Y_{(0)}^\veps(s))\right]ds.
\end{multline}
Denote the first term on the right-hand side by $T_1 = T_1(x,t,t',\veps,q)$, and denote the second term by $T_2 = T_2(x,t,t',\veps,q)$. It is sufficient to establish continuity bounds in time for each term. Proceeding in a similar fashion as before, we have 
\begin{align}
\|T_1\|_q & \leq a\veps^{1/4} \int_t^{t'} \left( \Ebd_a^\gamma \left| Y_{(0)}^\veps(s) \partial_y \Psi^x(t' - s + \veps, Y_{(0)}^\veps(s)) \right|^q  \right)^{1/q} ds \\
& \leq c_4 \veps^{\frac{1}{4} - \frac{1}{4q}} \int_t^{t'} \left( \int_0^\infty |y \partial_y \Psi^x(t' - s + \veps, y)|^{2q} dy \right)^{1/2q} ds \\
& \leq c_5 \veps^{\frac{1}{4} - \frac{1}{4q}} \int_t^{t'} (t' - s + \veps)^{-\frac{1}{2} + \frac{1}{4q}}ds 
 \leq c_6|t' - t|^{1/2}.
\end{align}
Note that in the second to last inequality, we use Lemma \ref{lem:derpsi}($i=1$). For the second term, writing
$$
T_2 = a\veps^{1/4} \int_0^t \left[ \int_{t - s + \veps}^{t' - s + \veps} Y_{(0)}^\veps(s)\partial_u \partial_y \Psi^x(u, Y_{(0)}^\veps(s)) du \right] ds,
$$
we  have 
\begin{align}
\|T_2\|_q & \leq a\veps^{1/4}\int_0^t \int_{t - s + \veps}^{t' - s + \veps} \left( \Ebd_a^\gamma |Y_{(0)}^\veps(s)\partial_u \partial_y \Psi^x(u, Y_{(0)}^\veps(s))|^{q} \right)^{1/q} du ds \\
& \leq c_8 \veps^{\frac{1}{4} - \frac{1}{4q}} \int_0^t \int_{t - s + \veps}^{t' - s + \veps} \left( \int_0^\infty |y\partial_u \partial_y \Psi^x(u, y) |^{2q} e^{-\veps^{-1/2}y} dy \right)^{1/2q} du ds \\
& \leq c_9 \veps^{\frac{1}{4} - \frac{1}{4q}} \int_0^t \int_{t - s + \veps}^{t' - s + \veps} u^{-\frac{3}{2} + \frac{1}{4q}} du ds \\
& \leq c_{10}|t' - t|^{1/2},
\end{align}
where we use Lemma \ref{lem:derpsi}($i=3$) in the third line.
\end{proof}

\begin{proof}[Proof of Lemma \ref{lem:secreq}]
For $x,y \in (0,\infty)$, $t \in [0,\infty)$, $\veps \in (0,1)$ and $\eta \in (0,1)$, let 
$$
F^{\veps,\eta}(x,y) \doteq \Psi^{xe^{a\eta}}(\veps,y) - \Psi^{xe^{-a\eta}}(\veps,y),
\mbox{ and }
\tilde{F}^{\veps}(x,y) \doteq F^{\veps,\veps^{1/2}}(x,y).
$$
Let 
$$
G^{\veps,\eta}(t,x) \doteq \veps^{1/4} \sum_{i = 0}^\infty F^{\veps,\eta}(x,Y_i^\veps(t)), \mbox{ and }
\tilde{G}^{\veps}(t,x) \doteq G^{\veps, \veps^{1/2}}(t,x)=\veps^{1/4} \sum_{i = 0}^\infty \tilde{F}^{\veps}(x,Y_i^\veps(t)).
$$
The proof proceeds by establishing the following three claims.

\textbf{Claim 1.} \textit{There exists a constant $C_0 \in (0,\infty)$ such that for all $x, y \in (0,\infty)$, and $\veps \in (0,1)$,}
\begin{equation}
|\Psi^x(\veps,y) - \mathbf{1}_{[0,x]}(y)| \leq C_0\tilde{F}^\veps(x,y).
\end{equation}

\textbf{Claim 2.} \textit{Let $0 < \ell < L < \infty$, $T \in (0,\infty)$, and $q \in [2,\infty)$. Then there is a $C_1=C_1(\ell,L,T,q) \in (0,\infty)$ such that for all $x, x' \in [\ell, L]$, $t, t' \in [0,T]$, and $\veps \in (0,1)$,
\begin{equation}
\| \tilde{G}^{\veps}(t,x) - \tilde{G}^{\veps}(t,x') \|_q \leq C_1|x' - x|^{1/8},
\end{equation}
and 
\begin{equation}
\| \tilde{G}^{\veps}(t',x) - \tilde{G}^{\veps}(t,x) \|_q \leq C_1|t' - t|^{1/4}\left(1 + \left|\log\left( \frac{1}{t' - t} \right)\right|^{1/2}\right).
\end{equation}
}

\textbf{Claim 3.} \textit{For all $x \in (0,\infty)$ and $t \in [0,\infty)$, $\tilde{G}^\veps(t,x) \to 0$ in $L^1(\Pbd_a^\gamma)$ as $\veps \to 0$.}\\

\textit{Proof of Claim 1.} In the proof, we will write $\vartheta(x,y,\veps) \doteq \frac{\log x}{a} - \frac{\log y}{a} - \frac{a\veps}{2}$, $x,y \in (0,\infty)$. By \eqref{eq:Psixexp} and standard Gaussian tail bounds, if $y \in [0,x]$, then,
\begin{align}
|\Psi^x(\veps,y) - \mathbf{1}_{[0,x]}(y)|  &= 1 - \int_{-\infty}^{\vartheta(x,y,\veps)} p_\veps(z)dz = \int_{\vartheta(x,y,\veps)}^\infty p_\veps(z)dz \\
&  \leq c_1\left(1 \wedge \frac{\veps^{1/2}}{\left|\vartheta(x,y,\veps)\right|}\right) \exp\left( -\frac{\vartheta(x,y,\veps)^2}{2\veps}\right), \label{eq:l8c1bd2}
\end{align}
and if $y > x$, then the same bound holds: 
\begin{align}
|\Psi^x(\veps,y) - \mathbf{1}_{[0,x]}(y)| = \int_{-\infty}^{\vartheta(x,y,\veps)} p_\veps(z)dz \leq c_1\left(1 \wedge \frac{\veps^{1/2}}{\left|\vartheta(x,y,\veps)\right|}\right) \exp\left( -\frac{\vartheta(x,y,\veps)^2}{2\veps}\right).\\ \label{eq:l8c1bd3}
\end{align}
In the above, we adopt the convention $1 \wedge \infty=1$. Now, by \eqref{eq:Psixexp}, if $|\vartheta(x,y,\veps)| \ge 2\veps^{1/2}$,
\begin{align}
\tilde{F}^\veps(x,y) & = \int_{\vartheta(x,y,\veps) - \veps^{1/2}}^{\vartheta(x,y,\veps) + \veps^{1/2}} p_{\veps}(w)dw \geq \frac{\veps^{1/2}}{|\vartheta(x,y,\veps)| + \veps^{1/2}}\int_{\vartheta(x,y,\veps) - \veps^{1/2}}^{\vartheta(x,y,\veps) + \veps^{1/2}} \frac{|w|}{\veps^{1/2}}\, p_{\veps}(w)dw \\
& = \frac{\veps^{1/2}}{\sqrt{2\pi}(|\vartheta(x,y,\veps)| + \veps^{1/2})} \Big(\exp\Big( -\frac{(|\vartheta(x,y,\veps)| - \veps^{1/2})^2}{2\veps}\Big)\\
&\quad\quad\quad\quad\quad  - \exp \Big( -\frac{(|\vartheta(x,y,\veps)| + \veps^{1/2})^2}{2\veps}\Big)\Big)\\
&\geq c_2 \frac{\veps^{1/2}}{|\vartheta(x,y,\veps)|}\exp\left( -\frac{\vartheta(x,y,\veps)^2}{2\veps}\right).
\label{inter}
\end{align}

On the other hand, again by \eqref{eq:Psixexp}, if $|\vartheta(x,y,\veps)| < 2\veps^{1/2}$,
\begin{align}
\tilde{F}^\veps(x,y) & = \int_{\vartheta(x,y,\veps) - \veps^{1/2}}^{\vartheta(x,y,\veps) + \veps^{1/2}} p_{\veps}(w)dw \geq 2\veps^{1/2}p_{\veps}\left( \left|\vartheta(x,y,\veps)\right| + \veps^{1/2}\right) \ge c_3.
\label{eq:l8c1bd1}
\end{align}
The bounds \eqref{eq:l8c1bd2}, \eqref{eq:l8c1bd3}, \eqref{inter} and \eqref{eq:l8c1bd1} imply the claim.

\textit{Proof of Claim 2.} Observe that 
\begin{align}
\tilde{G}^\veps(t,x) & = \veps^{1/4}\left( \langle Q_t^\veps, \phi^{xe^{a\veps^{1/2}},\veps}_t(t,\cdot)
\rangle - \langle Q_t^\veps,  \phi^{xe^{-a\veps^{1/2}},\veps}_t(t,\cdot) \rangle\right) \\
& = \hat{\clv}^\veps(t,xe^{a\veps^{1/2}}) - \hat{\clv}^\veps(t,xe^{-a\veps^{1/2}}) + \veps^{-1/4}\int_0^\infty \left[\Psi^{xe^{a\veps^{1/2}}}(\veps,y) - \Psi^{xe^{-a\veps^{1/2}}}(\veps,y)\right]dy, \label{eq:GasV}
\end{align}
where $\hat{\clv}^\veps$ is as defined in \eqref{eq:hatvdef}.  By Lemma \ref{lem:contbounds}(v) and (vi), for some $C= C(q,\ell, L, T)\in (0,\infty)$,
$$
\|\hat{\clv}^\veps(t,x'e^{\pm a\veps^{1/2}}) - \hat{\clv}^\veps(t,xe^{\pm a\veps^{1/2}})\|_q \leq C|x' - x|^{1/8},
$$
and 
$$
\|\hat{\clv}^\veps(t',xe^{\pm a\veps^{1/2}}) - \hat{\clv}^\veps(t,xe^{\pm a\veps^{1/2}})\|_q \leq C|t' - t|^{1/4} \left(1 + \left|\log\left( \frac{1}{t' - t} \right)\right|^{1/2}\right).
$$
Thus, it remains to show  appropriate continuity bounds for the last term in \eqref{eq:GasV}.
 Note that this term does not depend on $t$ so we only need a continuity estimate in $x$.
We have 
\begin{align}
&\veps^{-1/4}\left| \int_0^\infty \left( \Psi^{x'e^{a\veps^{1/2}}}(\veps,y) - \Psi^{x'e^{-a\veps^{1/2}}}(\veps,y) - \Psi^{xe^{a\veps^{1/2}}}(\veps,y) + \Psi^{xe^{-a\veps^{1/2}}}(\veps,y) \right)dy \right| \\
& = \veps^{-1/4} \left|\int_0^\infty \left( \int_{x'e^{-a\veps^{1/2}}}^{x'e^{a\veps^{1/2}}} \hat{q}_\veps(z,y) dz  - \int_{xe^{-a\veps^{1/2}}}^{xe^{a\veps^{1/2}}} \hat{q}_\veps(z,y) dz \right) dy \right| \\
& = \veps^{-1/4} |x' - x|(e^{a\veps^{1/2}} - e^{-a\veps^{1/2}}) \leq c_4\veps^{1/4}|x' - x|,
\end{align}
where the third line follows by changing the order of integration and using $\int_0^\infty \hat{q}_\veps(z,y) dy = 1$.

\textit{Proof of Claim 3.} We compute 
\begin{align}
\Ebd_a^\gamma[| \tilde{G}^\veps(t,x)|] & = \veps^{1/4}\Ebd_a^0\left[(Y_{(0)}(0))^{2\gamma/a}\left| \int \tilde{F}^\veps(x,y)N(\veps^{-1/2}dy) \right| \right] \\
& \leq \veps^{1/4}\Ebd_a^0\left[ (Y_{(0)}(0))^{4\gamma/a} \right]^{1/2} \left(\Ebd_a^0  \left| \int \tilde{F}^\veps(x,y)N(\veps^{-1/2}dy) \right|^2 \right)^{1/2} \\
& \leq c_5 \max_{k \in \{1,2\}} \veps^{\frac{1}{4} - \frac{1}{2k}} \left( \int_0^\infty |\tilde{F}^\veps(x,y)|^k dy \right)^{1/k},
\end{align}
where the last line uses Lemma \ref{lem:poissonint}(i). For $k \in \{1,2\}$, 
\begin{multline}
 \left( \int_0^\infty |\tilde{F}^\veps(x,y)|^k dy \right)^{1/k}  =  \left( \int_0^\infty \left| \int_{xe^{-a\veps^{1/2}}}^{xe^{a\veps^{1/2}}} \hat{q}_{\veps}(z,y)dz \right|^k dy \right)^{1/k} \\
   \leq \left( \int_{xe^{-a\veps^{1/2}}}^{xe^{a\veps^{1/2}}} \int_0^\infty \hat{q}_{\veps}(z,y)dy dz \right)^{1/k}
  = |x|\cdot|e^{a\veps^{1/2}} - e^{-a\veps^{1/2}}|^{1/k} \leq c_6 \veps^{1/2k}.
\end{multline}
Here the first inequality follows by noting that the quantity inside the absolute value signs is bounded above by 1. Substituting this bound into the previous bound gives us $\Ebd_a^\gamma[| \tilde{G}^\veps(t,x)|] \leq c_7 \veps^{1/4}$, and the claim follows.

\textit{Completing the proof of Lemma \ref{lem:secreq}.} Note for any $x_0 \in (0,\infty)$, the sequence $\{\tilde{G}^\veps(0,x_0), \veps > 0\}$ is tight by Claim 3. Hence, by the Kolmogorov-Chentsov criterion(cf.  \cite[Theorem 1.4.1]{kun}) and Claim 2, the sequence $\{\tilde{G}^\veps, \veps > 0\}$ is tight in $C([0,\infty) \times (0,\infty) : \mathbb{R})$. Thus, in view of Claim 3, $\tilde{G}^\veps \to 0$ uniformly on compacts in probability. Applying Claim 1, we therefore have 
\begin{align}
\sup_{x \in [\ell, L], t \in [0,T]} \veps^{1/4} \left| \sum_{i = 0}^\infty \left(\Psi^x(\veps, Y_i^\veps(t)) - \mathbf{1}_{(0,x]}(Y_i^\veps(t))\right) \right| \leq C_0\sup_{x \in [\ell, L], t \in [0,T]} \tilde{G}^\veps(t,x) \to 0, 
\end{align}
as $\veps \to 0$, as desired.
\end{proof}

\subsection{Proof of Theorem \ref{thm:formain1}.} \label{ssec:main3prf}

We begin with the following lemma which will be used in the proof of Lemma \ref{lem:rint} below.

\begin{lemma}\label{lem:qest} 
Fix $T > 0$. There exist constants $C_1$ and $C_2$ (depending only on $a$, $\gamma$, and $T$) such that for all $x,y \in (0,\infty)$ and $t \in (0,T]$, the following bounds hold:
\begin{longlist}
\item $
x q_t(x,y) \leq C_1 t^{-1/2} \exp(-C_2 t^{-1}|\log x - \log y|^2),
$



\item For all $m > 1$, 
$
\int_0^\infty [xq_t(x,y)]^m dx \leq C y t^{-(m-1)/2},
$
for some constant $C = C(a,\gamma,T,m)$.
\end{longlist}
\end{lemma}

\begin{proof}
(i) We have 
\begin{align}
xq_t(x,y) & = \frac{1}{a\sqrt{2\pi t}}\exp\left(-\frac{1}{2t}\left\{\frac{\log x}{a} - \frac{\log y}{a} + \frac{at}{2} \right\}^2 \right) \\
& \leq c_1t^{-1/2} \exp(-c_2 t^{-1}|\log x - \log y|^2). \label{eq:qbd}
\end{align}
This inequality is clear since $t$ is restricted to $(0,T]$, and when $|\log x - \log y|^2$ is large it dominates the quantity in the exponential on the first line.

(ii) For any $p > 0$, one may show that
\begin{equation}\label{eq:plogx}
\int_0^\infty e^{-p(\log(x))^2}dx = \sqrt{\pi} p^{-1/2} e^{1/4p}.
\end{equation}
Indeed, this can be obtained by making the substitution $u = \log x$, and re-expressing the resulting integrand as a constant times a Gaussian density. Using this, together with the bound (i), we obtain 
\begin{align}
\int_0^\infty [xq_t(x,y)]^m dx & \leq  c_1^m t^{-m/2} \int_0^\infty \exp\left(-\frac{mc_2}{t}\log(x/y)^2 \right) dx \\
& = c_1^m t^{-m/2} \sqrt{\pi} (mc_2/t)^{-1/2} e^{t/4mc_2} y 
 \leq c_2 y t^{-(m-1)/2}.
\end{align}
\end{proof}

The next lemma will be used to prove the convergence of certain quadratic variation processes in the proof of Lemma \ref{lem:Wfind} below.
\begin{lemma} \label{lem:rint}
Let $x_1, x_2, t_1, t_2, T \in (0,\infty)$ with $T \leq t_1 \wedge t_2$. The following convergence holds in $L^1(\Pbd_a^\gamma)$.
\begin{align}
&\int_{(0,T)} \Bigg| \veps^{1/2}\sum_{i = 0}^\infty Y_i^\veps(s)^2 q_{t_1-s + \veps}(Y_i^\veps(s), x_1)q_{t_2-s + \veps}(Y_i^\veps(s), x_2) \\
& \hspace{1.7in} - \int_{(0,\infty)} y^2 q_{t_1-s}(y, x_1)q_{t_2-s}(y, x_2)dy \Bigg|ds \to 0 \quad \text{ as } \veps \to 0. \label{eq:covlim}
\end{align}
\end{lemma}

\begin{proof}
For $j = 1,2$, let
$$
g_j^\veps(s,y) = yq_{t_j - s + \veps}(y, x_j), \quad\quad \text{ and } \quad\quad g_j^0(s,y) = yq_{t_j - s}(y, x_j).
$$
The quantity \eqref{eq:covlim} is bounded in expectation by 
\begin{align}
&\int_{(0,T)} \Ebd_a^\gamma\left| \veps^{1/2} \sum_{i = 0}^\infty g_1^\veps(s, Y_i^\veps(s))g_2^\veps(s, Y_i^\veps(s)) - \int_{(0,\infty)} g_1^\veps(s, y)g_2^\veps(s, y)dy \right|ds \\
& \quad + \int_{(0,T)} \left| \int_{(0,\infty)} \left[g_1^\veps(s, y)g_2^\veps(s, y) - g_1^0(s, y)g_2^0(s, y)\right] dy\right| ds. \label{eq:diff1}
\end{align}
To prove the lemma, it is enough to show that each term in \eqref{eq:diff1} converges to zero. To handle the first, we use the fact that, under the probability measure $\Pbd_a^0$, the point process $N^\veps(dy) : =\sum_{i = 0}^\infty \delta_{Y_i^\veps(0)}$ is distributed as a Poisson random measure of intensity $\veps^{-1/2}\mathbf{1}[0,\infty)(y) dy$. We have 
\begin{align}
& \Ebd_a^\gamma\left| \veps^{1/2} \sum_{i = 0}^\infty g_1^\veps(s, Y_i^\veps(s))g_2^{\veps}(s, Y_i^\veps(s)) - \int_{(0,\infty)} g_1^\veps(s, y)g_2^\veps(s, y)dy \right| \\
& \leq c_1 \Big(\Ebd_a^0\Big| \veps^{1/2} \sum_{i = 0}^\infty g_1^\veps(s, Y_i^\veps(0))g_2^{\veps}(s, Y_i^\veps(0)) - \int_{(0,\infty)} g_1^\veps(s, y)g_2^\veps(s, y)dy \Big|^2\Big)^{1/2} \\
& = c_1 \left( \veps^{1/2} \int_{(0,\infty)} g_1^\veps(s, y)^2 g_2^\veps(s, y)^2 dy \right)^{1/2} \\
& \leq \veps^{1/4}c_1 \left( \int_{(0,\infty)} g_1^\veps(s, y)^4 dy\right)^{1/4} \left( \int_{(0,\infty)} g_2^\veps(s, y)^4 dy \right)^{1/4} 
 \leq c_2\veps^{1/4}(T - s)^{-3/4}.\\
\label{eq:ebound}
\end{align}
Here the second line follows from stationarity and changing the probability measure to $\Pbd_a^0$ using \eqref{eq:margCOM}, and the Cauchy-Schwarz Inequality. The third line follows by a calculation similar to \eqref{eq:ZisGamma}, and by observing that the term in the paranthesis in the previous line is the variance of $\int \veps^{1/2} g_1^\veps(s,y) g_2^\veps(s,y) N^\veps(dy)$. In the  last line we use the following bound, which comes from Lemma \ref{lem:qest}(ii): For $j = 1,2$ and $s \in (0,T)$,
$$
\int_{(0,\infty)} g_j^\veps(s,y)^4 dy \leq \tilde{c}_1 x_j(t_j - s + \veps)^{-3/2} \leq \tilde{c}_1 x_j(T - s)^{-3/2}.
$$
Integrating the last line of \eqref{eq:ebound} over $(0,T)$ gives us a bound on the first term of \eqref{eq:diff1} which converges to zero as $\veps \to 0$.

To show the second term in \eqref{eq:diff1} converges to zero, we use dominated convergence. First, by Lemma \ref{lem:qest}(i), for any compact interval $[b,d] \subset (0,T)$,
\begin{align*}
& \sup_{s \in [b, d]} g_1^0(s,y)g_2^0(s,y) \leq c_2(t_1 - d)^{-1/2}(t_2 - d)^{-1/2}\\
& \hspace{1in} \times \exp\left(-c_3((t_1 - b)^{-1}|\log y - \log x_1|^2 +(t_2 - b)^{-1}|\log y- \log x_2|^2)\right).
\end{align*}
Since the bound is integrable in $y$ (see \eqref{eq:plogx}), this implies that $s \mapsto \int_0^\infty g_1^0(s,y)g_2^0(s,y) dy$ is continuous on $(0,T)$. In particular, for each $s \in (0,T)$, $$\int_0^\infty g_1^\veps(s,y)g_2^\veps(s,y) dy \to \int_0^\infty g_1^0(s,y)g_2^0(s,y) dy \quad \text{ as } \veps \to 0.$$ Moreover, by Lemma \ref{lem:qest}(ii) and Cauchy-Schwarz, for all $s \in (0,T)$
\begin{align}
&\left|\int_0^\infty [g_1^\veps(s,y)g_2^\veps(s,y) - g_1^0(s,y)g_2^0(s,y)]dy \right| \\
& \leq \left( \int_0^\infty |g_1^\veps(s,y)|^2 dy \right)^{1/2} \left( \int_0^\infty |g_2^\veps(s,y)|^2 dy \right)^{1/2}\\
&\quad + \left( \int_0^\infty |g_1^0(s,y)|^2 dy \right)^{1/2} \left( \int_0^\infty |g_2^0(s,y)|^2 dy \right)^{1/2} \\
& \leq c_4 ((T - s +\veps)^{-1/2} + (T - s)^{-1/2}) \leq 2c_4 (T - s)^{-1/2}.
\end{align}
Since the bound is integrable in $s$ on $(0,T)$, we conclude that the second term in \eqref{eq:diff1} converges to zero.
\end{proof}

The next lemma deals with the convergence of the finite-dimensional distributions of $(\hat{\clm}^\veps, \hat{\clw}^\veps)$. Fix $x_1, \dots, x_n \in (0,\infty)$ and $t_1,\dots, t_n \in [0,\infty)$. Define random vectors in $\mathbb{R}^n$ 
\begin{align*}
 \mathbf{M}^\veps = (\hat{\clm}^\veps(x_1,t_1), \dots, \hat{\clm}^\veps(x_n,t_n)), \;\; \mathbf{W}^\veps = (\hat{\clw}^\veps(x_1,t_1), \dots, \hat{\clw}^\veps(x_n,t_n)),
 \end{align*}
 and define $\mathbf{M}$, $\mathbf{W}$ similarly by replacing $\hat{\clm}^\veps$ with $\hat{\clm}$ and $\hat{\clw}^\veps$ with $\hat{\clw}$, respectively.

We suppress from our notation the dependence of these quantities on the $x_i$'s and $t_i$'s. In view of the definitions of $\hat\clm$ and $\hat\clw$, $\mathbf{M}$ and $\mathbf{W}$ are Gaussians with covariance matrices given, respectively, by 
$$
\Sigma_{j\ell}^{\mathbf{M}} = \int_0^{t_j \wedge t_\ell} 
a^2 y^2 q_{t_j - s}(y, x_j)q_{t_\ell - s}(y, x_\ell)dy,\;
\Sigma_{j\ell}^{\mathbf{W}} = \int_0^\infty \Psi^{x_j}(t_j, y)\Psi^{x_\ell}(t_\ell, y)dy, \; 
 1 \leq j, \ell \leq k.
$$

\begin{lemma} \label{lem:Wfind}
For any $\mathbf{u} \in \mathbb{R}^n$, $\mathbf{u} \cdot \mathbf{W}^\veps \Rightarrow \mathbf{u} \cdot \mathbf{W}$.
\end{lemma}

\begin{proof} 
We prove the result first for a drift of $\gamma = 0$, and then for a general drift (namely, first under $\Pbd_a^{0}$ and then under $\Pbd_a^{\gamma}$ for $\gamma > 0$).

\textit{Case 1: $\gamma = 0$.} Fix $\mathbf{u} = (u_1, \dots, u_n) \in \mathbb{R}^n$. By definition of $\hat{\clw}^\veps$ and conservation of mass property in Lemma \ref{lem:areprobs}, we have  
$$
\hat{\clw}^\veps(t,x) = \veps^{1/4}\sum_{i = 1}^{\infty} \left( \Psi^x(t + \veps, Y_i^\veps(0)) - \veps^{-1/2}\int_0^\infty \Psi^x(t + \veps, y)dy \right).
$$
Therefore, letting 
$$
\varphi^\veps(y) \doteq \sum_{j = 1}^n u_j\Psi^{x_j}(t_j + \veps,y),
$$
we can write  
\begin{equation}
\mathbf{u} \cdot \mathbf{W}^\veps = \sum_{j = 1}^n u_j \hat{\clw}^\veps(x_j, t_j) = \veps^{1/4}\left( \int_0^\infty \varphi^\veps(y)N^\veps(dy) - \int_0^\infty \varphi^\veps(y) \veps^{-1/2} dy \right), \label{eq:fdd_as_pi}
\end{equation}
where $N^\veps(dy) \doteq \sum_{i = 0}^\infty \delta_{Y_i^\veps(0)}(dy)$. By the characteristic function formula for integrals with respect to Poisson random measures \cite[Theorem 2.7, Page 41]{kyprianou2006introductory}, for any $\theta \in \mathbb{R}$,
\begin{equation}\label{master}
\Ebd_a^0\left[e^{i\theta(\mathbf{u} \cdot \mathbf{W}^\veps)}\right] = \exp\left(\int_0^{\infty}\left(e^{i\theta\veps^{1/4}\varphi^\veps(y)} - 1 - i\theta\veps^{1/4}\varphi^\veps(y)\right)\veps^{-1/2}dy\right).
\end{equation}
By Taylor's theorem,
$$
\veps^{-1/2}\int_0^{\infty}\left|e^{i\theta\veps^{1/4}\varphi^\veps(y)} - 1 - i\theta\veps^{1/4}\varphi^\veps(y) - \frac{(i\theta\veps^{1/4}\varphi^\veps(y))^2}{2}\right|dy \le \frac{|\theta|^3\veps^{1/4}}{3!}\int_0^{\infty}|\varphi^\veps(y)|^3dy 
$$
which converges to $0$ as $\veps \rightarrow 0$ by Lemma \ref{lem:psiinat}. Using this in \eqref{master}, and observing that
$$\int_0^{\infty}\left(\sum_{j = 1}^n u_j\Psi^{x_j}(t_j+\veps,y)\right)^2dy \to \int_0^{\infty}\left(\sum_{j = 1}^n u_j\Psi^{x_j}(t_j,y)\right)^2dy$$
as $\veps \to 0$,
we conclude that, as $\veps \rightarrow 0$,
$$
\Ebd_a^0\left[e^{i\theta(\mathbf{u} \cdot \mathbf{W}^\veps)}\right] \rightarrow \exp\left(-\frac{\theta^2}{2}\int_0^{\infty}\left(\sum_{j = 1}^n u_j\Psi^{x_j}(t_j,y)\right)^2dy\right) = \Ebd_a^0\left[e^{i\theta(\mathbf{u} \cdot \mathbf{W})}\right],
$$
finishing the proof for Case 1.

\textit{Case 2: $\gamma > 0$.} 
To handle the general drift case, the idea is to show that asymptotically $\mathbf{u} \cdot \mathbf{W}^\veps$ becomes independent of the position of the lowest particle $X_{(0)}$, and thus the replacement of $\Pbd_a^0$ with $\Pbd_a^\gamma$ does not effect the limiting distribution. In view of \eqref{eq:fdd_as_pi}, we can write 
\begin{equation} 
\mathbf{u} \cdot \mathbf{W}^\veps = L^\veps + U^\veps, \label{eq:UplusL}
\end{equation}
where 
\begin{align}
L^\veps & \doteq \veps^{1/4} \int_0^\infty \left[\varphi^\veps(y) - \varphi^\veps(y - Y_{(0)}^\veps(0))\right]N^\veps(dy), \text{ and } \\  U^\veps & \doteq \veps^{1/4}\left( \int_0^\infty \varphi^\veps(y - Y_{(0)}^\veps(0))N^\veps(dy) - \int_0^\infty \varphi^\veps(y) \veps^{-1/2} dy \right).
\end{align}
First, we will show that $L^\veps \to 0$ in $\Pbd_a^\gamma$ probability. 
Note that
$L^\veps = \sum_{j=1}^n u_j L^{\veps}_{t_j, x_j}, 
$
where, for $(t,x)  \in [0, \infty)\times (0,\infty)$
$$
L^{\veps}_{t, x} = 
\veps^{1/4} \int_0^\infty \left[\Psi^{x}(t+ \veps, y) - \Psi^x(t+ \veps, y - Y_{(0)}^\veps(0))\right]N^\veps(dy).$$
Thus it suffices to show that for each $(t,x)  \in [0, \infty)\times (0,\infty)$, $L^{\veps}_{t, x}\to 0$ in $\Pbd_a^\gamma$ probability. 
Note that, for $\eta >0$
$$\Pbd_a^\gamma(|L^{\veps}_{t, x}|>\eta) =
\Ebd_a^0(e^{2\gamma X_{(0)}(0)} \one\{|L^{\veps}_{t, x}|>\eta\})\le \left[\Ebd_a^0 e^{4\gamma X_{(0)}(0)}\right]^{1/2}\left[\Pbd_a^0 (|L^{\veps}_{t, x}|>\eta)\right]^{1/2}.$$
Thus it suffices to show that for each $(t,x)  \in [0, \infty)\times (0,\infty)$, $L^{\veps}_{t, x}\to 0$ in $\Pbd_a^0$ probability.
Note that 
$$
|L^{\veps}_{t, x}| \le \veps^{1/4} \int_0^\infty 
\int_{y - Y_{(0)}^\veps(0)}^y 
|\partial_z \Psi^x(t + \veps,z)| dz N^\veps(dy).
$$
For a fixed $M >0$,
\begin{multline*}
\Pbd_a^0 (|L^{\veps}_{t, x}|>\eta) \le
\Pbd_a^0 (|L^{\veps}_{t, x}|>\eta ; Y_{(0)}^{\veps}\le M\sqrt{\veps}) + \Pbd_a^0 (Y_{(0)}^{\veps} > M\sqrt{\veps})\\
\le \Pbd_a^0\left(\veps^{1/4} \int_0^\infty 
\int_{(y - M\sqrt{\veps})^+}^y 
|\partial_z \Psi^x(t + \veps,z)| dz N^\veps(dy) > \eta\right)
+ \Pbd_a^0 (Y_{(0)}^{\veps} > M\sqrt{\veps}).
\end{multline*}
Since $\limsup_{M\to \infty} \sup_{\veps>0}\Pbd_a^0 (Y_{(0)}^{\veps} > M\sqrt{\veps}) = 0$, it suffices to show that for each $(t,x)  \in [0, \infty)\times (0,\infty)$ and $M>0$
$$
\veps^{1/4} \int_0^\infty 
\int_{(y - M\sqrt{\veps})^+}^y 
|\partial_z \Psi^x(t + \veps,z)| dz N^\veps(dy) \to 0 \mbox{ in } L^1(\Pbd_a^0).
$$
Note that
\begin{multline*}
\veps^{1/4}\Ebd_a^0 \int_0^\infty 
\int_{(y - M\sqrt{\veps})^+}^y 
|\partial_z \Psi^x(t + \veps,z)| dz N^\veps(dy)\\
=\veps^{-1/4}\int_0^\infty 
\int_{(y - M\sqrt{\veps})^+}^y 
|\partial_z \Psi^x(t + \veps,z)| dz dy
\le M \veps^{1/4} \int_0^\infty |\partial_z \Psi^x(t + \veps,z)| dz 
\doteq M \tilde L^{\veps}_{t,x}.
\end{multline*}
By \eqref{eq:ydPsi}, 
$
\partial_z \Psi^x(t + \veps,z) = -q_{t + \veps}(z,x).
$
Therefore, 
$$
\tilde{L}_{t,x}^\veps = \veps^{1/4} \int_0^\infty q_{t + \veps}(z,x)dz = \veps^{1/4} \to 0, \quad \text{ as } \veps \to 0.
$$
where the last equality follows by Lemma \ref{lem:areprobs} (see the line after \eqref{eq:cofmass}).




This completes the proof of the statement $L^\veps_{t,x} \to 0$ in $\Pbd_a^\gamma$ probability  for each $(t,x)  \in [0, \infty)\times (0,\infty)$.

Recall \eqref{eq:UplusL}. We will be done if we can show that $U^\veps \Rightarrow \mathbf{u} \cdot \mathbf{W}$ under the measure $\Pbd_a^\gamma$. 
Note that
\begin{align*}
U^{\veps} - \veps^{1/4} \varphi^{\veps}(0)
&= \veps^{1/4} \left( \sum_{i=0}^{\infty} \varphi^{\veps} (Y^{\veps}_{(i)}(0) -
Y^{\veps}_{(0)}(0)) -  \varphi^{\veps}(0) - \int_0^\infty \varphi^\veps(y) \veps^{-1/2} dy \right)\\
&=\veps^{1/4} \left( \sum_{i=1}^{\infty} \varphi^{\veps} (Y^{\veps}_{(i)}(0) -
Y^{\veps}_{(0)}(0))  - \int_0^\infty \varphi^\veps(y) \veps^{-1/2} dy \right)\\
&= \veps^{1/4} \left( \int_0^\infty \varphi^\veps(y)\bar N^\veps(dy)
- \int_0^\infty \varphi^\veps(y) \veps^{-1/2} dy \right),
\end{align*}
where $\bar N^\veps(dy) \doteq \sum_{i = 0}^\infty \delta_{(Y_{(i)}^\veps(0)- 
Y_{(0)}^\veps(0))}(dy)$. Using  from Lemma \ref{lem:ispp} that $\{Y^{\veps}_{(i)} (0) - Y^{\veps}_{(0)} (0): i \ge 1\}$
is a Poisson point process with rate $\veps^{-1/2}$ under $\Pbd_a^\gamma$, 
we see that the law of $N^{\veps}$ under $\Pbd_a^0$  is same as the law of $\bar N^{\veps}$ under $\Pbd_a^\gamma$, and therefore
$$ \Pbd_a^{\gamma} \circ (U_{\veps} - \veps^{1/4} \varphi^{\veps}(0))^{-1} = \Pbd_a^0 \circ (\mathbf{u} \cdot \mathbf{W}^{\veps})^{-1}.
$$
The result now follows from Case 1 and the fact that, since $\varphi^{\veps}$
is bounded, $\veps^{1/4} \varphi^{\veps}(0)\to 0$ as $\veps \to 0$.
\end{proof}

\begin{proof}[Proof of Theorem \ref{thm:formain1}]
Fix $\mathbf{u} = (u_1, \dots, u_n)$ and $\mathbf{v} = (v_1, \dots, v_n)$ in $\mathbb{R}^n$. By tightness given in Lemma \ref{lem:tight}, it is enough to show that $(\mathbf{u} \cdot \mathbf{W}^\veps,  \mathbf{v} \cdot \mathbf{M}^\veps) \Rightarrow (\mathbf{u} \cdot \mathbf{W}, \mathbf{v} \cdot \mathbf{M}$). Proving joint convergence requires some care, since $\mathbf{W}^\veps$ and $\mathbf{M}^\veps$ are not independent in the prelimit. We define auxiliary $\sigma$-fields
$$
\mathcal{Y} \doteq \sigma\{Y_i(0), i \in \mathbb{N}_0\},
$$ 
and
$$
{\tilde{\mathcal{F}}_t \doteq \mathcal{Y} \vee \sigma\{W_i(s), 0 \leq s \leq t, i \in \mathbb{N}_0\} \vee \sigma\{W_{-1}(s), 0 \leq s \leq t\}, \text{ for } t \geq 0,}
$$
where $W_{-1}(\cdot)$ is a standard Brownian motion independent of the $W_i(\cdot)$'s
and $\mathcal{Y}$. For $t \geq 0$, let
\begin{equation}
V^\veps(t) = W_{-1}(t) + \sum_{j = 1}^n v_j\hat{\clm}^\veps(x_j, t \wedge t_j). \label{eq:Vdef}
\end{equation}
Observe that 
\begin{equation}
V^\veps(t) = W_{-1}(t) + \mathbf{v} \cdot \mathbf{M}^\veps, \text{ for all } t \geq \max\{t_1, \dots, t_n\}. \label{eq:thepointofallthis}
\end{equation}
Using the definition of $\hat{\mathcal{M}}^\veps(\cdot)$ and independence of $W_{-1}(\cdot)$ and $\hat{\mathcal{M}}^\veps(\cdot)$, we compute
\begin{align}
A^\veps(t) & \doteq \langle V^\veps \rangle_t \\
& = t + \sum_{1 \leq j, \ell \leq n} v_j v_\ell \langle \hat{\clm}^\veps(x_j, \cdot), \hat{\clm}^\veps(x_\ell, \cdot) \rangle_{t \wedge t_j \wedge t_\ell} \\
& = t + \sum_{1 \leq j, \ell \leq n} v_j v_\ell a^2 \veps^{1/2} \sum_{i = 0}^\infty \int_0^{t \wedge t_j \wedge t_\ell} [Y_i^\veps(s)]^2\partial_y \phi_{t_j}^{x_j, \veps}(s,Y_i^\veps(s)) \partial_y \phi_{t_l}^{x_\ell, \veps}(s,Y_i^\veps(s))ds \\
& = t + \sum_{1 \leq j, \ell \leq n} v_j v_\ell a^2 \veps^{1/2} \sum_{i = 0}^\infty \int_0^{t \wedge t_j \wedge t_\ell} [Y_i^\veps(s)]^2 q_{t_j - s + \veps}(Y_i^\veps(s), x_j) q_{t_\ell - s + \veps}( Y_i^\veps(s), x_l)ds.
\end{align}
Note that expressing the quadratic covariation as the infinite series in the third line is justified by Lemma \ref{lem:mart}. The last line uses the fact that, from \eqref{eq:ydPsi}, for $s\le t$,
\begin{equation}
\partial_y \phi_t^{x, \veps}(s,y) = \partial_y \Psi^x(t-s+\veps, y) = -q_{t-s+\veps}(y,x).
\end{equation}
The point of introducing the extra Brownian motion $W_{-1}(\cdot)$ in the definition of $V^\veps(\cdot)$ is that $A^\veps(\cdot)$ is strictly increasing and $A^\veps(t) \to \infty$ as $t \to \infty$. Hence, for any $s \geq 0$, the stopping time 
$$
T^\veps(s) \doteq \inf\{t \geq 0 : A^\veps(t) \geq s\}, 
$$
is well-defined, and moreover, $T^\veps(\cdot)$ is the unique inverse of $A^\veps(\cdot)$. \ab{By the Dubins-Schwarz Theorem}, the process 
$$
B^\veps(s) \doteq V^\veps(T^\veps(s)), s \geq 0,
$$
is a standard Brownian motion with respect to the filtration $\{\mathcal{H}_s^\veps\}$, where 
$
\mathcal{H}_s^\veps \doteq \tilde{\mathcal{F}}_{T^\veps(s)}
$
(cf. Theorem 4.6 in \cite{KSbook}.) Consequently, $B^\veps(\cdot)$ is independent of the $\sigma$-field $\mathcal{H}_0^\veps = \tilde{\mathcal{F}}_{T^\veps(0)} = \tilde{\mathcal{F}}_0 = \mathcal{Y}$. (Specifically, note that $B^\veps(\cdot)$ and $\mathbf{W}^\veps$ are independent.) By Lemma \ref{lem:Wfind}, $\mathbf{u} \cdot \mathbf{W}^\veps \Rightarrow \mathbf{u} \cdot \mathbf{W}$. Therefore, the following weak limit holds:
\begin{equation}
(\mathbf{u} \cdot \mathbf{W}^\veps, B^\veps(\cdot)) \Rightarrow (\mathbf{u} \cdot \mathbf{W}, B(\cdot)), \label{eq:weak_wB}
\end{equation}
where $B(\cdot)$ is a standard Brownian motion independent of $\mathbf{W}$. Moreover, the process $V^\veps(\cdot)$ has the following representation: 
\begin{equation}
V^\veps(s) = B^\veps(A^\veps(s)), \text{ for all } s \geq 0. \label{eq:Vrep}
\end{equation}
For $t \geq 0$, let
\begin{equation}
A^0(t) \doteq t + \sum_{1 \leq j,\ell \leq n} v_jv_\ell a^2 \int_0^{t \wedge t_j \wedge t_\ell} \int_0^\infty y^2 q_{t_j - s}(y,x_j)q_{t_\ell - s}(y,x_\ell)dy ds. \label{eq:A0def}
\end{equation}
For any $T \geq 0$,
\begin{align}
    & \sup_{t \in [0,T]} \left| A^\veps(t) - A^0(t) \right| \\
    & \leq \sum_{1 \leq j, \ell \leq n} |v_j v_\ell| a^2 \int_0^{t_j \wedge t_\ell \wedge T} \Bigg| \veps^{1/2}\sum_{i = 0}^\infty [Y_i^\veps(s)]^2 q_{t_j - s + \veps}( Y_i^\veps(s), x_j) q_{t_\ell - s + \veps}(Y_i^\veps(s), x_{\ell}) \\ 
    & \hspace{2in} - \int_0^\infty y^2 q_{t_j-s}(y, x_j)q_{t_\ell-s}(y, x_\ell)dy \Bigg|ds, \label{eq:Aunifloc}
\end{align}
and the right-hand side converges to zero in $L^1(\Pbd_a^\gamma)$ by Lemma \ref{lem:rint}.  Consequently, (see \cite[Lemma, Section 14]{billingsley2013convergence})
$$
(\mathbf{u} \cdot \mathbf{W}^\veps, V^\veps(\cdot) )
= (\mathbf{u} \cdot \mathbf{W}^\veps, B^\veps(A^\veps(\cdot)) )
\Rightarrow (\mathbf{u} \cdot \mathbf{W}, B(A^0(\cdot)) ).
$$
In particular, with $t= \max\{t_1, \ldots, t_n\}$,
$(\mathbf{u} \cdot \mathbf{W}^\veps, V^\veps(t) )
\Rightarrow (\mathbf{u} \cdot \mathbf{W}, B(A^0(t)) ).$
Recall from \eqref{eq:thepointofallthis} that
$V^{\veps}(t) = W_{-1}(t) + \mathbf{v} \cdot \mathbf{M}^\veps$
where $W_{-1}$ and $\mathbf{M}^\veps$ are independent.
Also note that $B(A^0(t))$ has the same distribution as $\mathbf{v} \cdot \mathbf{M} + \sqrt{t} \xi$ where $\xi$ is a standard Normal r.v. independent of $\mathbf{M}$. Combining the above we see that
$(\mathbf{u} \cdot \mathbf{W}^\veps, \mathbf{v} \cdot \mathbf{M}^\veps )
\Rightarrow (\mathbf{u} \cdot \mathbf{W}, \mathbf{v} \cdot \mathbf{M} )$.
The result follows.
\end{proof}

\subsection{Proof of Proposition \ref{prop:yaxis}}
\label{sec:proofprop}
The proofs in this section use several constructions and ideas from
 \cite{dembo2017equilibrium}. One key such idea is to establish uniform in time probability estimates by combining uniform estimates in small time windows with large deviation bounds and union estimates (see proof of Lemma \ref{lem:111}). Another idea is to estimate the processes $X_i$ by a pair of processes $X_i^l$ and $X_i^r$ that give lower and upper bounds on $X_i$, respectively (see Lemmas \ref{lem:expbd} and \ref{lem:111}).
 Define for $(t,x) \in [0,\infty) \times (0,\infty)$,
$
\rho^{\veps}_t(x) \doteq Y_{(I^{\veps}_t(x))}(t) - \veps^{-1/2}x$,
and for $j, j' \in \NN_0$, 
$$\cld(j, j', t) = (j-j') - (Y_{(j)}(t) - Y_{(j')}(t)).$$
It then follows that, for $(t,x) \in [0,\infty) \times (0,\infty)$,
\begin{equation}
\chi^{\veps}(t,x) - \tilde \chi^{\veps}(t,x) =
\veps^{1/4} \cld(I^{\veps}_t(x), I^{\veps}_0(x), t) +
\veps^{1/4} \rho^{\veps}_t(x) \label{eq:1253}
\end{equation}
and
\begin{equation}\label{eq:eq410}
\tilde \chi^{\veps}(t,x) - \check \chi^{\veps}(t,x) =
\veps^{1/4} \cld(I^{\veps}_0(x), i_{\veps}(x), t).
\end{equation}

We will need the following lemma.
\begin{lemma}
\label{lem:expbd}
For all $\gamma\ge 0$, $a >0$, there are $\kappa_1, \kappa_2 \in (0, \infty)$ such that
{for all $j \in \NN_0$ and $m \in (2a+1, \infty)$ with $j> (m-2\gamma)/a$}, and $t \in (0,\infty)$,
$$\Ebd^{\gamma}_a\left(\sup_{0\le s \le t}
e^{m |X_{(j)}(s) - X_{(j)}(0)|}\right) \le \kappa_2 (j+1) e^{\kappa_1 m^2 t}.$$
\end{lemma}
\begin{proof}
Define for $i \in \NN_0$ and $t \ge 0$
\begin{equation}\label{eq:xrl}
X^{l}_i(t) = X_i(0) + W_i(t), \;\; X^{r}_i(t) = X_i(0) + W_i(t) + \gamma t.
\end{equation}
Fix $m \in (2a+1, \infty)$.
For $i, j \in \NN_0$, let
$$
U^r(t,i,j) = \sup_{s \in [0,t]} e^{m(X^r_i(s) - X_{(j)}(0))}, \;\;
U^l(t,i,j) = \sup_{s \in [0,t]} e^{-m(X^l_i(s) - X_{(j)}(0))}.
$$
Let 
$$
b_{i,j} \doteq \prod_{k=i+1}^j \left(\frac{2\gamma + ka}{2\gamma+ka +m}\right), \; i < j, \;
B_{i,j} \doteq \prod_{k=j+1}^i \left(\frac{2\gamma + ka}{2\gamma+ka +m}\right), \; j > i,
$$
and set $b_{ii}=B_{ii}=1$.
Then, for suitable $\kappa_1, \tilde\kappa_2 \in (0, \infty)$ (which depend only on $\gamma$), {writing $|W_i|_{*,t} \doteq \sup_{s \in [0,t]} |W_i(s)|$}, for $i\le j$
\begin{equation}\label{eq:332}
\Ebd^{\gamma}_a U^r(t,i,j) \le 
\Ebd^{\gamma}_a e^{-m(X_{(j)}(0) - X_{(i)}(0))}
\Ebd^{\gamma}_ae^{m (|W_i|_{*,t}+ \gamma t)}
\le b_{i,j} \tilde\kappa_2 e^{\kappa_1 m^2t},
\end{equation}
and for $i>j$
\begin{equation}\label{eq:333}
\Ebd^{\gamma}_a U^l(t,i,j) \le 
\Ebd^{\gamma}_a e^{-m(X_{(i)}(0) - X_{(j)}(0))}
\Ebd^{\gamma}_ae^{m |W_i|_{*,t}}
\le B_{i,j} \tilde\kappa_2 e^{\kappa_1 m^2t}.
\end{equation}
Here we have used the fact that, under $\Pbd^{\gamma}_a$, the gap sequence in \eqref{eq:gapt} (with $N=\infty$) is distributed as
$\pi_a^{\gamma} \doteq \otimes_{i=1}^{\infty} \mbox{Exp}(2\gamma+ ia)$.
Since $b_{i,j} \le 1$,
$$\Ebd^{\gamma}_a \sup_{0\le s \le t} e^{m(X^r_{(j)}(s) - X_{(j)}(0))} \le
\Ebd^{\gamma}_a\sum_{i=0}^j U^r(t, i,j) \le 
(\sum_{i=0}^j b_{i,j}) \tilde\kappa_2 e^{\kappa_1 m^2t}
\le (j+1) \tilde\kappa_2 e^{\kappa_1 m^2t},
$$
where the first inequality uses the fact that, for $s\ge 0$ and $j \in \NN_0$, $\max_{0\le i\le j} X^r_i(s) \ge X^r_{(j)}(s)$.
Also, since for all $j \in \NN_0$ and $s\ge 0$, $X^l_{(j)}(s) \ge \inf_{j\le i < \infty} X^l_i(s)$, we have
\begin{equation}\Ebd^{\gamma}_a \sup_{0\le s \le t} e^{-m(X^l_{(j)}(s) - X_{(j)}(0))} \le
\Ebd^{\gamma}_a \sum_{i=j}^{\infty} U^l(t, i,j) \le 
(\sum_{i=j}^{\infty} B_{i,j}) \tilde\kappa_2 e^{\kappa_1 m^2t}. \label{eq:347}
\end{equation}
We claim that, for all $j \in \NN_0$
$
\sum_{i=j}^{\infty} B_{i,j} \le 2(2\gamma+a+1) (j+1)$.
To see this, note that
$$\log \prod_{k=j+1}^i \left(\frac{2\gamma + ka}{2\gamma+ka +m}\right)
= - \sum_{k=j+1}^i \log (1 + m/(2\gamma +ka)),$$
and since $m/(2\gamma +ka)<1$ and $\log(1+x) \ge x/2$ for $x \in (0,1)$, we have 
$$B_{ij} = \prod_{k=j+1}^i \left(\frac{2\gamma + ka}{2\gamma+ka +m}\right) \le
\left(\frac{2\gamma + a(i+1)}{2\gamma + a(j+1)}\right)^{-m/2a}.$$
Thus, recalling that $m-2a>1$,
\begin{align*}
\sum_{i=j+1}^{\infty} B_{i,j} &\le
(2\gamma + a (j+1))^{m/2a} \sum_{i=j+1}^{\infty} (2\gamma + a (i+1))^{-m/2a}\\
&\le 
\frac{1}{a} \frac{2a}{m-2a} (2\gamma + a (j+1))
\le 2(2\gamma + a(j+1)).
\end{align*}
and so
$$\sum_{i=j}^{\infty} B_{i,j} \le 1+ 2(2\gamma + a(j+1)) \le 2(2\gamma+a+1)(j+1),$$
which proves the claim.
Combining this with \eqref{eq:347}
we have that
$$\Ebd^{\gamma}_a \sup_{0\le s \le t} e^{-m(X^l_{(j)}(s) - X_{(j)}(0))} \le 2(2\gamma+a+1) (j+1) \tilde\kappa_2 e^{\kappa_1 m^2t}.$$
Finally,
\begin{align*}
\Ebd^{\gamma}_a \sup_{0\le s \le t}
e^{m |X_{(j)}(s) - X_{(j)}(0)|} &\le
\Ebd^{\gamma}_a \sup_{0\le s \le t} e^{m(X^r_{(j)}(s) - X_{(j)}(0))}
+ \Ebd^{\gamma}_a \sup_{0\le s \le t} e^{-m(X^l_{(j)}(s) - X_{(j)}(0))}\\
&\le (j+1) \kappa_2 e^{\kappa_1 m^2t},
\end{align*}
where $\kappa_2 = (1+ 2(2\gamma+a+1))\tilde \kappa_2$.
\end{proof}

\begin{lemma}\label{lem:expdec1}
Fix $\gamma \ge 0$, $a>0$, $\upsilon >0$, $0< \ell < L<\infty$, and $\theta \in (3/2, \infty)$. Then, for every $\beta \in (0, 1/2)$ there are $\kappa_3, \kappa_4 \in (0, \infty)$ and $\veps^* >0$ such that for all $\veps \in (0, \veps^*)$,
\begin{align*}
&\sup_{\ell \le x \le L} \Pbd_a^{\gamma}
\left(\veps^{1/4} Y_{(i_{\veps}(x))}(0)
\sup_{0\le s \le \veps^{\theta/2}} \left|
\exp\{a(X_{(i_{\veps}(x))}(s) - X_{(i_{\veps}(x))}(0)) + a^2s/2\}
-1\right| > \upsilon \right) \\
&\le \kappa_3 e^{-\kappa_4\veps^{-\beta/2}}.
\end{align*}
\end{lemma}
\begin{proof}
Fix $\theta \in (3/2, \infty)$ and $0< \ell < L<\infty$.
Let $\sigma(\veps) \doteq \frac{\upsilon}{4e L \log 2} \veps^{1/4}$.
Choose $\veps_0$ sufficiently small so that $\sigma(\veps_0)<1$, $\min\{\ell, 1\}\veps_0^{-1/2} >1$.
Then, for $\veps\in (0,\veps_0)$,
\begin{align}
&\Pbd_a^{\gamma}
\Big(\veps^{1/4} Y_{(i_{\veps}(x))}(0)
\sup_{0\le s \le \veps^{\theta/2}} \Big|
\exp\{a(X_{(i_{\veps}(x))}(s) - X_{(i_{\veps}(x))}(0)) + a^2s/2\}
-1\Big| > \upsilon \Big)\\
&\le \Pbd_a^{\gamma}
\Big(\sup_{0\le s \le \veps^{\theta/2}} |a(X_{(i_{\veps}(x))}(s) - X_{(i_{\veps}(x))}(0)) + a^2s/2|
\ge \sigma(\veps)\Big)
+ \Pbd_a^{\gamma}(e \sigma(\veps)\veps^{1/4} Y_{(i_{\veps}(x))}(0) > \upsilon),
\end{align}
where we have used the fact that, for $|x| \in [0,1]$, $|e^x-1| \le |x|e$.
Next note that since, for all $x \in [\ell,L]$, $Y_{(i_{\veps}(x))}(0)$ is the sum of $i_{\veps}(x)$ iid Exponential random variables with 
mean $1$ under the law $\Pbd_a^0$ (note that by choice of $\veps_0$, $i_{\veps}(x)>1$ for $x \in [\ell, L]$), writing $\kappa \doteq \left(\Ebd_a^0\left[(Y_{(0)}(0))^{4\gamma/a}\right]\right)^{1/2}$,
\begin{align}
\Pbd_a^{\gamma}(e \sigma(\veps)\veps^{1/4} Y_{(i_{\veps}(x))}(0) > \upsilon)
&\le \kappa \left[\Pbd_a^0(e \sigma(\veps)\veps^{1/4} Y_{(i_{\veps}(x))}(0) > \upsilon)\right]^{1/2}\\
&\le \kappa \exp\left\{-\frac{\upsilon}{4e\sigma(\veps)} \veps^{-1/4}\right\} \left(\int_0^{\infty} e^{x/2} e^{-x} dx\right)^{\veps^{-1/2}L/2}\nonumber\\
&= \kappa\exp\{-(L\log 2)\veps^{-1/2}\} \exp\{(L\log 2)\veps^{-1/2}/2\}\\
& = \kappa\exp\{-(L\log 2)\veps^{-1/2}/2\}.\label{eq:502}
\end{align}
Fix $\beta \in (0, 1/2)$ and let $m(\veps) \doteq \veps^{-\beta_1/2} 4eL \log 2/\upsilon$, where 
$\beta_1 = \beta + 1/2$.
Choose $\veps_1 \in  (0,\veps_0)$ so that $a^2\veps_1^{\theta/2} \le \sigma(\veps_1) = \frac{\upsilon}{4 eL \log 2} \veps_1^{1/4}$,
$(\veps^{-1/2}\ell-1) > (m(\veps) - 2\gamma)/a$ for all $\veps \in (0, \veps_1)$,
and $m(\veps_1) > 2a+1$.
Then, for all $\veps \in (0, \veps_1)$,
\begin{align*}
&\Pbd_a^{\gamma}
\left(\sup_{0\le s \le \veps^{\theta/2}} |a(X_{(i_{\veps}(x))}(s) - X_{(i_{\veps}(x))}(0)) + a^2s/2|
\ge \sigma(\veps)\right)\\
&\le
\Pbd_a^{\gamma}\left(\sup_{0\le s \le \veps^{\theta/2}} |(X_{(i_{\veps}(x))}(s) - X_{(i_{\veps}(x))}(0)|
\ge \frac{\sigma(\veps)}{2a}\right)\\
&\le 
\Ebd_a^{\gamma}\left(\sup_{0\le s \le\veps^{\theta/2}} e^{m(\veps) |(X_{(i_{\veps}(x))}(s) - X_{(i_{\veps}(x))}(0)| }
\right) e^{-m(\veps) \sigma(\veps)/2a}\\
&\le \kappa_2 (L\veps^{-1/2} + 1) e^{\kappa_1 m(\veps)^2 \veps^{\theta/2}}e^{-m(\veps) \sigma(\veps)/2a},
\end{align*}
where the last line follows from Lemma \ref{lem:expbd}
on using the fact that, by our choice of $\veps_1$,
for all $x \in [\ell, L]$, $i_{\veps}(x) > (m(\veps)-2\gamma)/a$.

Since $2\beta_1-\theta <\beta$, we can choose $\veps_2\in (0, \veps_1)$ such that for all $\veps \in (0,\veps_2)$,
$$(L\veps^{-1/2} + 1) e^{\kappa_1 16e^2(\log 2)^2 L^2 \veps^{-(2\beta_1-\theta)/2}/\upsilon^2} e^{-\frac{1}{2a}\veps^{-\beta/2}} \le e^{-\frac{1}{4a} \veps^{-\beta/2}}.$$
Then for all $\veps\in (0, \veps_2)$,
$$
\Pbd_a^{\gamma}
\left(\sup_{0\le s \le \veps^{\theta/2}} |a(X_{(i_{\veps}(x))}(s) - X_{(i_{\veps}(x))}(0)) + a^2s/2|
\ge \sigma(\veps)\right) \le \kappa_2 e^{-\frac{1}{4a} \veps^{-\beta/2}}.
$$
The result follows.
\end{proof}

We will use the following tail estimate for a standard rate $1$ Poisson process (cf. \cite[Proposition 6.2]{budhiraja2017control}). 
\begin{lemma}\label{lem:PPLD}
Let $\{\cln(t)\}_{t\ge 0}$ be  a rate $1$ Poisson process. For each $\gamma>0$ and $b_1, b_2 \ge 0$, there exist $\veps_0 \in (0,1)$ and $B_1, B_2 \in (0, \infty)$ such that for all $\veps \in (0,\veps_0)$ and $L_1 \in (0, \infty)$
\begin{equation}\label{eq:ppldp}
P\left(\sup_{0 \le s \le \veps^{-(2b_1+ b_2)}L_1} |\cln(s)-s| \ge \gamma \veps^{-(b_1+b_2)}L_1\right)
\le B_1 e^{-L_1B_2\veps^{-b_2}}.
\end{equation}
\end{lemma}

The following lemma controls the fluctuations of $\chi^{\veps}$, $\tilde \chi^{\veps}$ and $\check \chi^{\veps}$. 
\begin{lemma}\label{lem:111}
Fix $\theta \in (3/2, \infty)$, $T  \in (0,\infty)$, $0< \ell<L<\infty$, $\gamma \ge 0$ and $a>0$.
Define $t_k^{\veps} \doteq k\veps^{\theta/2}$, $k \in \NN_0$.
Let
$$F^{\veps}_1(T,\ell, L) \doteq \sup\{ |\chi^{\veps}(t,x) - \chi^{\veps}(t_k^{\veps}, x)|: t_k^{\veps} \le T, t \in [t_k^{\veps}, t_{k+1}^{\veps}], x \in [\ell,L]\},$$
$$F^{\veps}_2(T,\ell, L) \doteq \sup\{ |\tilde \chi^{\veps}(t,x) - \tilde \chi^{\veps}(t_k^{\veps}, x)|: t_k^{\veps} \le T, t \in [t_k^{\veps}, t_{k+1}^{\veps}], x \in [\ell,L]\},
$$
$$F^{\veps}_3(T,\ell, L) \doteq \sup\{ |\check \chi^\veps(t,x) - \check \chi^\veps(t_k^\veps, x)|: t_k^\veps \le T, t \in [t_k^{\veps}, t_{k+1}^{\veps}], x \in [\ell,L]\}.
$$
Then, for every $\upsilon > 0$ and $m \in \NN$,
\begin{equation}\label{eq:i123}
\limsup_{\veps\to 0} \veps^{-m/2} \Pbd^{\gamma}_a(F^N_i(T,L) > \upsilon)= 0, \; i=2,3.
\end{equation}
Furthermore, if $\theta =2$, we also have
\begin{equation}\label{eq:i123b}
\limsup_{\veps\to 0} \veps^{-m/2} \Pbd^{\gamma}_a(F^{\veps}_1(T,L) > \upsilon)= 0. 
\end{equation}

\end{lemma}

\begin{proof}
 Let $T,L, \theta, \gamma , a$ be as in the statement of the lemma. Consider first $F^{\veps}_3$.
Fix $\upsilon>0$, $x\in [\ell,L]$ and $t_k^{\veps} \le T$.
Then, for fixed $\beta \in (0, 1/2)$ 
\begin{multline*}
\Pbd^{\gamma}_a\left( \sup_{t \in [t^{\veps}_k, t^{\veps}_{k+1}]} |\check \chi^{\veps}(t,x) - \check \chi^{\veps}(t_k^{\veps}, x)| > \upsilon\right)
= \Pbd^{\gamma}_a\left( \sup_{t \in [t^{\veps}_k, t^{\veps}_{k+1}]} \veps^{1/4} |Y_{(i_{\veps}(x))}(t) - Y_{(i_{\veps}(x))}(t^{\veps}_k)| > \upsilon\right)\\
= \Pbd^{\gamma}_a\left( \sup_{t \in [0, \veps^{\theta/2}]}\veps^{1/4} |Y_{(i_{\veps}(x))}(t) - Y_{(i_{\veps}(x))}(0)| > \upsilon\right)\\
= \Pbd_a^{\gamma}
\left(\veps^{1/4} Y_{(i_{\veps}(x))}(0)
\sup_{0\le s \le \veps^{\theta/2}} \left|
\exp\{a(X_{(i_{\veps}(x))}(s) - X_{(i_{\veps}(x))}(0)) + a^2s/2\}
-1\right| > \upsilon \right)\\
\le \kappa_3 e^{-\kappa_4\veps^{-\beta/2}},
\end{multline*}
where we have use the stationarity of $\{(Y_i(t+ \cdot))_{i \in \NN_0}, t \ge 0\}$ in the second equality and made use of Lemma \ref{lem:expdec1} in obtaining the last inequality.

By a simple union bound it now follows that
$$\Pbd^{\gamma}_a(F^{\veps}_3(T,\ell, L) > \upsilon) \le (\veps^{-\theta/2}T+1)(\veps^{-1/2}L+1) \kappa_3e^{-\kappa_4\veps^{-\beta/2}}.$$
This proves \eqref{eq:i123} for $i=3$.

Consider now $i=2$. Then 
\begin{align*}
&\Pbd^{\gamma}_a\left( \max_{t^{\veps}_k \le T} \sup_{t \in [t^{\veps}_k, t^{\veps}_{k+1}]}\sup_{x \in [\ell,L]}|\tilde \chi^{\veps}(t,x) - \tilde \chi^{\veps}(t_k^{\veps}, x)| > \upsilon\right)\\
&= \Pbd^{\gamma}_a\left(\max_{t^{\veps}_k \le T} \sup_{t \in [t^{\veps}_k, t^{\veps}_{k+1}]}\sup_{x \in [\ell,L]}\veps^{1/4}|Y_{(I^{\veps}_0(x))}(t) - Y_{(I_0^{\veps}(x))}(t^{\veps}_k)| > \upsilon\right)\\
&\le \Pbd^{\gamma}_a\left( \max_{t^{\veps}_k \le T}\sup_{t \in [t^{\veps}_k, t^{\veps}_{k+1}]}\sup_{x \in [\ell/2,2L]}\veps^{1/4}|Y_{(i_{\veps}(x))}(t) - Y_{(i_{\veps}(x))}(t^{\veps}_k)| > \upsilon\right) +
\Pbd^{\gamma}_a\left(I^{\veps}_0(L) > i_{\veps}(L)\right)
\\
&\quad + \Pbd^{\gamma}_a\left(I^{\veps}_0(\ell) < i_{\veps}(\ell/2)\right)\\
&\le \Pbd^{\gamma}_a\left(F^{\veps}_3(T, \ell/2, 2L) > \upsilon\right) + 
\Pbd^{\gamma}_a\left(I^{\veps}_0(L) > i_{\ell}(2L)\right) + \Pbd^{\gamma}_a\left(I^{\veps}_0(\ell) < i_{\veps}(\ell/2)\right).
\end{align*}
Applying Lemma \ref{lem:PPLD}, we have that, for all  $\veps \in (0, L^2/4)$, for some  $k_0, k_1, k_2 \in (0, \infty)$ 
\begin{multline}\Pbd^{\gamma}_a\left(I^{\veps}_0(L) > i_{\veps}(2L)\right) 
+ \Pbd^{\gamma}_a\left(I^{\veps}_0(\ell) < i_{\veps}(\ell/2)\right)\\
\le 
k_0\left(\left(\Pbd^{0}_a\left(I^{\veps}_0(L) > i_{\veps}(2L)\right)\right)^{1/2}
+ \left(\Pbd^{0}_a\left(I^{\veps}_0(\ell) < i_{\veps}(\ell/2)\right)\right)^{1/2}\right)
\le
k_1 e^{-\veps^{-1/2}k_2}. \label{eq:514}
\end{multline}
Indeed, when $\veps \in (0, L^2/4)$,
\begin{align*}
\Pbd^{0}_a\left(I^{\veps}_0(L) > i_{\veps}(2L)\right) 
&\le \Pbd^{0}_a\left(|I^{\veps}_0(L) - L \veps^{-1/2}| > \frac{L}{2} \veps^{-1/2}  \right),
\end{align*}
and
\begin{align*}
\Pbd^{0}_a\left(I^{\veps}_0(\ell) < i_{\veps}(\ell/2)\right) \le 
\Pbd^{0}_a\left(I^{\veps}_0(\ell) < \ell \veps^{-1/2}/2)\right)
\le \Pbd^{0}_a\left(|I^{\veps}_0(\ell)- \ell \veps^{-1/2}| > \ell \veps^{-1/2}/2)\right).
\end{align*}
The estimate in \eqref{eq:514} now follows on applying Lemma \ref{lem:PPLD}, with  $b_1=0$, $b_2=1/2$, $\gamma=1/2$ and,  first with $L_1=L$ and then with $L_1=\ell$.

Proof of \eqref{eq:i123} for $i=2$ now follows on combining the  estimate in \eqref{eq:514}  with \eqref{eq:i123} for $i=3$ (which was shown previously).

Finally consider $i=1$ and take $\theta =2$.
Consider for $x \in [0,L]$ and $t \in [0, \veps]$
\begin{align}
| \chi^{\veps}(t,x) -  \chi^{\veps}(0, x)|
&=\veps^{1/4} \left|\sum_{i=0}^{\infty} \one_{(0, \veps^{-1/2}x]}(Y_i(t)) - \sum_{i=0}^{\infty} \one_{(0, \veps^{-1/2}x]}(Y_i(0))\right|\nonumber\\
&= \veps^{1/4} \left|\sum_{i=0}^{\infty} \one_{(-\infty, c_{\veps}(x,t)]}(X_i(t)) - \sum_{i=0}^{\infty} \one_{(-\infty, c_{\veps}(x,0)]}(X_i(0))\right|, \label{eq:853}
\end{align}
where
$c_{\veps}(x,t) \doteq -(\log \veps)/(2a) + a^{-1} \log x - at/2$.
Note that
\begin{align}
&\veps^{1/4} \left|\sum_{i=0}^{\infty} \one_{(-\infty, c_{\veps}(x,t)]}(X_i(t)) - \sum_{i=0}^{\infty} \one_{(-\infty, c_{\veps}(x,0)]}(X_i(0))\right|\nonumber\\
&\le \veps^{1/4} \left (\sum_{i <  I^{\veps}_0(x)} \one_{(c_{\veps}(x,t), \infty)}(X_i(t))
+ \sum_{i \ge  I^{\veps}_0(x)} \one_{(-\infty, c_{\veps}(x,t)]}(X_i(t)) \right).
\label{eq:854}
\end{align}
Let, for $i, j \in \NN_0$,
\begin{align*}
H^{\veps,l}(i,j) &= \one \{\inf_{t \in [0, \veps]} X^{l}_i(t) + \frac{a}{2}t \le 
X_{(j)}(0)\}, \; 
H^{\veps,r}(i,j)&= \one \{\sup_{t \in [0, \veps]} X^{r}_i(t) + \frac{a}{2}t \ge 
X_{(j-1)}(0)\},
\end{align*}
where $X^l_i$ and $X^r_i$ are as in \eqref{eq:xrl}.
Let
$$H^{\veps,l}(j) \doteq \sum_{i\ge j} H^{\veps,l}(i,j), \;\;
H^{\veps,r}( j) \doteq \sum_{i<j} H^{\veps,r}(i,j).
$$
Note that
\begin{align*}
\sum_{i \ge I_0^{\veps}(x)} \one_{(-\infty, c_{\veps}(x,t)]}(X_i(t))
&= \sum_{i \ge I_0^{\veps}(x)} \one_{(-\infty, c_{\veps}(x,0)]}(X_i(t) +at/2)\\
&\le \sum_{i \ge I_0^{\veps}(x)} \one_{(-\infty, X_{I_0^{\veps}(x)}(0)]}(X_i(t) +at/2),
\end{align*}
where the last line follows on noting that since 
$I_0^{\veps}(x) = \max\{i \ge 0: X_i(0) \le c_{\veps}(x,0)\}+1$ (with maximum over an empty set taken to be $-1$), $X_{I_0^{\veps}(x)}(0) \ge c_{\veps}(x,0)$.
It thus follows that, for $t \in [0, \veps]$,
\begin{align}
\sum_{i \ge I_0^{\veps}(x)} \one_{(-\infty, c_{\veps}(x,t)]}(X_i(t))
&\le \sum_{i \ge I_0^{\veps}(x)} \one \{\inf_{t \in [0, \veps]} X^{l}_i(t) + \frac{a}{2}t \le 
X_{(I_0^{\veps}(x))}(0)\}\nonumber\\
&=\sum_{i \ge I_0^{\veps}(x)} H^{\veps,l}(i,I_0^{\veps}(x)) = H^{\veps,l}(I_0^{\veps}(x)).
\label{eq:851}
\end{align}
In a similar manner
\begin{multline*}
\sum_{i <  I^{\veps}_0(x)} \one_{(c_{\veps}(x,t), \infty)}(X_i(t))
= \sum_{i <  I^{\veps}_0(x)} \one_{(c_{\veps}(x,0), \infty)}(X_i(t) + at/2)\\ 
\le \sum_{i <  I^{\veps}_0(x)} \one_{(X_{(I_0^{\veps}(x)-1)}(0), \infty)}(X_i(t) + at/2),
\end{multline*}
where the last line uses the observation that, by definition of $I_0^{\veps}(x)$,
$c_{\veps}(x,0) \ge X_{(I_0^{\veps}(x)-1)}(0)$.
It now follows that, for $t \in [0, \veps]$,
\begin{align}
\sum_{i <  I^{\veps}_0(x)} \one_{(c_{\veps}(x,t), \infty)}(X_i(t))
&\le \sum_{i <  I^{\veps}_0(x)} \one \{\sup_{t \in [0, \veps]} X^{r}_i(t) + \frac{a}{2}t \ge 
X_{(I^{\veps}_0(x)-1)}(0)\}\nonumber\\
&= \sum_{i <  I^{\veps}_0(x)} H^{\veps,r}(i,I^{\veps}_0(x))
= H^{\veps,r}( I^{\veps}_0(x)). \label{eq:852}
\end{align}
Combining \eqref{eq:851} and \eqref{eq:852} with \eqref{eq:853} and \eqref{eq:854} we obtain that
\begin{equation}\label{eq:in3}
\sup_{t \in [0, \veps]}  | \chi^{\veps}(t,x) -  \chi^{\veps}(0, x)|
\le \veps^{1/4} H^{\veps,l}(I^{\veps}_0(x)) +  \veps^{1/4} H^{\veps,r}(I^{\veps}_0(x)).
\end{equation}
Let $A^{\veps} \doteq \{ I^{\veps}_0(\ell) \ge \ell \veps^{-1/2}/2 \mbox{ and } I^{\veps}_0(L) \le 2L\veps^{-1/2}\}$.
Then we can find $k_3, k_4 \in (0, \infty)$ such that
\begin{equation}\label{eq:in1}
\Pbd^{\gamma}_a((A^{\veps})^c) \le k_3 e^{-\veps^{-1/2} k_4}.
\end{equation}
Next, for $\upsilon>0$, $j \in [\frac{\ell}{2}\veps^{-1/2}, 2L\veps^{-1/2}]$ and $q \in \NN$,
\begin{align}\label{eq:1245n}
\Pbd^{\gamma}_a(\veps^{1/4} H^{\veps,l}(j) >\upsilon)
&\le \veps^{q/4} \upsilon^{-q} \Ebd^{\gamma}_a|H^{\veps,l}(j)|^q.
\end{align}
Also, for $i\ge j$
\begin{multline*}
\Ebd^{\gamma}_a|H^{\veps,l}(i,j)|^q
= \Pbd^{\gamma}_a\left(\inf_{t \in [0, \veps]} (X^{l}_i(t) + \frac{a}{2}t) \le 
X_{(j)}(0)\right)\\
= \Pbd^{\gamma}_a\left(\inf_{t \in [0, \veps]} (W_i(t) + \frac{a}{2}t) \le 
X_{(j)}(0) - X_{(i)}(0)\right)
\le \Pbd^{\gamma}_a\left(|W_i(\veps)|
\ge X_{(i)}(0) - X_{(j)}(0) \right).
\end{multline*}
By scaling properties of the Brownian motion, independence of $W_i$ from $\{X_j(0)\}_{j\in \NN_0}$, and Markov's inequality, we have now, for any $\beta>0$
\begin{align*}
\Ebd^{\gamma}_a|H^{\veps,l}(i,j)|^q
\le E(e^{\beta|W_1(1)|}) 
\Ebd^{\gamma}_a\left(e^{-\beta \veps^{-1/2}(X_{(i)}(0) - X_{(j)}(0))}\right).
\end{align*}
Choose $\beta \in (0,1)$ sufficiently small so that, for all $\veps\in (0,1)$,
$$\beta \veps^{-1/2} < 2\gamma + \veps^{-1/2} a\ell/2.
$$
Then we claim that
\begin{align}
\Ebd^{\gamma}_a\left(e^{-\beta \veps^{-1/2}(X_{(i)}(0) - X_{(j)}(0))}\right)\le
\left(\frac{2\gamma +(i+1)a}{2\gamma+(j+1)a}\right)^{-\beta \veps^{-1/2}/2a}.\label{claim1001}
\end{align}
To see this, note that the left side, denoted as
$A_{ij}$, equals $\prod_{k=j+1}^i \frac{2\gamma+ka}{2\gamma + ka + \sigma}$ where $\sigma = \beta\veps^{-1/2}$ and the product is taken to be $1$ when $i=j$. Thus, using the inequality
$\log (1+x) \ge x/2$ for $x \in [0,1]$,
\begin{align*}
\log A_{ij} \le -\frac{\sigma}{2} \sum_{k=j+1}^i \frac{1}{2\gamma+ka} 
\le -\frac{\sigma}{2a} (\log(2\gamma + (i+1) a) - \log(2\gamma + (j+1) a)). 
\end{align*}
The claim now follows readily from the above estimate.

From the estimate in \eqref{claim1001} we see that there is a $\veps_0 \in (0,1)$ and $d_1 \in (0, \infty)$ such that, for all $\veps \in (0,\veps_0)$, using Minkowski's inequality,
\begin{multline}
\left(\Ebd^{\gamma}_a|H^{\veps,l}(j)|^q\right)^{1/q}
\le \sum_{i\ge j} \left(\Ebd^{\gamma}_a|H^{\veps,l}(i,j)|^q\right)^{1/q}\\ 
\le 
\sum_{i\ge j}
[\Ebd^{\gamma}_a(e^{\beta|W_1(1)|})]^{1/q} 
[\Ebd^{\gamma}_a\left(e^{-\beta \veps^{-1/2}(X_{(i)}(0) - X_{(j)}(0))}\right)]^{1/q}
\le d_1 \frac{2\gamma +ja}{\beta\veps^{-1/2} -2aq}.\label{eq:1014}
\end{multline}
To see the last inequality, note that
\begin{align*}
\sum_{i=j}^{\infty} \left(\frac{2\gamma +(i+1)a}{2\gamma+(j+1)a}\right)^{-\beta \veps^{-1/2}/2aq}
&= (2\gamma+(j+1)a)^{\beta \veps^{-1/2}/2aq} 
\sum_{i=j}^{\infty} (2\gamma +(i+1)a)^{-\beta \veps^{-1/2}/2aq}\\
&\le 
\left(\frac{2\gamma+(j+1)a}{2\gamma+ja} \right)^{\beta \veps^{-1/2}/2aq}
\frac{2q(2\gamma +ja)}{\beta \veps^{-1/2} - 2aq}.
\end{align*}
This, together with the fact that
$$
\limsup_{\eps \downarrow 0}\sup_{j \ge \ell \veps^{-1/2}/2}\left(\frac{2\gamma+(j+1)a}{2\gamma+ja} \right)^{\beta \veps^{-1/2}/2aq} <\infty
$$
gives the inequality in \eqref{eq:1014}.

Thus, there is a $\veps_1\in (0, \veps_0)$ such that
$\sup_{\veps \in (0,\veps_1)}\sup_{j \in [\frac{\ell}{2}\veps^{-1/2}, 2L\veps^{-1/2}]}
\Ebd^{\gamma}_a|H^{\veps,l}(j)|^q \doteq d_2<\infty$.
From \eqref{eq:1245n}
it now follows that, for all $\upsilon>0$ and $q \in \NN$, 
\begin{equation}
\limsup_{\veps\to 0}\veps^{-q/8}\left(\sup_{j \in [\ell \veps^{-1/2}/2, 2L\veps^{-1/2}]} \Pbd^{\gamma}_a(\veps^{1/4} H^{\veps,l}(j) >\upsilon)\right) =0.\label{eq:in2}
\end{equation}
Now, for $\upsilon>0$, $j \in [\frac{\ell}{2}\veps^{-1/2}, 2L\veps^{-1/2}]$ and $q \in \NN$
\begin{align}
\Pbd^{\gamma}_a(\veps^{1/4} H^{\veps,r}(j) >\upsilon)
&\le \veps^{q/4} \upsilon^{-q} \Ebd^{\gamma}_a| H^{\veps, r}(j)|^q.
\label{eq:1001}
\end{align}
Also, for $0 \le i<j$
\begin{align*}
\Ebd^{\gamma}_a| H^{\veps, r}(i,j)|^q
&= \Pbd^{\gamma}_a(\sup_{t \in [0, \veps]} (X^r_i(t) + at/2) > X_{(j-1)}(0))\\
&= \Pbd^{\gamma}_a(\sup_{t \in [0, \veps]} (W_i(t) + (\gamma +a/2)t) > X_{(j-1)}(0) - X_{(i)}(0))\\
&\le \Pbd^{\gamma}_a(|W_1(1)| > \veps^{-1/2} (X_{(j-1)}(0) - X_{(i)}(0)
- (\gamma +a/2)\veps))\\
&\le \Ebd^{\gamma}_a e^{|W(1)|} e^{\veps^{1/2}(\gamma+ a/2)}
\Ebd^{\gamma}_a e^{- \veps^{-1/2}(X_{(j-1)}(0) - X_{(i)}(0))}\\
&= \Ebd^{\gamma}_a e^{|W(1)|} e^{\veps^{1/2}(\gamma+ a/2)}
\prod_{k=i+1}^{j-1}\Ebd^{\gamma}_a e^{- \veps^{-1/2}(X_{(k)}(0) - X_{(k-1)}(0))}.
\end{align*}
We can find $\veps_2 \in (0, 1)$ and $\varsigma \in (0, \infty)$
such that for all $\veps \in (0,\veps_2)$ and $0 \le k \le j \le 2L\veps^{-1/2}$
$$\veps^{-1/2}(X_{(k)}(0) - X_{(k-1)}(0)) \ge_d U \mbox{ where } U \sim \mbox{Exp}(\varsigma).$$
Letting $b = Ee^{-U}$, we then have for some $\veps_3 \in (0,\veps_2)$ and $d_3 \in (0, \infty)$, for all $\veps \in (0,\veps_3)$
$$\left(\Ebd^{\gamma}_a| H^{\veps, r}(i,j)|^q\right)^{1/q}
\le d_3 b^{(j-i-1)/q}.
$$
Then, 
\begin{align*}
\left(\Ebd^{\gamma}_a| H^{\veps,r}(j)|^q\right)^{1/q}
\le \sum_{0\le i< j} \left(\Ebd^{\gamma}_a|H^{\veps,r}(i,j)|^q\right)^{1/q}
\le d_3 \sum_{0\le i< j} b^{(j-i-1)/q} 
\le (1 - b^{1/q})^{-1}.
\end{align*}
It then follows from \eqref{eq:1001} that
for all $\upsilon>0$ and $q \in \NN$, 
\begin{equation}\label{eq:in4}
\limsup_{\veps\to 0}\veps^{-q/8}\left(\sup_{j \in [\ell \veps^{-1/2}/2, 2L\veps^{-1/2}]} \Pbd^{\gamma}_a(\veps^{1/4} H^{\veps,r}(j) >\upsilon)\right) =0.
\end{equation}
The statement in \eqref{eq:i123b} now follows on using the stationarity of $\chi^\veps(t+\cdot,x)$, the estimates in
\eqref{eq:in3},\eqref{eq:in1}, \eqref{eq:in2},  \eqref{eq:in4}, recalling that $q \in \NN$ is arbitrary,
and an application of the union bound.
\end{proof}

We now return to the proof of Proposition \ref{prop:yaxis}.

{\bf Proof of Proposition \ref{prop:yaxis}.}
We begin by showing \eqref{eq:show2}. For $t \in [0,T]$ and $x \in [0, L]$,
 we have from \eqref{eq:eq410} that
$$|\tilde \chi^\veps(t,x) - \check \chi^\veps(t,x)| =
\veps^{1/4} |\cld(I^\veps_0(x), i_\veps(x), t)|.$$

Recalling, for each $t\ge 0$, $\{(Y_{(i)}(t) - Y_{(i-1)}(t)\}_{i\ge 1}$ are iid mean $1$ Exponential random variables,
it follows that, for every $m>1$, there is a $\tilde C_m \in (0, \infty)$ such that
for all $j< j'$
\begin{align*}
\Ebd^{\gamma}_a|\cld(j, j', t)|^m = \Ebd^{\gamma}_a|(j-j') - (Y_{(j)}(t) - Y_{(j')}(t))|^m \le \tilde C_m |j-j'|^{m/2}.
\end{align*}
Let $\theta =2$ and $t_k^\veps \ge 0$ be defined as in Lemma \ref{lem:111}.
Fix $\upsilon >0$ and $\varsigma \in (1/2,1)$.
Then, for $m \in \NN$, $t^\veps_k \le T$, $j, j' \le 2(L+1)\veps^{-1/2}$, $|j-j'| \le \veps^{-\varsigma/2}L$,
$$
\Pbd^{\gamma}_a\left(\veps^{1/4} |\cld(j, j', t)| > \upsilon\right)
\le \upsilon^{-m} \veps^{m/4}\Ebd^{\gamma}_a|\cld(j, j', t)|^m
\le \tilde C_m\upsilon^{-m} \veps^{m/4} \veps^{-m\varsigma/4}L^{m/2}.
$$
By a union bound it then follows that
\begin{align}
&\Pbd^{\gamma}_a\left(\veps^{1/4} |\cld(j, j', t_k^{\veps})| > \upsilon
\mbox{ for some } t_k^\veps \le T,\; j,j' \le 2(L+1)\veps^{-1/2}, \;|j-j'| \le \veps^{-\varsigma/2}L\right)\nonumber\\
&\qquad \le \tilde C_m L^{m/2} \upsilon^{-m} \veps^{\frac{m}{4} (1-\varsigma)} (T \veps^{-1}+1) (4(L+1)^2\veps^{-1}).\label{eq:est1}
\end{align}
Taking $m$ to be suitably large we see that the last expression converges to $0$ as $\veps\to 0$. 

Next,
\begin{multline}
\Pbd^{\gamma}_a\left(\sup_{t \in [0,T]} \sup_{x \in [\ell, L]}
|\tilde \chi^\veps(t,x) - \check\chi^\veps(t,x)| > 2\upsilon\right)\nonumber
\le 
\Pbd^{\gamma}_a\left(\max_{t_k^\veps \le T} \sup_{x \in [\ell, L]}
|\tilde \chi^\veps(t_k^\veps,x) - \check\chi^\veps(t_k^\veps,x)| > \upsilon\right)\nonumber\\
 + \Pbd^{\gamma}_a\left(F^\veps_3(T,\ell,L) > \upsilon/2\right) +
\Pbd^{\gamma}_a\left(F^\veps_2(T, \ell,L) > \upsilon/2\right). \label{eq:est2}
\end{multline}
Also,
\begin{align}
&\Pbd^{\gamma}_a\left(\max_{t_k^\veps\le T} \sup_{x \in [\ell, L]}
|\tilde \chi^\veps(t_k^\veps,x) - \check\chi^\veps(t_k^\veps,x)| > \upsilon\right)\nonumber\\
& = \Pbd^{\gamma}_a\left(\max_{t_k^\veps \le T} \sup_{x \in [\ell, L]}
\veps^{1/4} |\cld(I^\veps_0(x), i_\veps(x), t_k^\veps)| > \upsilon\right)\nonumber\\
&\le \Pbd^{\gamma}_a\left(\veps^{1/4} |\cld(j, j', t_k^\veps)| > \upsilon
\mbox{ for some } t_k^\veps \le T,\; j,j' \le 2(L+1)\veps^{-1/2}, \;|j-j'| \le \veps^{-\varsigma/2}L\right)\nonumber\\
& \quad + \Pbd^{\gamma}_a\left(I^\veps_0(L) > 2(L+1) \veps^{-1/2}\right)
+ \Pbd^{\gamma}_a\left(\sup_{x \in [\ell, L]} |I^\veps_0(x) - i_\veps(x)| > L\veps^{-\varsigma/2}\right).\\ \label{eq:est3}
\end{align}

Applying the estimate \eqref{eq:ppldp} in Lemma \ref{lem:PPLD}, as in \eqref{eq:in1}, we 
can find $d_1, d_2 \in (0, \infty)$ such that
\begin{align}
\Pbd^{\gamma}_a\left(I^\veps_0(L) > 2(L+1) \veps^{-1/2}\right)
\le\Pbd^{\gamma}_a\left( |I^\veps_0(L) - L\veps^{-1/2}| > (L+2)\veps^{-1/2}\right)
\le d_1e^{-\veps^{-1/2} d_2}.\\ \label{eq:est4}
\end{align}
Also, recalling that $\varsigma \in (1/2, 1)$, we have on applying
\eqref{eq:ppldp} with $L_1 = L, \gamma = 1/2$, $b_1= (1-\varsigma)/2$
and $b_2 = (\varsigma-1/2)$, that for some $\tilde d_3, d_3, d_4 \in (0, \infty)$,
and $\cln$ a standard rate $1$ Poisson process,
\begin{align}
 \Pbd^{\gamma}_a\left(\sup_{x \in [\ell, L]} |I^\veps_0(x) - i_\veps(x)| > \veps^{-\varsigma/2}\right) &\le
\tilde d_3 \left(\Pbd^{0}_a\left(\sup_{x \in [\ell, L]} |I^\veps_0(x) - i_\veps(x)| > \veps^{-\varsigma/2}\right)\right)^{1/2}\\
&\le
 \tilde d_3\left(P(\sup_{0 \le s \le \veps^{-1/2}L} |\cln(s) -s| > \frac{1}{2} \veps^{-\varsigma/2})\right)^{1/2}\nonumber\\
 &\le d_3 e^{ - d_4\veps^{-(2\varsigma-1)/2} }. \label{eq:est5}
\end{align}
Combining Lemma \ref{lem:111}, and the estimates in \eqref{eq:est1}, 
\eqref{eq:est2}, \eqref{eq:est3}, \eqref{eq:est4}, \eqref{eq:est5}, we have the convergence stated in \eqref{eq:show2}.

We now consider the statement in \eqref{eq:show1}, 
Note that 
\begin{multline*}
\Pbd^{\gamma}_a\left(\sup_{t \in [0,T]} \sup_{x \in [\ell, L]}
| \chi^\veps(t,x) - \tilde\chi^\veps(t,x)| > 2\upsilon\right)
\le 
\Pbd^{\gamma}_a\left(\max_{t_k^\veps \le T} \sup_{x \in [\ell, L]}
| \chi^\veps(t_k^\veps,x) - \tilde\chi^\veps(t_k^\veps,x)| > \upsilon\right)\\
\quad + \Pbd^{\gamma}_a\left(F^\veps_1(L) > \upsilon/2\right) +
\Pbd^{\gamma}_a\left(F^\veps_2(L) > \upsilon/2\right). 
\end{multline*}
Furthermore, using \eqref{eq:1253},
\begin{align}
&\Pbd^{\gamma}_a\left(\max_{t_k^\veps \le T} \sup_{x \in [\ell, L]}
| \chi^\veps(t_k^\veps,x) - \tilde\chi^\veps(t_k^\veps,x)| > \upsilon\right)\nonumber\\
& = \Pbd^{\gamma}_a\left(\max_{t_k^\veps \le T} \sup_{x \in [\ell, L]}
|\veps^{1/4} \cld(I^\veps_{t_k^\veps}(x), I^\veps_0(x), t_k^\veps) + 
\veps^{1/4} \rho^\veps_{t_k^\veps}(x)|  > \upsilon\right)\nonumber\\
&\le \Pbd^{\gamma}_a\left(\veps^{1/4} |\cld(j, j', t_k^\veps)| > \upsilon/2
\mbox{ for some } t_k^\veps \le T,\; j,j' \le 2(L+1)\veps^{-1/2}, \;|j-j'| \le \veps^{-\varsigma/2}L\right)\nonumber\\
& \quad + \Pbd^{\gamma}_a\left(I^\veps_0(L) > 2(L+1) \veps^{-1/2}\right)
+ \Pbd^{\gamma}_a\left(\max_{t_k^\veps \le T}I^\veps_{t_k^\veps}(L) > 2(L+1) \veps^{-1/2}\right)\\
&\quad + \Pbd^{\gamma}_a\left(\veps^{1/4} \max_{t_k^\veps \le T} \sup_{x \in [\ell, L]} |\rho^\veps_{t_k^\veps}(x)|  > \upsilon/2\right) \nonumber\\
&\quad + \Pbd^{\gamma}_a\left(\sup_{x \in [\ell, L]} |I^\veps_0(x) - i_\veps(x)| > \frac{1}{2}\veps^{-\varsigma/2}\right)
+ \Pbd^{\gamma}_a\left(\sup_{x \in [\ell, L]} \max_{t_k^\veps\le T}|I^\veps_{t_k^\veps}(x) - i_\veps(x)| > \frac{1}{2}\veps^{-\varsigma/2}\right). \label{eq:est3'}
\end{align}
All terms except the fourth term on the right side are estimated as before (with an additional union bound).
{For the fourth term we will use the fact that, under $\Pbd_a^0$, $\rho^\veps_t(x)$ is distributed as $Exp(1)$ for any $t \ge 0,\, x>0$.
Fix $\beta > 1/2$ and define for $j \in \NN_0$, $x^\veps_j 
\doteq j\veps^{\beta/2}$. Note that, if $x \in [x^\veps_j, x^\veps_{j+1}]$ for some $j \in \NN_0$, then for any $t \ge 0$,
$$
\left|\rho^\veps_t(x) - \rho^\veps_t(x^\veps_j)\right| \le \left|\rho^\veps_t(x^\veps_{j+1}) - \rho^\veps_t(x^\veps_j)\right| + 2 \veps^{-1/2}|x^\veps_{j+1} - x^\veps_j| \le \left|\rho^\veps_t(x^\veps_{j+1})\right| + \left|\rho^\veps_t(x^\veps_j)\right| + 2\veps^{(\beta - 1)/2}.
$$
Combining these observations, we obtain positive $\veps_0,d_5$ such that for $\veps \in (0, \veps_0)$,
\begin{align*}
\Pbd^{\gamma}_a\left(\veps^{1/4} \max_{t_k^\veps \le T} \sup_{x \in [\ell, L]} |\rho^\veps_{t_k^\veps}(x)|  > \upsilon/2\right) 
&\le d_5\left(\Pbd^0_a\left(\veps^{1/4} \max_{t_k^\veps \le T} \sup_{x \in [\ell, L]} |\rho^\veps_{t_k^\veps}(x)|  > \upsilon/2\right)\right)^{1/2}\\
&\le d_5\left(\Pbd^0_a\left(\veps^{1/4} \max_{t_k^\veps \le T} \max_{x_j^\veps \le L} |\rho^\veps_{t_k^\veps}(x_j^\veps)|  > \upsilon/8\right)\right)^{1/2}\\
&\le d_5(\veps^{-\beta/2}L+1)^{1/2}(\veps^{-1}T+1)^{1/2}  e^{-\upsilon \veps^{-1/4}/16}.
\end{align*}
This gives the desired estimate on the fourth term and completes the proof  in \eqref{eq:show1}.}


\subsection{Proof of Corollary \ref{cor:fixxprocess}} \label{ssec:corproof} 

\begin{proof} By Theorem \ref{thm:main2}, as $\veps \to 0$, under $\Pbd^{\gamma}_a$,
$$ \veps^{-1/4}\left(X^{\veps}_{(i_{\veps}(x))}(\cdot) - X^{\veps}_{(i_{\veps}(x))}(0)\right) \Rightarrow \frac{1}{ax} (u(\cdot,x) - u(0,x)),$$
where $u$ may be expressed as 
\begin{equation}
u(t,x) = \hat\clw(t,x) + \hat\clm(t,x), \quad (t,x) \in [0,\infty) \times (0,\infty).
\end{equation}
Here the pair of processes $(\hat\clw,\hat\clm)$ is defined as in Theorem \ref{thm:formain1}. Since $u$ is Gaussian, it is enough to prove that the process $u(\cdot,x) - u(0,x)$ has the correct covariance structure for each $x \in (0,\infty)$. By independence of the white noise measures $B$ and $W$, for any $t, t' \in [0,\infty)$, 
\begin{align} 
& \Ebd_a^\gamma[(u(t,x) - u(0,x))(u(t',x) - u(0,x))] \\
& = \Ebd_a^\gamma[\hat\clw(t',x)\hat\clw(t,x)] - \Ebd_a^\gamma[\hat\clw(0,x)\hat\clw(t,x)] - \Ebd_a^\gamma[\hat\clw(t',x)\hat\clw(0,x)] + \Ebd_a^\gamma[\hat\clw(0,x)^2] \\
& \hspace{1in} + \Ebd_a^\gamma[\hat\clm(t',x)\hat\clm(t,x)], \label{eq:ucov}
\end{align} 
noting that $\hat\clm(0,x) = 0$ a.s. We sketch calculation of each term below. We may assume without lost of generality that $t, t' > 0$. We have 
\begin{align} 
\Ebd_a^\gamma[\hat\clw(t',x)\hat\clw(t,x)] & = \int_0^\infty \Psi^x(t,y)\Psi^x(t',y)dy \\
& = - \int_0^\infty \left[y \partial_y\Psi^x(t,y) \cdot \Psi^x(t',y) + y \partial_y\Psi^x(t',y) \cdot \Psi^x(t,y) \right] dy \label{eq:wcov}
\end{align}
by integration by parts. By \eqref{eq:Psixexp} and \eqref{eq:ydPsi}, 
\begin{align}
-\int_0^\infty y \partial_y\Psi^x(t,y) \cdot \Psi^x(t',y) dy &= \frac{1}{a} \int_0^\infty p_t( a^{-1}(\log x - \log y- \frac{a^2t}{2})) \\
&\quad \times\int_{-\infty}^{a^{-1}(\log x - \log y- \frac{a^2t}{2})} p_{t'}(w)dw.
\end{align}
By a change of variables, the right-hand side may be seen to be equal to 
\begin{align}
x \int_{-\infty}^\infty \int_{-\infty}^{z - \frac{at}{2} - \frac{at'}{2}} p_t(z)p_{t'}(w) dw dz &= x \int_{-\infty}^\infty \int_{-\infty}^{\xi(z, t, t')} p_1(z)p_1(w)dw dz\\
& = x \int_{-\infty}^{-\frac{a}{2}(t + t')^{1/2}} p_1(z)dz. 
\end{align}
where $\xi(z,t,t') \doteq (\frac{t}{t'})^{1/2}z - \frac{a}{2}(t')^{-1/2}(t + t')$ and the last equality may be obtained by rotational symmetry of the Gaussian distribution $p_1(z)p_1(w)dw dz$. Computing similarly $-\int_0^\infty y \partial_y\Psi^x(t',y) \cdot \Psi^x(t,y) dy$, we conclude from \eqref{eq:wcov} that 
\begin{equation}
\Ebd_a^\gamma[\hat\clw(t',x)\hat\clw(t,x)] = 2x \int_{-\infty}^{-\frac{a}{2}(t + t')^{1/2}} p_1(z)dz = 2x \left(1 - \Phi\left(\frac{a}{2}(t + t')^{1/2}\right)\right). \label{eq:cov1}
\end{equation}
Using dominated convergence as $t \to 0$ and $t' \to 0$, one may obtain as well that 
\begin{align}
& \quad\quad\quad \Ebd_a^\gamma[\hat\clw(0,x)\hat\clw(t,x)] = 2x  \left(1 - \Phi\left(\frac{a}{2}t^{1/2}\right)\right), \label{eq:cov2} \\
& \Ebd_a^\gamma[\hat\clw(t',x)\hat\clw(0,x)] = 2x  \left(1 - \Phi\left(\frac{a}{2}(t')^{1/2}\right)\right), \quad \text{ and } \quad \Ebd_a^\gamma[\hat\clw(0,x)^2] = x. \label{eq:cov3}
\end{align}

We compute the last term in \eqref{eq:ucov} as follows:
\begin{align}
\Ebd_a^\gamma[\hat\clm(t',x)\hat\clm(t,x)] = a^2 \int_0^{t \wedge t'} \int_0^\infty y^2 q_{t - s}(y,x)q_{t' - s}(y,x) dy ds.
\end{align}
For $\sigma, \sigma' \in (0,\infty)$, we have
\begin{align}
a^2 \int_0^\infty y^2 q_\sigma(y,x)q_{\sigma'}(y,x) dy & = \int_0^\infty p_{\sigma}\left( \frac{\log y}{a} - \frac{\log x}{a} + \frac{a\sigma}{2} \right)p_{\sigma'}\left( \frac{\log y}{a} - \frac{\log x}{a} + \frac{a\sigma'}{2} \right)dy \\
& = ax \int_{-\infty}^\infty p_{\sigma}(w)e^{-\frac{aw}{2} - \frac{a^2 \sigma}{8}} p_{\sigma'}(w)e^{-\frac{aw}{2} - \frac{a^2 \sigma'}{8}} e^{aw} dw \\
& = ax e^{-\frac{a^2}{8}(\sigma + \sigma')} \int_{-\infty}^\infty p_{\sigma}(w)p_{\sigma'}(w) dw = \frac{ax}{\sqrt{2\pi}} \frac{e^{-\frac{a^2}{8}(\sigma + \sigma')}}{(\sigma + \sigma')^{1/2}},
\end{align}
where the second equality follows by a change of variables and last by noting that $p_{\sigma}(w)p_{\sigma'}(w) = (2\pi)^{-1/2}(\sigma + \sigma')^{-1/2}p_{(\sigma^{-1} + (\sigma')^{-1})^{-1}}(w)$. From this, one may show that 
\begin{align}
\Ebd_a^\gamma[\hat\clm(t',x)\hat\clm(t,x)] & = \frac{ax}{\sqrt{2\pi}} \int_0^{t \wedge t'} \frac{e^{-\frac{a^2}{8}(t + t' - 2s)}}{(t + t' - 2s)^{1/2}}ds \\
& = 2x \int_{\frac{a}{2}|t - t'|^{1/2}}^{\frac{a}{2}(t + t')^{1/2}} p_1(w) dw = 2x\left( \Phi\left(\frac{a}{2}(t + t')^{1/2}\right) - \Phi\left(\frac{a}{2}|t' - t|^{1/2}\right) \right). \label{eq:cov4}
\end{align}
Combining  \eqref{eq:cov1}, \eqref{eq:cov2}, \eqref{eq:cov3}, and \eqref{eq:cov4}, we obtain that the quantity in \eqref{eq:ucov} equals, 
\begin{align}
2x\left(\Phi\left(\frac{a}{2}t^{1/2}\right) + \Phi\left(\frac{a}{2}(t')^{1/2}\right) - \Phi\left(\frac{a}{2}|t' - t|^{1/2}\right) - \frac{1}{2}\right),
\end{align}
which implies the desired covariance formula.
\end{proof}

\section{Proof of Lowest Particle Asymptotics}

\begin{proof}[Proof of Theorem \ref{thm:lpasymptotics}]
Let $\tilde{X}_i(t) = X_i(t) + \frac{at}{2}$ and $\tilde{X}_{(i)}(t) = X_{(i)}(t) + \frac{at}{2}$.
Letting $F_\eta(x) = \int_{-\infty}^x f_\eta(z)dz$, it suffices to show that, for all $x, y \in \RR$,
\begin{equation}
    \Pbd_a^\gamma(\tilde{X}_{(0)}(0) \leq x, \tilde{X}_{(0)}(t) \leq y) \to F_\eta(x)F_\eta(y) \text{ as } t \to \infty.
\end{equation}

For $x \in \mathbb{R}$, let $J_0(x) \doteq \min\{i \in \mathbb{N}_0 : X_i(0) > x\}$. For each $T \in (0,\infty)$ and $M \in \mathbb{R}$, we let
$$
A_{T,M} \doteq \{\text{for all } t \geq T, X_i(t) \neq X_{(0)}(t) \text{ for } 0 \leq i < J_0(M)\}.
$$
(We set $A_{T,M} = \Omega$ if $J_0(M) = 0$.)

\textit{Step 1.} We will show that $\lim_{M \to \infty}\lim_{T \to \infty}\Pbd_a^\gamma(A_{T,M}) = 1$. First, notice that 
\begin{align}
A_{T,M} \supset B_{T,M}^0 \cap B_{T,M}^1,
\end{align}
where 
$$
B_{T,M}^0 \doteq \Big\{\text{for all } t \geq T, \inf_{0 \leq i < J_0(M)} X_i(t) > -\frac{at}{4}\Big\}, \; B_{T,M}^1 \doteq \Big\{\text{for all } t \geq T, \inf_{i \geq J_0(M)} X_i(t) \leq -\frac{at}{4} \Big\}.
$$
It is sufficient to show that $\lim_{M \to \infty}\lim_{T \to \infty}\Pbd_a^\gamma(B_{T,M}^0) = 1$ and $\lim_{M \to \infty} \lim_{T \to \infty}\Pbd_a^\gamma(B_{T,M}^1) = 1$. For the first limit, we use union bound to write  
\begin{align}
\Pbd_a^\gamma((B_{T,M}^0)^c) & \leq \left\{\text{for some $0 \leq i < J_0(M)$ and some $t \in [T,\infty)$}, X_i(t) \le -\frac{at}{4} \right\} \\
& \leq \Pbd_a^\gamma(J_0(M) \geq e^{2aM}) + \sum_{i = 0}^{\lfloor e^{2aM} \rfloor} \Pbd_a^\gamma(\text{for some } t \in [T,\infty), X_i(t) \le -\frac{at}{4}). \label{eq:B0cbd}
\end{align}
To estimate the terms of the sum, note that for each $i$, $X_i(t) \geq X_i(0) + W_i(t)$. Using this, we obtain for $R > 0$, $i \in \NN_0$,
\begin{align}
& \Pbd_a^\gamma(\text{for some } t \in [T,\infty), X_i(t) \le -\frac{at}{4}) \leq \Pbd_a^\gamma(\text{for some } t \in [T,\infty), X_i(0) + W_i(t) \le -\frac{at}{4}) \\
& \vspace{1in}\leq \Pbd_a^\gamma(X_0(0) < -R) + \Pbd_a^\gamma(\text{for some } t \in [T,\infty), -R + W_i(t) \le -\frac{at}{4}).
\end{align}
The second term goes to zero as $T \to \infty$, while the first term goes to zero as $R \to \infty$ by the a.s. finiteness of  $X_0(0)$. Thus the sum in \eqref{eq:B0cbd} converges to zero as $T \to \infty$ for any fixed $M$. To see that the first term in \eqref{eq:B0cbd} converges to zero as $M \to \infty$, we estimate
\begin{align}
\Pbd_a^\gamma(J_0(M) > e^{2aM}) \leq \Pbd_a^\gamma(Y_{(\lfloor e^{2aM} \rfloor)}(0) \leq e^{aM}) \leq e^{aM} \Ebd_a^\gamma[Y_{(\lfloor e^{2aM} \rfloor)}(0)^{-1}] \leq c_1 e^{-aM},\\ \label{eq:Ibigbd}
\end{align}
where the last bound is obtained using the \ab{Gamma} distribution of the random variable $Y_{(\lfloor e^{2aM},  \rfloor)}(0)$, under $\Pbd^{0}_a$, using a calculation similar to \eqref{eq:initmoment}. We conclude that $\lim_{M \to \infty}\lim_{T \to \infty}\Pbd^{\gamma}_a((B_{T,M}^0)^c) = 0$.

To prove the second limit, we proceed as follows. 
Let 
$
Y^{M,*}(t) \doteq \inf_{i \geq J_0(M)} e^{a(X_i(t) + \frac{at}{2})}
$, for $t \in [0,\infty)$.
Observe that 
$$
(B_{T,M}^1)^c = \{Y^{M,*}(t) > e^{a^2t/4} \text{ for some $t \geq T$}\} \subset \bigcup_{n = \lfloor T \rfloor}^\infty \left\{\sup_{t \in [n,n+1]} Y^{M,*}(t) > e^{a^2n/4}\right\}.
$$
Hence, applying a union bound and Markov's inequality,
\begin{multline}
\Pbd_a^\gamma((B_{T,M}^1)^c)  \leq \Pbd_a^\gamma(J_0(M) > e^{2aM}) \\
+ \sum_{n = \lfloor T \rfloor}^\infty \Pbd_a^\gamma\left(\{J_0(M) \leq e^{2aM}\} \cap \left\{\sup_{t \in [n,n+1]} Y^{M,*}(t) > e^{a^2n/4}\right\} \right) \\
 \leq \Pbd_a^\gamma(J_0(M) > e^{2aM}) + \sum_{n = \lfloor T \rfloor}^\infty e^{-a^2n/4}\Ebd_a^\gamma\Big[ \mathbf{1}\{J_0(M) \leq e^{2aM}\} \sup_{t \in [n,n+1]} Y^{M,*}(t) \Big]. \label{eq:notYlow}
\end{multline}
To conclude that, for each fixed $M$, the second term goes to zero as $T \to \infty$, it is enough to show that the expectation inside the sum is bounded uniformly in $n$. Observe that for all $t$,
$$
Y^{M,*}(t) = \inf_{i \geq J_0(M)} Y_i(t) \leq Y_{(J_0(M))}(t).
$$
Hence, 
\begin{multline}
\Ebd_a^\gamma\left[ \mathbf{1}\{J_0(M) \leq e^{2aM}\} \sup_{t \in [n,n+1]} Y^{M,*}(t) \right]  \leq \Ebd_a^\gamma\left[ \sup_{t \in [n,n+1]} Y_{(\lfloor e^{2aM} \rfloor)}(t) \right]\\
 = \Ebd_a^\gamma\left[ \sup_{t \in [0,1]}Y_{(\lfloor e^{2aM} \rfloor)}(t) \right] 
 \leq \Ebd_a^\gamma\left[ \sup_{t \in [0,1]}\max\{Y_{i}(t), 0 \leq i \leq e^{2aM}\}\right] \\
  \leq \sum_{0 \leq i \leq e^{2aM}} \Ebd_a^\gamma\left[\sup_{t \in [0,1]} Y_i(t) \right], \label{eq:expbd}
\end{multline}
using stationarity of $\{Y_{(i)}(t), i \geq 0\}$ to obtain the equality. For each $i$, $Y_{i}(t) \leq Y_i(0)e^{a(W_i(t) + \gamma t) + a^2t/2}$, which gives us,
\begin{align}
\Ebd_a^\gamma\left[ \sup_{t \in [0,1]} Y_{i}(t) \right]  \leq c_2\Ebd_a^\gamma\left[Y_{i}(0)\exp\left(a \sup_{t \in [0,1]} W_i(t) \right) \right] = {c_2 \Ebd_a^\gamma[Y_i(0)]\Ebd_a^\gamma\exp\left(a |W_i(1)| \right) < \infty.} \label{eq:exp2bd}
\end{align}
  This proves the desired uniform in $n$ bound.
Applying the bounds \eqref{eq:expbd}, \eqref{eq:exp2bd}, and \eqref{eq:Ibigbd} to \eqref{eq:notYlow}, we obtain that $\lim_{M \to \infty} \lim_{T \to \infty}\Pbd_a^\gamma((B_{T,M}^1)^c) = 0$, as desired.
This completes the proof of Step 1.

\textit{Step 2.} To complete the proof, let $\tilde{X}^{M,*}(t) \doteq \inf_{i \geq J_0(M)} \tilde{X}_i(t)$. We make the following observation:
\begin{equation} \label{eq:keyob}
\text{On the event } \{J_0(M) \geq 1\} \cap A_{T,M}, \quad \tilde{X}^{M,*}(t) = \tilde{X}_{(0)}(t) \text{ for } t \geq T.
\end{equation}
Thus, for $t\ge T$,
\begin{align}
&\Pbd_a^\gamma(\tilde{X}_{(0)}(0) \leq x, \tilde{X}_{(0)}(t) \leq y) \\
& = \Pbd_a^\gamma(\{\tilde{X}_{(0)}(0) \leq x, \tilde{X}^{M,*}(t) \leq y\} \cap \{J_0(M) \geq 1\} \cap A_{T,M}) + \epsilon_1(T,M) \\
& = \Pbd_a^\gamma(\{\tilde{X}_{(0)}(0) \leq x, \tilde{X}^{M,*}(t) \leq y\} \cap \{J_0(M) \geq 1\}) + \epsilon_2(T,M), \label{eq:todecuple}
\end{align}
where, for $j = 1,2$, {$\epsilon_j(T,M) \leq c_3(\Pbd_a^\gamma((A_{T,M})^c) + \Pbd_a^\gamma(J_0(M) = 0))$}, 
and so\\ $\lim_{M \to \infty} \lim_{T \to \infty} \epsilon_j(T,M) = 0$.
We have
\begin{align}
& \Pbd_a^\gamma(\{\tilde{X}_{(0)}(0) \leq x, \tilde{X}^{M,*}(t) \leq y\} \cap \{J_0(M) \geq 1\}) \\
& = \frac{1}{\Gamma(1 + 2\gamma/a)}\Ebd_a^0[Y_{(0)}(0)^{2\gamma/a}\mathbf{1}\{J_0(M) \geq 1\}\mathbf{1}\{\tilde{X}_{(0)}(0) \leq x, \tilde{X}^{M,*}(t) \leq y\}]. 
\end{align}
Let for $-\infty <a \le b<\infty$, $J_0(a,b] \doteq J_0(b) - J_0(a)$ and define
$$\mathcal{G}_1 \doteq \sigma\{J_0(a,b]: -\infty<a \le b \le M\}, \;\;
\mathcal{G}_2 \doteq \sigma\{J_0(a,b]: M<a \le b <\infty\} \vee \sigma\{W_i, i \in \NN_0\}.
$$
Under $\Pbd_a^0$, $\mathcal{G}_1$ and $\mathcal{G}_2$ are independent.
Furthermore, $\mathbf{1}(J_0(M) \ge 1)Y_{(0)}(0)^{2\gamma/a}$ is $\mathcal{G}_1$ measurable and
$\tilde{X}^{M,*}(t)$ is $\mathcal{G}_2$ measurable.

Thus, the last line above is  equal to 
\begin{align}
 \frac{1}{\Gamma(1 + 2\gamma/a)}\Ebd_a^0[Y_{(0)}(0)^{2\gamma/a}\mathbf{1}\{J_0(M) \geq 1\}\mathbf{1}\{\tilde{X}_{(0)}(0) \leq x\}]\Ebd_a^0[\mathbf{1}\{\tilde{X}^{M,*}(t) \leq y\}].
 \end{align}
Next note that
\begin{align*}
&\Ebd_a^0[\mathbf{1}\{\tilde{X}^{M,*}(t) \leq y\}] \\
&= \frac{1}{\Gamma(1 + 2\gamma/a)}\Ebd_a^0[\mathbf{1}\{J_0(M) \geq 1\}Y_{(0)}(0)^{2\gamma/a}] \Ebd_a^0[
\mathbf{1}\{\tilde{X}^{M,*}(t) \leq y\}] + \epsilon_3(M)\\
&=  \frac{1}{\Gamma(1 + 2\gamma/a)}\Ebd_a^0[\mathbf{1}\{J_0(M) \geq 1\}Y_{(0)}(0)^{2\gamma/a}
\mathbf{1}\{\tilde{X}^{M,*}(t) \leq y\}] + \epsilon_3(M),
\end{align*}
where $\lim_{M \to \infty} \epsilon_3(M) = 0$.
Also, in view of \eqref{eq:keyob} and Step 1,
\begin{align*}
&\Ebd_a^0[\mathbf{1}\{J_0(M) \geq 1\}Y_{(0)}(0)^{2\gamma/a}
\mathbf{1}\{\tilde{X}^{M,*}(t) \leq y\}]\\
 &=
 \Ebd_a^0[\mathbf{1}\{J_0(M) \geq 1\}Y_{(0)}(0)^{2\gamma/a}
\mathbf{1}\{\tilde{X}_{(0)}(t) \leq y\}] + \epsilon_4(T,M)\\
&=\Ebd_a^0[Y_{(0)}(0)^{2\gamma/a}
\mathbf{1}\{\tilde{X}_{(0)}(t) \leq y\}] + \epsilon_5(T,M),
\end{align*}
where $\lim_{M \to \infty} \lim_{T \to \infty} \epsilon_j(T,M) = 0$ for $j=4,5$.
Combining these equalities we have
\begin{multline*}
\Pbd_a^\gamma(\tilde{X}_{(0)}(0) \leq x, \tilde{X}_{(0)}(t) \leq y)
= \frac{1}{\Gamma(1 + 2\gamma/a)}\Ebd_a^0[Y_{(0)}(0)^{2\gamma/a}\mathbf{1}\{\tilde{X}_{(0)}(0) \leq x\}]\\
 \cdot \frac{1}{\Gamma(1 + 2\gamma/a)}\Ebd_a^0[Y_{(0)}(0)^{2\gamma/a}\mathbf{1}\{\tilde{X}_{(0)}(t) \leq y\}] + \epsilon_6(T,M) \\
 = \Pbd_a^\gamma(\tilde{X}_{(0)}(0) \leq x )\Pbd_a^\gamma(\tilde{X}_{(0)}(t) \leq y) + \epsilon_6(T,M),
 \end{multline*}
where $\lim_{M \to \infty} \lim_{T \to \infty} \epsilon_6(T,M) = 0$.
Under
$\Pbd^{\gamma}_a$, by stationarity, $\tilde{X}_{(0)}(t)$ has the same distribution as $\tilde{X}_{(0)}(0)$, for all $t \in [0,\infty)$, which  by the discussion in the introduction (see \eqref{eq:inhomogeneous2}), is given by the density function $f_\eta$, defined in \eqref{eq:feta}. 
Thus
$$\Pbd_a^\gamma(\tilde{X}_{(0)}(0) \leq x, \tilde{X}_{(0)}(t) \leq y) = F_\eta(x)F_\eta(y) + \epsilon_6(T,M).$$
The result follows on sending $T\to \infty$ and then $M\to \infty$.
\end{proof}

%
%

\begin{funding}
Research supported in part by  the RTG award (DMS-2134107) from the NSF.  SB was supported in part by the NSF-CAREER award (DMS-2141621).
AB was supported in part by the NSF (DMS-2152577). \ab{We acknowledge the valuable feedback of an anonymous referee which led to significant improvements in the article.}
\end{funding}



\bibliographystyle{imsart-number} 
\bibliography{atlas_ref, inert_ref}       


\end{document}